\documentclass[reqno]{amsart}
\usepackage[utf8]{inputenc}
\usepackage[whole]{bxcjkjatype}
\pdfoutput=1

\usepackage[left=20mm, right=20mm, top=20mm, bottom=15mm]{geometry}

\usepackage{bm,bbm,amsfonts}
\usepackage{amsmath}
\usepackage{amssymb}
\usepackage{amsthm}
\usepackage{graphicx}
\usepackage{hyperref}
\usepackage{physics}
\usepackage{color}
\usepackage[mathscr]{euscript}
\usepackage{tikz}
\usetikzlibrary{calc}
\usetikzlibrary{cd}
\usepackage[root radius=.8mm,edge length= 5mm,mark=o]{dynkin-diagrams}
\usepackage{enumitem}
\usepackage{lscape}
\usepackage{accents}
\usepackage{booktabs}
\usepackage{stmaryrd}

\hypersetup{colorlinks,breaklinks,
            linkcolor = purple,
            citecolor = teal,
            urlcolor = purple,
            }

\newcommand{\rem}[1]{\textbf{\textcolor{blue}{(#1)}}}
\newlength{\dhatheight}

\theoremstyle{definition}
\newtheorem{theorem}{Theorem}[section]
\newtheorem{proposition}[theorem]{Proposition}
\newtheorem{lemma}[theorem]{Lemma}

\newtheorem{definition}[theorem]{Definition}
\newtheorem{conjecture}[theorem]{Conjecture}
\newtheorem{example}[theorem]{Example}

\theoremstyle{remark}
\newtheorem{remark}[theorem]{Remark}

\newcommand{\sA}{\mathsf{A}}
\newcommand{\sY}{\mathsf{Y}}
\newcommand{\sS}{\mathsf{S}}
\newcommand{\sT}{\mathsf{T}}
\newcommand{\si}{\mathsf{i}}
\newcommand{\scA}{\mathscr{A}}
\newcommand{\scS}{\mathscr{S}}
\newcommand{\scY}{\mathscr{Y}}
\newcommand{\fq}{\mathfrak{q}}

\numberwithin{equation}{section}

\title{Higgsing $qq$-character and irreducibility}

\author[Taro Kimura]{Taro Kimura}
\address{Institut de Math\'ematiques de Bourgogne, Universit\'e Bourgogne Europe, CNRS, UMR 5584, France}

\dedicatory{Dedicated to Professor Hiraku Nakajima on his 60th anniversary}

\begin{document}

\begin{abstract}
    We study the $qq$-character of quantum affine and toroidal algebra modules, with a focus on the role of spectral parameters. In particular, we revisit how their specialization affects the irreducibility of these modules.
\end{abstract}

\maketitle

\setcounter{tocdepth}{1}
\tableofcontents

\section{Introduction and summary}

The $qq$-character is a two-parameter deformation of the ordinary character associated with a module constructed on a quiver, which has been introduced in the context of supersymmetric gauge theory through double quantization of Seiberg--Witten geometry by Nekrasov~\cite{Nekrasov:2015wsu}.
The $qq$-character has a geometric realization using Nakajima's quiver variety~\cite{Nakajima:1994nid,Nakajima:1998DM,Nakajima:1999JAMS}, which could be interpreted as a natural generalization of that for the $q$-character of Yangian and quantum affine algebra~\cite{Knight:1995JA,Frenkel:1998ojj}:%
\footnote{%
In fact, the notion of the $qq$-character was already mentioned in \cite{Frenkel:1997lee} in the context of the deformed W-algebras.
Moreover, there exists another two parameter deformation of the character, called the $t$-analog of $q$-character~\cite{Nakajima:2001PC,Nakajima:2004AM}, which also has a natural geometric construction based on the (graded) quiver variety.
The role of the $t$-parameter is different from the second deformation parameter of the $qq$-character, so that the relation between the $qq$-character and the $t$-analog of $q$-character is not yet clear at this moment.
}
Let $\Gamma = (\Gamma_0,\Gamma_1)$ be a quiver, where $\Gamma_0$ is a set of nodes, and $\Gamma_1$ is a set of edges.
We denote by $\mathfrak{g}_\Gamma$ the simple Lie algebra whose Dynkin diagram coincides with a quiver $\Gamma$.
We denote the quiver variety associated with a quiver $\Gamma$ by $\mathfrak{M}_{w,v}^\Gamma$ with the dimension vectors, $v = (v_i)_{i \in \Gamma_0}$ and $w = (w_i)_{i \in \Gamma_0}$, such that the vector spaces attached to the nodes and the corresponding framing spaces are given by $\mathbf{V}_i = \mathbb{C}^{v_i}$ and $\mathbf{W}_i = \mathbb{C}^{w_i}$ for $i\in \Gamma_0$.
We also write $\mathbf{V} = (\mathbf{V}_i)_{i \in \Gamma_0}$, $\mathbf{W} = (\mathbf{W}_i)_{i \in \Gamma_0}$ and denote the automorphism groups by $G_v = \mathrm{GL}(\mathbf{V}) = \prod_{i \in \Gamma_0} \mathrm{GL}(\mathbf{V}_i)$, $G_w = \mathrm{GL}(\mathbf{W}) = \prod_{i \in \Gamma_0} \mathrm{GL}(\mathbf{W}_i)$.
Then, in addition to the quantum deformation parameters $q_1, q_2 \in \mathbb{C}^\times$, the $q$- and $qq$-characters depend on the equivariant parameters of $G_w$ denoted by $x = (x_{i,\alpha})_{i \in \Gamma_0, \alpha = 1,\ldots,w_i}$.
In fact, these parameters are identified with the zeros of the Drinfeld polynomials, called the spectral parameters, %
playing a central role in the construction of the finite-dimensional module of the Yangian and the quantum affine algebra~\cite{Chari:1991CMP,Chari:1994pf}.
It is known that the finite-dimensional module associated with generic spectral parameters is irreducible.
However, when the spectral parameters are in a special relation, 
the corresponding module may be reducible.
The purpose of this paper is to elucidate the role of the spectral parameters from the point of view of the $qq$-character.

\subsection{Finite-dimensional module of quantum affine algebra and $qq$-character}

A finite-dimensional module of the quantum affine algebra $U_q(\widehat{\mathfrak{g}}_\Gamma)$ is in general an $\ell$-highest-weight module, which allows the $\ell$-weight space decomposition.%
\footnote{%
In the following, we call them a highest-weight module and the weight space decomposition for simplicity.}
We denote the highest-weight module parametrized by the dimension vector $w$ (identified with the Dynkin label) and the spectral parameter $x$ by $\mathbf{M}_{w,x} = \bigoplus_{\lambda} \mathbf{M}_\lambda$ with the $\ell$-weight $\lambda$.
Then, the $q$-character of the $U_q(\widehat{\mathfrak{g}}_\Gamma)$-module $\mathbf{M}_{w,x}$ is given by
\begin{align}
 \mathsf{T}^{(\mathsf{q})}_{w,x} = \mathsf{T}^{(\mathsf{q})}[\mathbf{M}_{w,x}] = \sum_\lambda \dim \mathbf{M}_\lambda \mathsf{Y}_\lambda^{(\mathsf{q})} \, ,
\end{align}
where $\mathsf{Y}_\lambda^{(\mathsf{q})}$ is a monomial of the $\Gamma_0$-graded formal commutative variable $(\sY_{i,x})_{i \in \Gamma_0}$, that we call the $\sY$-variable, assigned for each $\ell$-weight $\lambda$.
The monomial for the highest-weight (the highest-weight monomial) is given by $\sY_{w,x} = \prod_{i\in\Gamma_0} \prod_{\alpha=1}^{w_i} \sY_{i,x_{i,\alpha}}$.
The $q$-character takes a value in the Laurent polynomial ring of the $\sY$-variables, $\mathsf{T}^{(\mathsf{q})} \in \mathbb{Z}[\sY_{i,x}^{\pm 1}]_{i\in\Gamma_0}$.

An important result by Frenkel--Reshetikhin and Frenkel--Mukhin on the $q$-character is the following.
\begin{theorem}[\cite{Frenkel:1999ky}; Conjecture in \cite{Frenkel:1998ojj}]\label{thm:q-ch_ker}
 The image of the $q$-character equals the intersection of the kernels of the screening operators.
\end{theorem}
The screening operator mentioned above is seen as a reduced version of the original one by Frenkel and Reshetikhin~\cite{Frenkel:1997lee}, which was introduced to establish the deformed W-algebra as the maximal subalgebra of the Heisenberg algebra that commutes with these operators.
In fact, this original form of the screening operator allows us to consider the vertex operator formalism of the $qq$-character.
Let $(\sY_{i,x})_{i\in\Gamma_0}$ be the operator version of the $\sY$-variable ($\sY$-operator),%
\footnote{%
In the algebraic construction of the $qq$-character based on the vertex operators, $x$ is interpreted as a formal variable. 
In the geometric construction, on the other hand, it is interpreted as the $G_w$-weight as discussed in \S\ref{sec:geom_const}.} 
and define an operator-valued monomial ${:\sY_{w,x}:} = {:\prod_{i\in\Gamma_0} \prod_{\alpha=1}^{w_i} \sY_{i,x_{i,\alpha}}:}$, where we denote by ${:\cdot:}$ the normal ordering symbol, so that there is no ambiguity of the operator ordering for this monomial.
\begin{definition}\label{def:qq-ch_alg}
 Let $\mathbf{M}_{w,x}$ be a standard module of $U_q(\widehat{\mathfrak{g}}_\Gamma)$.
 The operator-valued $qq$-character (operatorial $qq$-character) $\sT_{w,x}^{(q_1,q_2)} = \mathsf{T}^{(q_1,q_2)}[\mathbf{M}_{w,x}]$ is a Laurent polynomial of the $\sY$-operators and their derivatives, that contains the highest-weight monomial $:\sY_{w,x}:$ and the remaining monomials are generated in order that $\sT_{w,x}^{(q_1,q_2)}$ commutes with the screening charges, $(\mathsf{Q}_{i,x})_{i\in\Gamma_0}$.
\end{definition}

See Definition~\ref{def:sc_charge} for the definition of the screening charges.
This definition of the $qq$-character is extended to infinite-dimensional modules of the quantum toroidal algebras. See \S\ref{sec:A_0^} and \S\ref{sec:A_r^}.
The operatorial $qq$-character of the fundamental representation of the quantum affine algebra is nothing but the generating current of the deformed W-algebra in the construction of Frenkel and Reshetikhin~\cite{Frenkel:1997lee}.
By construction, we have $\sT^{(q_1,q_2)}_{w,x} \in \bigcap_{i\in\Gamma_0} \operatorname{Ker}(\operatorname{ad}_{\mathsf{Q}_{i}})$, which is an analog of Theorem~\ref{thm:q-ch_ker} for the $qq$-character.

The $qq$-character has a similar form to the $q$-character,
\begin{align}
 \mathsf{T}^{(q_1,q_2)}_{w,x} = \sum_\lambda Z_\lambda {:\mathsf{Y}_\lambda^{(q_1,q_2)}:} \, , \label{eq:qq-ch_weight_decomposition}
\end{align}
where $Z_\lambda$ is a rational function of $q_1$, $q_2$, and $x$ and $:\sY_\lambda^{(q_1,q_2)}:$ is a monomial of the $\sY$-operators and their derivarives.
Therefore, the $qq$-character takes a value in the differential polynomial ring, $\mathbb{Q}(q_1,q_2,x)\{\sY_{i,x}^{\pm 1}\}_{i\in\Gamma_0}$. 
We emphasize that there is no ambiguity of the $\sY$-operator ordering in the monomials in the expression \eqref{eq:qq-ch_weight_decomposition} since they occur in the normal ordering symbol. 
Hence, we may replace the $\sY$-operators with the formal commutative variables as in the case of the $q$-character, and we don't write the normal ordering symbol explicitly when no confusion can arise.
We provide details on the construction in \S\ref{sec:construction}.
We remark that the same construction is formally applicable to the highest-weight modules of the quantum toroidal algebras although they will be infinite dimensional.
See \S\ref{sec:A_0^} and \S\ref{sec:A_r^}.



\subsection{Irreducibility of $qq$-character}\label{sec:intro_irr}

Let us briefly demonstrate how the specialization of the spectral parameters affects the irreducibility of the $q$- and $qq$-characters.
As in the case of the $q$-character, we introduce the notion of the irreducibility for the $qq$-character.
\begin{definition}\label{def:dominant}
    A dominant monomial is a monomial of the $\sY$-variables that contains only non-negative powers. 
    We define the irreducible $qq$-character as the one that contains a unique dominant monomial, i.e., the highest-weight monomial of the corresponding module.
\end{definition}

The simplest example is the weight one module of the quantum affine algebra $U_q(\widehat{\mathfrak{sl}}_2)$ associated with the $A_1$ quiver $\mathfrak{g}_{A_1} = \mathfrak{sl}_2$.
Denoting $\sY_{1,x} = \sY_x$, we have the $q$- and $qq$-characters as follows:%
\footnote{
Compared with the original notation~\cite{Frenkel:1998ojj}, our parametrization is given by $\mathsf{q}^{\frac{1}{2}} = q_\text{FR}$ for the $q$-character.
}
\begin{align}
    \sT_{1,x}^{(\mathsf{q})} & = \sY_x + \sY_{x \mathsf{q}}^{-1}
    \, , \qquad
    \sT_{1,x}^{(q_1,q_2)} = \sY_x + \sY_{x q_1 q_2}^{-1}
    \, .
    \label{eq:qq-ch_A1_1}
\end{align}
Apparently, there is no difference on the functional form for the $q$- and $qq$-characters under the identification $\mathsf{q} = q_1 q_2$.

The next example is the weight two module.
For generic spectral parameters $x = (x_1,x_2)$, the $qq$-character of weight two is given by
\begin{align}
    \sT_{2,x}^{(q_1, q_2)}
    & = \sY_{x_1} \sY_{x_2} + \scS\left( \frac{x_2}{x_1} \right) \frac{\sY_{x_2}}{\sY_{x_1 q_1 q_2}} + \scS\left( \frac{x_1}{x_2} \right) \frac{\sY_{x_1}}{\sY_{x_2 q_1 q_2}} + \frac{1}{\sY_{x_1 q_1 q_2} \sY_{x_2 q_1 q_2}}
    \, ,
    \label{eq:qq-ch_A1_2}
\end{align}
where we have a rational function, called the $\scS$-function, in the coefficient (see \S\ref{sec:S-func}).
This $qq$-character corresponds to the tensor product module of the weight one modules, which has a unique dominant monomial, therefore it is irreducible by Definition~\ref{def:dominant}.
Setting $x_2 = x_1 q_1 q_2$ ($x_1 = x_2 q_1 q_2$, resp.), the second (the third) term becomes a dominant monomial, $\sY_{x_2}/\sY_{x_1 q_1 q_2} \xrightarrow{x_2 = x_1 q_1 q_2} 1$, which is however singular since it hits the pole of the $\scS$-function.
In other words, we may find another dominant monomial in the residue of the $qq$-character.
The poles of the second and third terms at $x_1 = x_2$ are canceled with each other, and thus the $qq$-character is regular at $x_1=x_2$.
Since there is no further dominant monomial on this locus, the $qq$-character is irreducible for any $(x_1,x_2)$ as long as it is regular, i.e., except on the singular loci.

Moreover, noticing that the $\scS$-function has zeros $\scS(x) = 0$ at $x = q_1, q_2$, we put $x = (x, xq_1)$ and obtain the reduced $qq$-character that consists of three monomials,
\begin{align}
    \sT_{2,x}^{(q_1, q_2)} 
    \ \xrightarrow{x \to (x, xq_1)} \ 
    \widetilde{\sT}_{2,x}^{(q_1, q_2)} =
    \sY_x \sY_{x q_1} + \scS(q_1^{-1}) \frac{\sY_x}{\sY_{x q_1^2 q_2}} + \frac{1}{\sY_{x q_1 q_2} \sY_{x q_1^2 q_2}}
    \, .
\end{align}
We denote such a reduced version of the $qq$-character by $\widetilde{\sT}_{2,x}^{(q_1, q_2)}$, which corresponds to the three-dimensional module of $U_q(\widehat{\mathfrak{sl}}_2)$.%
\footnote{%
By Nekrasov's compactness theorem~\cite{Nekrasov:2016qym}, a reduced $qq$-character is in general regular with respect to the spectral parameter $x$.
}
Noticing $\lim_{q_1 \to 1} \scS(q_1^{-1}) = 2$ and $\lim_{q_2 \to 1} \scS(q_1^{-1}) = 1$, the reduced $qq$-character $\widetilde{\sT}_{2,x}^{(q_1, q_2)}$ is further reduced to the $q$-characters as follows:
\begin{subequations}
\begin{align}
    \lim_{q_1 \to 1} \widetilde{\sT}_{2,x}^{(q_1, q_2)} & = \sY_x^2 + 2 \frac{\sY_x}{\sY_{x q_2}} + \sY_{x q_2}^{-2} = \left( \sY_x + \sY_{x q_2}^{-1} \right)^2 = \left( \sT_{1,x}^{(q_2)} \right)^2
    \, , \\
    \lim_{q_2 \to 1} \widetilde{\sT}_{2,x}^{(q_1, q_2)} & = \sY_x \sY_{x q_1} + \frac{\sY_{x}}{\sY_{x q_1^2}} + \frac{1}{\sY_{x q_1} \sY_{x q_1^2}} = \sT_{2,x}^{(q_1)}
    \, .
\end{align}
\end{subequations}
See~\cite[\S4.1]{Frenkel:1998ojj} for the $q$-character formula for $U_q(\widehat{\mathfrak{sl}}_2)$ .
This implies that we have the irreducible three-dimensional module in the limit $q_2 \to 1$, while we have the tensor product module of the same spectral parameters 
in the other limit $q_1 \to 1$.
Such a reduction to the irreducible module is specific to the $qq$-character.
For example, the product of the $q$-characters with the spectral parameters $(x,x \mathsf{q})$ is given by
\begin{align}
    \sT_{1,x}^{(\mathsf{q})} \sT_{1,x\mathsf{q}}^{(\mathsf{q})} = \sY_x \sY_{x \mathsf{q}} + \frac{\sY_x}{\sY_{x \mathsf{q}^2}} + \frac{1}{\sY_{x \mathsf{q}} \sY_{x \mathsf{q}^2}} + 1
    = \sT_{2,x}^{(\mathsf{q})} + \sT_{0,x}^{(\mathsf{q})}
    \, ,
\end{align}
which is reducible since there are two dominant monomials, $\sY_x \sY_{x \mathsf{q}}$ and $1$, and we write the corresponding $q$-characters by $\sT_{2,x}^{(\mathsf{q})}$ and $\sT_{0,x}^{(\mathsf{q})}$.%
\footnote{%
This relation among the $q$-characters $(\sT_{w,x}^{(\mathsf{q})})_{w = 0,1,2}$ is interpreted as the simplest example of the T-system. See, e.g.,~\cite{Kuniba:2010ir}.
}
See \S\ref{sec:A1} for more detailed analysis on $A_1$ quiver.

We remark that since the $qq$-character is symmetric under exchanging $q_1 \leftrightarrow q_2$ in this example, we may also consider another specialization $x = (x,xq_2)$ to obtain the irreducible $qq$-character.
Such a symmetry between $q_1$ and $q_2$ is expected to hold for the Kirillov--Reshetikhin (KR) module of simply-laced cases~\cite{Kirillov:1987JSM}, while it would not be guaranteed in other cases.
See \S\ref{sec:A2_w11} for $A_2$ quiver (weight $w = (1,1)$) and \S\ref{sec:BC2} for $B_2/C_2$ quiver examples.

\subsection{Higgs mechanism}

Specialization of the spectral parameters demonstrated above has a natural interpretation as the Higgs mechanism.
Geometrically realizing the $qq$-character in the eight-dimensional setup, called the gauge origami, the spectral parameters are interpreted as the Coulomb moduli of the dual gauge theory defined in the transverse surface~\cite{Nekrasov:2016ydq}.
From this point of view, the gauge symmetry is originally described by $\mathrm{U}(\mathbf{W}) = \prod_{i \in \Gamma_0} \mathrm{U}(w_i)$, which is the maximal compact subgroup of $G_w = \mathrm{GL}(\mathbf{W})$.
Specialization of the spectral parameters implies freezing all the Coulomb moduli except for the center of mass factor.
Therefore, we have the symmetry breaking, $\mathrm{U}(\mathbf{W}) \to \mathrm{U}(1)$, freezing the non-Abelian part of the gauge group (W- and Z-bosons).
The $\mathrm{U}(1)$ factors of each $\mathrm{U}(w_i)$ are fixed by the bifundamental mass parameters except for the overall center of mass factor in this context.
It has been established that the $qq$-character has no singularity with respect to this center of mass factor, a.k.a., the compactness theorem~\cite{Nekrasov:2016qym}.

Such an interpretation of the parameter specialization as Higgsing has been discussed in various contexts.
It has been known that one can move on to the root of Higgs branch of the moduli space of supersymmetric vacua by tuning the Coulomb moduli~\cite{Dorey:1998yh,Dorey:1999zk}.
In the context of topological string theory, the idea of Higgsing is used to describe the geometric transition~\cite{Gopakumar:1998ki}.
For the primary example of the transition between resolved and deformed conifolds, one may discuss open string degrees of freedom on the Lagrangian submanifold on the deformed conifold side, which is given as the cotangent bundle $T^\vee\mathbb{S}^3$ and its Lagrangian $\mathbb{S}^3$.
From this point of view, the Higgsing condition on the K\"ahler parameter is interpreted as the flux quantization condition.
In this direction, we may apply a similar interpretation to the current situation.
Geometrically, the irreducible part corresponds to a subvariety of the whole quiver variety~\cite{Nakajima:1999JAMS}.
Hence, the Higgsed $qq$-character, which describes the irreducible module, is expected to correspond to this irreducible subvariety of the quiver variety.
This is a geometric representation theoretical interpretation of the Higgs mechanism.

\subsection{Summary and organization}

In order to present the statement, we prepare the following definitions.
In this paper, we consider the vertex operator formalism of the $qq$-character (Definition~\ref{def:qq-ch_alg}).
See also a geometric definition in Definition~\ref{def:qq-ch}.
\begin{definition}\label{def:non-highest} 
    A non-highest dominant monomial is a dominant monomial in a $qq$-character that is not the highest-weight monomial.
    A non-highest $qq$-character is the $qq$-character 
 that is generated from a non-highest dominant monomial. 
\end{definition}

We consider the support of the $qq$-character denoted by $\operatorname{supp}_{\sY}(\mathsf{T}^{(q_1,q_2)})$, which is the set of monomials of the $\sY$-variables in $\mathsf{T}^{(q_1,q_2)}$ with non-zero coefficient.
The number of monomials in $\mathsf{T}^{(q_1,q_2)}$ is then given by $|\operatorname{supp}_{\sY}(\mathsf{T}^{(q_1,q_2)})|$.

\begin{definition}
    For a given weight dimension vector $w$, a maximally reduced $qq$-character $\widetilde{\mathsf{T}}_{w,x}^{(q_1,q_2)}$ is a $qq$-character with the minimal-support specialization of the $qq$-character denoted by $\mathsf{T}_{w,x}^{(q_1,q_2)}$ with generic spectral parameters $x = (x_{i,\alpha})_{i \in \Gamma_0,\alpha=1,\ldots,w_i}$. $\widetilde{\mathsf{T}}_{w,x}^{(q_1,q_2)}$ depends on a unique spectral parameter.
\end{definition}
For a given $qq$-character $\mathsf{T}_{w,x}^{(q_1,q_2)}$, specialization of the spectral parameters to obtain $\widetilde{\mathsf{T}}_{w,x}^{(q_1,q_2)}$ is expected to be unique for the KR modules, while it is not the case for general modules.
See Conjecture~\ref{conj:irreducibility} and examples of $A_2$ quiver (\S\ref{sec:A2_w11}) and $B_2/C_2$ quiver (\S\ref{sec:BC2_11}).

Then, we have the following Conjectures regarding the irreducibility of the $qq$-character.
\begin{conjecture}\label{conj:irr}
    The $qq$-character of the finite-dimensional module of the quantum affine algebra is irreducible for any spectral parameters except on the singular loci, whereas the residue of each pole contains a non-highest $qq$-character.
\end{conjecture}
\begin{remark}\label{rem:singularity}
    In this paper, the singularity of the $qq$-character means that of the coefficient of each monomial, i.e., the $\scS$-functions appearing in the $qq$-character.
    We assume that each $\sY$-variable does not cause any singularities.    
\end{remark}

We examine this Conjecture with various examples in this paper.
Moreover, we have another specific Conjecture regarding the KR module.
\begin{conjecture}[Theorems~\ref{thm:A_1_I},~\ref{thm:A_1_II} for $A_1$ quiver]\label{conj:irreducibility}
Let $\Gamma$ be a finite-type quiver and fix $i \in \Gamma_0$. 
Considering the weight dimension vector having a single non-zero entry, $w_i > 0$ and $w_j = 0$ for $j \neq i$, the following holds:
\begin{enumerate}
    \item\label{conj1} The maximally reduced $qq$-character is given by specialization of spectral parameters that obey the $q_1$-segment condition, $x = (x, xq_1^{d_i}, xq_1^{2d_i},\ldots,xq_1^{d_i(w_i-1)})$. 
    \item\label{conj2} In the limit $q_2 \to 1$, the maximally reduced $qq$-character is further reduced to the $q_1$-character of the corresponding simple module.
    \item\label{conj3} In the limit $q_1 \to 1$, the maximally reduced $qq$-character is factorized into the product of the $q_2$-characters of the $i$-th fundamental modules of the same spectral parameters.
\end{enumerate}
\end{conjecture}

The remaining part of this paper is organized as follows:
In \S\ref{sec:construction}, as a preliminary, we explain the vertex operator formalism of the $qq$-character for the standard module, which also has a geometric realization in terms of the quiver variety.
In \S\ref{sec:A1}, we consider $A_1$ quiver as a primary example and provide a proof of Conjecture~\ref{conj:irreducibility} in this case.
In \S\ref{sec:A2} and \S\ref{sec:BC2}, we consider $A_2$ and $B_2/C_2$ quivers to show several examples, which support Conjecture~\ref{conj:irreducibility} for the KR modules and Conjecture~\ref{conj:irr} for other cases.
In \S\ref{sec:A_0^} and \S\ref{sec:A_r^}, we consider affine quivers $\widehat{A}_0$ and $\widehat{A}_r$ and 
discuss the tensor product of the Fock modules of the corresponding quantum toroidal algebras.

\subsubsection*{Acknowledgments}

I would like to thank Ryo Fujita and Vasily Pestun for useful discussions, and especially Hiraku Nakajima for patient explanations about various aspects of the geometric representation theory and encouragements.
Nakajima-sensei, kanreki omedetou gozaimasu!
This work was partly supported by CNRS through the MITI interdisciplinary programs, ``Investissements d'Avenir'' program, Project ISITE-BFC (No.~ANR-15-IDEX-0003), EIPHI Graduate School (No.~ANR-17-EURE-0002), and Bourgogne-Franche-Comté region.
I'm also grateful to the anonymous referee for valuable comments and suggestions.

\section{Construction of $qq$-character}\label{sec:construction}

In this Section, we explain the vertex operator formalism of the $qq$-character for the standard module, i.e., the module constructed as the tensor product of the fundamental modules,%
\footnote{%
Namely, the module that allows a geometric construction based on the quiver variety as discussed in \S\ref{sec:geom_const}.} 
and an algorithm to calculate it based on the iWeyl reflection~\cite{Kimura:2015rgi}.

\subsection{Quiver}

Let $\Gamma = (\Gamma_0, \Gamma_1)$ be a quiver with the set of nodes $\Gamma_0$ and the set of edges $\Gamma_1$.
We call $\operatorname{rk} \Gamma = |\Gamma_0|$ the rank of quiver.
In this paper, we consider the double quiver, so that there exist the dual edge $e^\vee$ for any $e \in \Gamma_1$.
We define a decorated quiver $\Gamma_d = (\Gamma, d)$ with the set of positive integers $d = (d_i)_{i \in \Gamma_0} \in \mathbb{Z}_{\ge 1}^{\operatorname{rk} \Gamma}$ assigned to each node, which would be identified with the relative root length of the corresponding root system.
We denote $d_{ij} = \operatorname{gcd}(d_i,d_j)$.

\begin{definition}\label{def:Cartan_matrix}
Let $q_1, q_2 \in \mathbb{C}^\times$ and $\nu_e \in \mathbb{C}^\times$ for $e \in \Gamma_1$.
We define a deformation of the Cartan matrix associated with the decorated quiver $\Gamma_d$,
\begin{align}
    c_{ji} = (1 + q_1^{d_i} q_2) \delta_{ij} - \sum_{e:i \to j} \sum_{r=0}^{d_i/d_{ij}-1} \nu_e q_1^{rd_{ij}} - \sum_{e: j \to i} \sum_{r=0}^{d_i/d_{ij}-1} \nu_e^{-1} q_1^{(r+1)d_{ij}} q_2
    \qquad 
    (i,j \in \Gamma_0)
    \, .
    \label{eq:Cartan_matrix}
\end{align}
We also define the quiver Cartan matrix by $c^{[0]} = ({c}_{ji}^{[0]})_{i,j\in\Gamma_0}$ as the classical limit of the deformed Cartan matrix, ${c}_{ji}^{[0]} = \lim_{q_1, q_2, \nu_e \to 1} c_{ji}$.
\end{definition}
We may trivialize the parameters $(\nu_e)_{e \in \Gamma_1}$ for all acyclic quivers (see, e.g., \cite{Nekrasov:2013xda}). 
Hence, we simply put $\nu_e = 1$ except in \S\ref{sec:A_0^} and \S\ref{sec:A_r^} where we discuss cyclic quivers.

\begin{remark}
In this paper, we apply a different convention for the deformation parameters compared with our previous notation~\cite{Kimura:2015rgi,Kimura:2016dys,Kimura:2017hez}, i.e., $(q_1, q_2, \{\nu_e\}_{e\in\Gamma_1}) \leftrightarrow (q_1^{-1},q_2^{-1},\{\nu_e^{-1}\}_{e\in\Gamma_1})$.
\end{remark}

\subsection{$\scS$-function}\label{sec:S-func}

We define a rational function that we call the $\scS$-function,
\begin{align}
    \scS(x)
    = \frac{(1 - x / q_1)(1 - x / q_2)}{(1 - x)(1 - x / q_1 q_2)}
    \, ,
    \label{eq:S-fn}
\end{align}
which obeys the inversion relation for $x \neq 1, q_1 q_2$, 
\begin{align}
    \scS(x) & = \scS(q_1 q_2/x)
    \, .
    \label{eq:S_reflection}
\end{align}
We also define the higher degree $\scS$-function,
\begin{align}
    \scS_r(x) = \frac{(1 - x / q_1^r)(1 - x / q_2)}{(1 - x)(1 - x / q_1^r q_2)} = \prod_{s=0}^{r-1} \scS(x q_1^{-s}) \, ,
\end{align}
with the inversion relation
\begin{align}
    \scS_r(x) = \scS_r(q_1^r q_2 / x) 
    \, , \qquad 
    x \neq 1, q_1^r q_2 \, .
    \label{eq:S_reflection2}
\end{align}
For generic $x$, the $\scS$-function becomes trivial in the classical limit,
\begin{align}
    \scS_r(x) \ \xrightarrow{q_1, q_2 \to 1} \ 1 \, .
\end{align}

More precisely, the inversion relations \eqref{eq:S_reflection} and \eqref{eq:S_reflection2} are described as follows~\cite{Kimura:2020jxl}. 
\begin{lemma}\label{lem:S_inversion}
    Let $\delta(x) = \sum_{n \in \mathbb{Z}} x^n$ be the delta function.
    The difference between $\scS_r(x)$ and $\scS_r(q_1^r q_2 / x)$ is given by
    \begin{align}
        \scS_r(x) - \scS_r(q_1^r q_2 / x) = \frac{(1-q_1^{-r})(1-q_2^{-1})}{1-q_1^{-r}q_2^{-1}} \left( \delta(x) - \delta(q_1^{r}q_2/x) \right) \, ,
    \end{align}
    where $\scS_r(x)$ and $\scS_r(q_1^r q_2 / x)$ are interpreted as formal series of $x$ and $x^{-1}$, respectively.
\end{lemma}
\begin{proof}
Consider the case $r = 1$ for simplicity.    
Then, we may rewrite the $\scS$-function as follows,
\begin{subequations}
\begin{align}
    \scS(x) & = \frac{(1 - x / q_1)(1 - x / q_2)}{(1 - 1 / q_1 q_2) x} \left( \frac{1}{1 - x} - \frac{1}{1 - x/q_1 q_2} \right) \, , \\
    \scS(q_1 q_2 / x) & = \frac{(1 - q_1/x)(1 - q_2/x)}{(1 - 1 / q_1 q_2) q_1 q_2 / x} \left( \frac{1}{1 - q_1 q_2 / x} - \frac{1}{1 - 1/x} \right) \, ,
\end{align}    
\end{subequations}    
from which we compute the difference,
\begin{align}
    \scS(x) - \scS(q_1 q_2 / x)
    & = \frac{(1 - x / q_1)(1 - x / q_2)}{(1 - 1 / q_1 q_2) x} \left( \frac{1}{1 - x} + \frac{1}{1 - x^{-1}} - \frac{1}{1 - x/q_1 q_2} - \frac{1}{1 - q_1 q_2 / x} \right) \nonumber \\
    & = \frac{(1 - x / q_1)(1 - x / q_2)}{(1 - 1 / q_1 q_2) x} \left( \delta(x) - \delta(q_1 q_2/x) \right) \, .
\end{align}
Recalling the property of the delta function, $f(x) \delta(x) = f(1) \delta(x)$ for any formal series $f(x) \in \mathbb{C}\llbracket x^{\pm 1} \rrbracket$, we conclude the proof.
The same proof is applied for the case $r > 1$.
\end{proof}

\subsection{Vertex operators}

Let us introduce the vertex operators associated with the quiver, which are the building blocks of the $qq$-characters. 
The vertex operators considered in this paper have the free field realization, so that they are elements of $\operatorname{End}(\mathcal{F})\llbracket x^{\pm 1} \rrbracket$, the space of operator-valued series acting on the Fock space denoted by $\mathcal{F}$.
We apply the following notations:
\begin{itemize}
    \item 
    For the product of any vertex operators $\mathsf{V}_x$ and $\mathsf{V}'_{x'}$, we impose the radial ordering $|x| > |x'|$ for $\mathsf{V}_x \mathsf{V}'_{x'}$ and $|x'| > |x|$ for $\mathsf{V}'_{x'} \mathsf{V}_x$.
    
    \item 
    For any vertex operator $\mathsf{V}_x$, we write $\displaystyle \mathsf{V}_{i,x;j,k} = \mathsf{V}_{i,x q_1^j q_2^k}$.
    
    \item 
    We denote the normal ordering symbol by ${: \cdot :}$. 
    We can freely change the ordering of the vertex operator product inside this symbol.
\end{itemize}

\begin{definition}\label{def:YA_op}
We define the $\Gamma_0$-graded vertex operators $(\sY_{i,x})_{i \in \Gamma_0}$ and $(\sA_{i,x})_{i \in \Gamma_0}$ obeying the following relation,
\begin{align}
    \sA_{i,x;d_i,1} 
    & = {:\sY_{i,x} \sY_{i,x;d_i,1} \prod_{e:i \to j} \prod_{r=0}^{d_i/d_{ij}-1} \sY_{j,\nu_e x;rd_{ij},0}^{-1} \prod_{e:j \to i} \prod_{r=0}^{d_i/d_{ij}-1} \sY_{j,\nu_e^{-1}x;(r+1)d_{ij},1}^{-1} :}
    \, , \label{eq:A_def}
\end{align}
and also the operator product relations,
\begin{align}
    \sY_{i,x} \sA_{j,x'} = \scS_{d_i} \left( \frac{x'}{x} \right)^{-\delta_{ij}} {: \sY_{i,x} \sA_{j,x'} :} \, , \qquad 
    \sA_{j,x'} \sY_{i,x} = \scS_{d_i} \left( q_1^{d_i} q_2 \frac{x}{x'} \right)^{-\delta_{ij}} {: \sY_{i,x} \sA_{j,x'} :} \, . \label{eq:YA_prod}
\end{align}
\end{definition}
In this paper, we use only the operator product relations, and an explicit free field realization is not necessary in the following discussion. 
See~\cite{Kimura:2015rgi,Kimura:2016dys,Kimura:2017hez} for the construction of these vertex operators based on the free field realization.

\begin{proposition}\label{prop:YA_comm}
    We have the following commutation relation, 
    \begin{subequations}
    \begin{align}
        \left[ \sY_{i,x} , \sA_{i,x'}^{-1} \right] & = \frac{(1-q_1^{-d_i})(1-q_2^{-1})}{1-q_1^{-r}q_2^{-1}} \left( \delta\left( \frac{x'}{x} \right) - \delta\left( q_1^{d_i} q_2 \frac{x}{x'} \right) \right) : \sY_{i,x} \sA_{i,x'}^{-1}:  \, , \\ \qquad 
        \left[ \sY_{i,x} , \sA_{j,x'}^{-1} \right] & = 0 \qquad (i \neq j) \, .
    \end{align}
    \end{subequations}
\end{proposition}
\begin{proof}
This directly follows from Lemma~\ref{lem:S_inversion}.
\end{proof}

We introduce the screening currents $(\mathsf{S}_{i,x})_{i\in \Gamma_0}$ (see~\cite{Kimura:2015rgi,Kimura:2016dys,Kimura:2017hez} for the free field realization), which obeys
\begin{align}
    \sA_{i,x} = q_1^{-d_i} {: \frac{\sS_{i,x}}{\sS_{i,q_2^{-1} x}} :} \, , \label{eq:S_def}
\end{align}
with the following operator product relation,
\begin{align}
    \sY_{i,x} \sS_{j,x'} = \left( \frac{1 - x'/x}{1 - q_1^{-d_i} x'/x} \right)^{\delta_{ij}} :\sY_{i,x} \sS_{j,x'}: \, , \qquad 
    \sS_{j,x'} \sY_{i,x} = \left( q_1^{d_i} \frac{1 - x/x'}{1 - q_1^{d_i} x/x'} \right)^{\delta_{ij}} :\sY_{i,x} \sS_{j,x'}: \, . \label{eq:YS_OPE}
\end{align}
This is consistent with the relations \eqref{eq:YA_prod} since the factor $q_1^{-d_i}$ in \eqref{eq:S_def} does not affect the operator product factor.

\begin{proposition}\label{prop:YS_comm}
We have the following commutation relation,
\begin{subequations}
    \begin{align}
        \left[ \sY_{i,x} , \sS_{i,x'} \right] & = (1 - q_1^{d_i}) \delta\left(q_1^{-{d_i}}\frac{x'}{x}\right) {:\sY_{i,x} \sS_{i,x'}:} \, , \\
        \left[ \sY_{i,x} , \sS_{j,x'} \right] & = 0 \qquad (i \neq j) \, .
    \end{align}
\end{subequations}
\end{proposition}

These vertex operators obey non-trivial commutation relations mentioned above. 
Meanwhile, they become all commutative in the limits $q_1 \to 1$ and $q_2 \to 1$.
It is clear for $\sY$ and $\sA$ from Proposition~\ref{prop:YA_comm}.
By Proposition~\ref{prop:YS_comm}, $\sY$ and $\sS$ commute with each other in the limit $q_1 \to 1$.
In the limit $q_2 \to 1$, on the other hand, the screening current itself is not well-defined from definition~\eqref{eq:S_def}.
Hence, the $q$-character obtained from the $qq$-character in these limits consists of monomials of the commutative $\sY$-variables.

\begin{definition}\label{def:sc_charge} 
We define the screening charge by the formal series of the screening current,
\begin{align}
    \mathsf{Q}_{i,x} = \sum_{k \in \mathbb{Z}} \sS_{i,x q_2^k} \, , \quad i \in \Gamma_0 \, . \label{eq:sc_charge}
\end{align}
\end{definition}

This definition of the screening charge is a $q$-analog of that for the ordinary Virasoro and W-algebras given by the contour integral of the screening currents.
In fact, we may interpret the screening charge~\eqref{eq:sc_charge} as the $qq$-character of the vector module of the quantum toroidal $\mathfrak{gl}_1$, which is not a highest-weight type~\cite{Kimura:2023bxy}.
We may also use the Jackson integral or the contour integral to define the screening charge for the deformed case with a specific integration contour~\cite{Aganagic:2017smx}.

The operatorial $qq$-character is now constructed based on the $\sY$-operators and the screening charges (Definition~\ref{def:qq-ch_alg}).
Before discussing a concrete construction of the $qq$-character in \S\ref{sec:iWeyl_ref}, we address a remark as follows.
As established by Frenkel and Reshetikhin~\cite{Frenkel:1998ojj}, the $q$-character is a ring homomorphism from the Grothendieck ring of the category of finite-dimensional representations of $U_q(\widehat{\mathfrak{g}})$ denoted by $\operatorname{Rep}U_q(\widehat{\mathfrak{g}})$ to the Laurent polynomial ring $\mathbb{Z}[\mathsf{Y}_{i,x}^{\pm 1}]_{i\in\Gamma_0}$.
We emphasize that it is not the case for the $qq$-character. 
For example, the $qq$-character of the weight one and two modules of $U_q(\widehat{\mathfrak{sl}}_2)$ is given in \eqref{eq:qq-ch_A1_1} and \eqref{eq:qq-ch_A1_2}, respectively. 
In this case, we have the following relations,
\begin{align}
    \sT_{1,x_1}^{(q_1,q_2)} \sT_{1,x_2}^{(q_1,q_2)} = \mathsf{f}\left(\frac{x_2}{x_1}\right)^{-1} \sT_{2,x}^{(q_1, q_2)}
    \, , \qquad 
    \sT_{1,x_2}^{(q_1,q_2)} \sT_{1,x_1}^{(q_1,q_2)} = \mathsf{f}\left(\frac{x_1}{x_2}\right)^{-1} \sT_{2,x}^{(q_1, q_2)} \, ,
\end{align}
where we define
\begin{align}
    \mathsf{f}(z) = \exp \left( \sum_{n = 1}^\infty \frac{(1-q_1^{-n})(1-q_2^{-n})}{n(1 + q_1^{-n} q_2^{-n})} z^n \right) \, .
\end{align}
In fact, we have%
\begin{align}
    \mathsf{f}\left(\frac{x_2}{x_1}\right) \sT_{1,x_1}^{(q_1,q_2)} \sT_{1,x_2}^{(q_1,q_2)} - \mathsf{f}\left(\frac{x_1}{x_2}\right) \sT_{1,x_2}^{(q_1,q_2)} \sT_{1,x_1}^{(q_1,q_2)}
    & = \frac{(1-q_1^{-1})(1-q_2^{-1})}{(1-q_1^{-1}q_2^{-1})} \left( \delta\left(q_1q_2 \frac{x_2}{x_1}\right) - \delta\left(q_1q_2 \frac{x_1}{x_2}\right) \right) \, . \label{eq:quadratic_rel_A1}
\end{align}
This is the defining relation of the $q$-deformed Virasoro algebra~\cite{Shiraishi:1995rp}.
The $\mathsf{f}$-function comes from the operator product of the $\sY$-operator.
The delta functions on the RHS originate from the poles of the $\scS$-function.
Since the $qq$-character commutes with the screening charges, its product also does unless the $\mathsf{f}$-function is singular.
Hence, the product of $qq$-characters can be also a $qq$-character of higher weight modulo $\mathsf{f}$.
Although, we can construct the higher weight $qq$-characters starting from the fundamental $qq$-character in this way, its computation is not so efficient since we have to manage all the operator products of $\sY$ and $\sA$ operators.
In the following, we explain an alternative method to construct the $qq$-character.


\begin{remark}\label{rmk:q-ch_commute}
 In the limits, $q_1 \to 1$ and $q_2 \to 1$, we have $\mathsf{f}(z) = 1$, and thr RHS of \eqref{eq:quadratic_rel_A1} vanishes, meaning that the $q$-character is commutative.
\end{remark}

\subsection{iWeyl reflection}\label{sec:iWeyl_ref}

Since the $\sY$-operator itself does not commute with the screening charge, we should consider an additional contribution to construct the $qq$-character.
We explain how to organize this process systematically and establish an algorithm to compute the $qq$-character.

\begin{lemma}
Let $\Gamma_d$ be a decorated quiver without a loop edge.
For the vertex operators associated with $\Gamma_d$ defined above, we have
\begin{align}
    \left[ \sY_{i,x} + :\sY_{i,x} \sA_{i,x;d_i,1}^{-1} : , \sS_{i,x'} \right] = (1 - q_1^{d_i}) \left( \delta\left(q_1^{-d_i}\frac{x'}{x}\right) :\sY_{i,x} \sS_{i,x'}: - \delta\left(q_1^{-d_i}q_2^{-1}\frac{x'}{x}\right) :\sY_{i,x} \sS_{i,x' q_2^{-1}}: \right) \, , \label{eq:iWeyl1}
\end{align}    
and then
\begin{align}
    \left[ \sY_{i,x} + :\sY_{i,x} \sA_{i,x;d_i,1}^{-1} : , \mathsf{Q}_{i,x'} \right] = 0 \, . \label{eq:iWeyl2}
\end{align} 
\end{lemma}
\begin{proof}
By Definition~\ref{def:YA_op}, we have
\begin{align}
    {:\sY_{i,x} \sA_{i,x;d_i,1}^{-1}:} = {: \sY_{i,x;d_i,1}^{-1} \prod_{e:i \to j} \prod_{r=0}^{d_i/d_{ij}-1} \sY_{j,\nu_e x;rd_{ij},0} \prod_{e:j \to i} \prod_{r=0}^{d_i/d_{ij}-1} \sY_{j,\nu_e^{-1}x;(r+1)d_{ij},1} :} \, ,
\end{align}
where $\sY_{i,x;d_i,1}^{-1}$ is the unique $\sY$-operator associated with the node $i \in \Gamma_0$ in this monomial under the no-loop assumption of the quiver. 
Since $\sS_i$ commutes with $\sY_j$ for $i \neq j$, we may focus only on this factor.
We have
\begin{align}
    \sY_{i,x;d_i,1}^{-1} \sS_{i,x'} = \frac{1 - x'/x q_1^{2d_i} q_2}{1 - x'/x q_1^{d_i} q_2} :\sY_{i,x;d_i,1}^{-1} \sS_{i,x'}: \, , \qquad
    \sS_{i,x'} \sY_{i,x;d_i,1}^{-1} = q_1^{-d_i} \frac{1 - q_1^{2d_i} q_2 x/x'}{1 - q_1^{d_i} q_2 x/x'} :\sY_{i,x;d_i,1}^{-1} \sS_{i,x'}: \, ,
\end{align}
which yields
\begin{align}
    \left[ \sY_{i,x;d_i,1}^{-1}, \sS_{i,x'} \right] = (1 - q_1^{-d_i}) \delta\left(q_1^{-d_i}q_2^{-1} \frac{x'}{x}\right) :\sY_{i,x;d_i,1}^{-1} \sS_{i,x'}: \, ,
\end{align}
and hence
\begin{align}
    \left[ :\sY_{i,x} \sA_{i,x;d_i,1}^{-1} : , \sS_{i,x'} \right] & = (1 - q_1^{-d_i}) \delta\left(q_1^{-d_i} q_2^{-1} \frac{x'}{x}\right) :\sY_{i,x} \sA_{i,x;d_i,1}^{-1} \sS_{i,x'}: \, . 
\end{align}
Together with \eqref{eq:S_def}, $\sA_{i,x;d_i,1} = q_1^{-d_i} : \sS_{i,x;d_i,1} \sS_{i,x q_1^{d_i}}^{-1}:$, we prove \eqref{eq:iWeyl1}.
Since the RHS of \eqref{eq:iWeyl1} is written as a total $q_2$-difference with respect to the variable $x'$, we also conclude \eqref{eq:iWeyl2}.
\end{proof}

We call this operation that converts $\sY_i$ to $\sY_i \sA_i^{-1}$ the iWeyl reflection,%
\footnote{%
In the context of the $q$-character, the operator involving this operation is called the screening operator in \cite{Frenkel:1998ojj}. In our formalism, we will see that the integral operator \eqref{eq:A_screening_ch} generates this operation instead of $\sS$ and $\mathsf{Q}$.
}
\begin{align}
    \text{iWeyl} \ : 
    \sY_{i,x} \ \longmapsto \ {:\sY_{i,x} \sA_{i,x;d_i,1}^{-1}:} 
    = {: \sY_{i,x;d_i,1}^{-1} \prod_{e:i \to j} \prod_{r=0}^{d_i/d_{ij}-1} \sY_{j, x;rd_{ij},0} \prod_{e:j \to i} \prod_{r=0}^{d_i/d_{ij}-1} \sY_{j,x;(r+1)d_{ij},1} :}
    \, . \label{eq:iWeyl_def}
\end{align}
Although this process assures the commutativity with the screening charge $\mathsf{Q}_i$ associated with the node $i \in \Gamma_0$ (local commutativity), the iWeyl reflection in general generates the $\sY$-operator of other nodes $j$.
Hence, we should recursively apply the process of the iWeyl reflection to construct the full $qq$-character, which commutes with all the screening charges (global commutativity).

In general, we have the following.
\begin{proposition}\label{prop:multi_iWeyl}
    Let $\Gamma_d$ be a decorated quiver without a loop edge.
    We assume that $\{x_1,\ldots,x_k\}$ are all distinct for $k \in \mathbb{Z}_{>0}$.
    Then, we have
    \begin{align}
        \left[ \sum_{I \sqcup J = \{1,\ldots,k\}} \prod_{\alpha \in I, \beta \in J} \scS_{d_i}\left(\frac{x_\alpha}{x_\beta}\right) :\prod_{\alpha=1}^k \sY_{i,x_\alpha} \prod_{\beta \in J} \sA_{i,x_\beta;d_i,1}^{-1}: , \mathsf{Q}_{i,x'} \right] = 0 \, .
    \end{align}
\end{proposition}

The following Lemma is essential to prove this Proposition.
\begin{lemma}\label{lemma:YAS_comm}
We fix $I$ and $J$, such that $I \sqcup J = \{1,\ldots,k\}$.
Assuming that $\{x_1,\ldots,x_k\}$ are all distinct, we have
\begin{align}
    & \left[ :\prod_{\alpha=1}^k \sY_{i,x_\alpha} \prod_{\beta \in J} \sA_{i,x_\beta;d_i,1}^{-1}: , \sS_{i,x'} \right] \nonumber \\
    & = (1 - q_1^{d_i}) \sum_{\alpha \in I} \delta\left(q_1^{-d_i}\frac{x'}{x_\alpha}\right) \prod_{\substack{\alpha'\in I \\ \alpha' \neq \alpha}} \frac{1 - q_1^{d_i} x_\alpha/x_{\alpha'}}{1 - x_\alpha/x_{\alpha'}} \prod_{\beta \in J} \frac{1 - x_\alpha/x_\beta q_1^{d_i} q_2}{1 - x_\alpha/x_\beta q_2} :\prod_{\alpha=1}^k \sY_{i,x_\alpha} \prod_{\beta \in J} \sA_{i,x_\beta;d_i,1}^{-1} \sS_{i,x'}: \nonumber \\
    & \quad - (1 - q_1^{d_i}) \sum_{\beta \in J} \delta\left(q_1^{-d_i}q_2^{-1} \frac{x'}{x_\beta}\right) \prod_{\alpha \in I} \frac{1 - q_1^{d_i}q_2 x_\beta / x_\alpha}{1 - q_2 x_\beta / x_\alpha} \prod_{\substack{\beta' \in J \\ \beta' \neq \beta}} \frac{1 - x_\beta / x_{\beta'} q_1^{d_i}}{1 - x_\beta / x_{\beta'}} :\prod_{\alpha=1}^k \sY_{i,x_\alpha} \prod_{\substack{\beta' \in J \\ \beta' \neq \beta}} \sA_{i,x_{\beta'};d_i,1}^{-1} \sS_{i,x'q_2^{-1}}:
\end{align}
\end{lemma}
\begin{proof}
We write $I = \{i_1, \ldots,i_l\}$ and $J = \{j_1,\ldots,j_{k-l}\}$.
Applying the operator product relation~\eqref{eq:YS_OPE}, we obtain
\begin{align}
    & \left[ :\prod_{\alpha=1}^k \sY_{i,x_\alpha} \prod_{\beta \in J} \sA_{i,x_\beta;d_i,1}^{-1}: , \sS_{i,x'} \right] \nonumber \\
    & = \Bigg( \prod_{\alpha = 1}^l \frac{1 - x'/x_{i_\alpha}}{1 - x'/x_{i_\alpha} q_1^{d_i}} \prod_{\beta = 1}^{k-l} \frac{1 -  x'/x_{j_\beta} q_1^{2d_i} q_2}{1 - x'/x_{j_\beta} q_1^{d_i} q_2} 
    \nonumber \\
    & \quad - \prod_{\alpha = 1}^l q_1^{d_i} \frac{1 - x_{i_\alpha}/x'}{1 - q_1^{d_i} x_{i_\alpha}/x'} \prod_{\beta = 1}^{k-l} q_1^{-d_i} \frac{1 -  x'/x_{j_\beta} q_1^{2d_i} q_2}{1 - x'/x_{j_\beta} q_1^{d_i} q_2} \Bigg) :\prod_{\alpha=1}^k \sY_{i,x_\alpha} \prod_{\beta \in J} \sA_{i,x_\beta;d_i,1}^{-1} \sS_{i,x'}: \, .
\end{align}
We may write
\begin{align}
    & \prod_{\alpha = 1}^l \frac{1 - x'/x_{i_\alpha}}{1 - x'/x_{i_\alpha} q_1^{d_i}} \prod_{\beta = 1}^{k-l} \frac{1 -  x'/x_{j_\beta} q_1^{2d_i} q_2}{1 - x'/x_{j_\beta} q_1^{d_i} q_2} - \prod_{\alpha = 1}^l q_1^{d_i} \frac{1 - x_{i_\alpha}/x'}{1 - q_1^{d_i} x_{i_\alpha}/x'} \prod_{\beta = 1}^{k-l} q_1^{-d_i} \frac{1 -  x'/x_{j_\beta} q_1^{2d_i} q_2}{1 - x'/x_{j_\beta} q_1^{d_i} q_2} \nonumber \\
    & = \prod_{\alpha = 1}^l \frac{1 - x'/x_{i_\alpha}}{1 - x'/x_{i_\alpha} q_1^{d_i}} \prod_{\beta = 1}^{k-l} \frac{1 -  x'/x_{j_\beta} q_1^{2d_i} q_2}{1 - x'/x_{j_\beta} q_1^{d_i} q_2} - q_1^{d_i} \frac{1 - x_{i_1}/x'}{1 - q_1^{d_i} x_{i_1}/x'} \prod_{\alpha = 2}^l \frac{1 - x'/x_{i_\alpha}}{1 - x'/x_{i_\alpha} q_1^{d_i}} \prod_{\beta = 1}^{k-l} q_1^{-d_i} \frac{1 -  x'/x_{j_\beta} q_1^{2d_i} q_2}{1 - x'/x_{j_\beta} q_1^{d_i} q_2} \nonumber \\
    & \quad + q_1^{d_i} \frac{1 - x_{i_1}/x'}{1 - q_1^{d_i} x_{i_1}/x'} \prod_{\alpha = 2}^l \frac{1 - x'/x_{i_\alpha}}{1 - x'/x_{i_\alpha} q_1^{d_i}} \prod_{\beta = 1}^{k-l} q_1^{-d_i} \frac{1 -  x'/x_{j_\beta} q_1^{2d_i} q_2}{1 - x'/x_{j_\beta} q_1^{d_i} q_2} - \prod_{\alpha = 1}^l q_1^{d_i} \frac{1 - x_{i_\alpha}/x'}{1 - q_1^{d_i} x_{i_\alpha}/x'} \prod_{\beta = 1}^{k-l} q_1^{-d_i} \frac{1 -  x'/x_{j_\beta} q_1^{2d_i} q_2}{1 - x'/x_{j_\beta} q_1^{d_i} q_2} \nonumber \\
    & = (1 - q_1^{d_i}) \delta\left(q_1^{-d_i}\frac{x'}{x_{i_1}}\right) \prod_{\alpha = 2}^l \frac{1 - x'/x_{i_\alpha}}{1 - x'/x_{i_\alpha} q_1^{d_i}} \prod_{\beta = 1}^{k-l} q_1^{-d_i} \frac{1 -  x'/x_{j_\beta} q_1^{2d_i} q_2}{1 - x'/x_{j_\beta} q_1^{d_i} q_2} \nonumber \\
    & \quad + q_1^{d_i} \frac{1 - x_{i_1}/x'}{1 - q_1^{d_i} x_{i_1}/x'} \prod_{\alpha = 2}^l \frac{1 - x'/x_{i_\alpha}}{1 - x'/x_{i_\alpha} q_1^{d_i}} \prod_{\beta = 1}^{k-l} q_1^{-d_i} \frac{1 -  x'/x_{j_\beta} q_1^{2d_i} q_2}{1 - x'/x_{j_\beta} q_1^{d_i} q_2} - \prod_{\alpha = 1}^l q_1^{d_i} \frac{1 - x_{i_\alpha}/x'}{1 - q_1^{d_i} x_{i_\alpha}/x'} \prod_{\beta = 1}^{k-l} q_1^{-d_i} \frac{1 -  x'/x_{j_\beta} q_1^{2d_i} q_2}{1 - x'/x_{j_\beta} q_1^{d_i} q_2} \, .
\end{align}
Repeating this process, we obtain
\begin{align}
    & (1 - q_1^{d_i}) \sum_{\alpha \in I} \delta\left(q_1^{-d_i}\frac{x'}{x_\alpha}\right) \prod_{\substack{\alpha'\in I \\ \alpha' \neq \alpha}} \frac{1 - q_1^{d_i} x_\alpha/x_{\alpha'}}{1 - x_\alpha/x_{\alpha'}} \prod_{\beta \in J} \frac{1 - x_\alpha/x_\beta q_1^{d_i} q_2}{1 - x_\alpha/x_\beta q_2} \nonumber \\
    & \quad + (1 - q_1^{-d_i}) \sum_{\beta \in J} \delta\left(q_1^{-d_i}q_2^{-1} \frac{x'}{x_\beta}\right) \prod_{\alpha \in I} \frac{1 - q_1^{d_i}q_2 x_\beta / x_\alpha}{1 - q_2 x_\beta / x_\alpha} \prod_{\substack{\beta' \in J \\ \beta' \neq \beta}} \frac{1 - x_\beta / x_{\beta'} q_1^{d_i}}{1 - x_\beta / x_{\beta'}} \, .
\end{align}
Recalling that
\begin{align}
    \delta\left(q_1^{-d_i}q_2^{-1} \frac{x'}{x_\beta}\right) {:\prod_{\alpha=1}^k \sY_{i,x_\alpha} \prod_{\beta' \in J} \sA_{i,x_{\beta'};d_i,1}^{-1} \sS_{i,x'}:} = q_1^{d_i} \delta\left(q_1^{-d_i}q_2^{-1} \frac{x'}{x_\beta}\right) :\prod_{\alpha=1}^k \sY_{i,x_\alpha} \prod_{\substack{\beta' \in J \\ \beta' \neq \beta}} \sA_{i,x_{\beta'};d_i,1}^{-1} \sS_{i,x' q_2^{-1}}: \, ,
\end{align}
we conclude the proof.
\end{proof}

\begin{proof}[Proof of Proposition~\ref{prop:multi_iWeyl}]
Fix $I$ and $J$ with $I \sqcup J = \{1,\ldots,k\}$.
For $\alpha \in I$, we have
\begin{align}
    & \prod_{\alpha' \in I, \beta \in J} \scS_{d_i}\left(\frac{x_{\alpha'}}{x_\beta}\right) \prod_{\substack{\alpha' \in I \\ \alpha' \neq \alpha}} \frac{1 - q_1^{d_i} x_\alpha/x_{\alpha'}}{1 - x_\alpha/x_{\alpha'}} \prod_{\beta \in J} \frac{1 - x_\alpha / x_\beta q_1^{d_i} q_2}{1 - x_\alpha / x_\beta q_2} \nonumber \\
    & = \prod_{\substack{\alpha' \in I, \alpha' \neq \alpha \\ \beta \in J}} \scS_{d_i}\left(\frac{x_{\alpha'}}{x_\beta}\right) \prod_{\substack{\alpha' \in I \\ \alpha' \neq \alpha}} \frac{1 - q_1^{d_i} x_\alpha/x_{\alpha'}}{1 - x_\alpha/x_{\alpha'}} \prod_{\beta \in J} \frac{1 - x_\alpha / x_\beta q_1^{d_i}}{1 - x_\alpha / x_\beta} 
\end{align}
and for $\beta \in J$, 
\begin{align}
    & \prod_{\alpha \in I, \beta' \in J} \scS_{d_i}\left(\frac{x_{\alpha}}{x_{\beta'}}\right) \prod_{\alpha \in I} \frac{1 - q_1^{d_i} q_2 x_\beta/x_{\alpha}}{1 - q_2 x_\beta/x_{\alpha}} \prod_{\substack{\beta' \in J \\ \beta' \neq \beta}} \frac{1 - x_\beta / x_{\beta'} q_1^{d_i}}{1 - x_\beta / x_{\beta'}} \nonumber \\
    & = \prod_{\substack{\alpha \in I \\ \beta' \in J, \beta' \neq \beta}} \scS_{d_i}\left(\frac{x_{\alpha}}{x_{\beta'}}\right) \prod_{\alpha \in I} \frac{1 - q_1^{d_i} x_\beta/x_{\alpha}}{1 - x_\beta/x_{\alpha}} \prod_{\substack{\beta' \in J \\ \beta' \neq \beta}} \frac{1 - x_\beta / x_{\beta'} q_1^{d_i}}{1 - x_\beta / x_{\beta'}} \, .
\end{align}
In the commutation relation, 
\begin{align}
    \left[ \sum_{I \sqcup J = \{1,\ldots,k\}} \prod_{\alpha \in I, \beta \in J} \scS_{d_i}\left(\frac{x_\alpha}{x_\beta}\right) :\prod_{\alpha=1}^k \sY_{i,x_\alpha} \prod_{\beta \in J} \sA_{i,x_\beta;d_i,1}^{-1}: , \mathsf{S}_{i,x'} \right] \, ,
\end{align}
there exist the following pair for any $\alpha \in I$ by Lemma~\ref{lemma:YAS_comm},
\begin{align}
    & (1 - q_1^{d_i}) \delta\left(q_1^{-d_i}\frac{x'}{x_\alpha}\right) \prod_{\alpha' \in I', \beta \in J} \scS_{d_i}\left(\frac{x_{\alpha'}}{x_\beta}\right) \prod_{\alpha' \in I'} \frac{1 - q_1^{d_i} x_\alpha/x_{\alpha'}}{1 - x_\alpha/x_{\alpha'}} \prod_{\beta \in J} \frac{1 - x_\alpha / x_\beta q_1^{d_i}}{1 - x_\alpha / x_\beta} :\prod_{\alpha'=1}^k \sY_{i,x_{\alpha'}} \prod_{\beta \in J} \sA_{i,x_\beta;d_i,1}^{-1} \sS_{i,x'}: \nonumber \\
    & \quad - (1 - q_1^{d_i}) \delta\left(q_1^{-d_i}q_2^{-1} \frac{x'}{x_\alpha}\right) \prod_{\alpha' \in I', \beta' \in J} \scS_{d_i}\left(\frac{x_{\alpha'}}{x_{\beta}}\right) \prod_{\alpha' \in I} \frac{1 - q_1^{d_i} x_\alpha/x_{\alpha'}}{1 - x_\alpha/x_{\alpha'}} \prod_{\beta \in J} \frac{1 - x_\alpha / x_{\beta} q_1^{d_i}}{1 - x_\alpha / x_{\beta}} :\prod_{\alpha'=1}^k \sY_{i,x_{\alpha'}} \prod_{\beta \in J} \sA_{i,x_\beta;d_i,1}^{-1} \sS_{i,x'q_2^{-1}}: \, ,
\end{align}
where $I' = I \backslash \{\alpha\}$.
As it is a $q_2$-difference with respect to $x'$, we conclude the proof.
\end{proof}

Summarizing these processes, we arrive at the following algorithm to compute the $qq$-character. 
\begin{enumerate}
    \item\label{alg:iWeyl1} Start with the highest-weight monomial associated with $w = (w_i)_{i \in \Gamma_0} \in \mathbb{Z}_{\ge 0}^{\operatorname{rk} \Gamma}$ with spectral parameters $x = (x_{i,\alpha})_{i\in \Gamma_0,\alpha =1,\ldots,w_i}$,%
    \footnote{%
    For $i \in \Gamma_0$, denote by $\mathbf{W}_i$ a vector space with the Grothendieck roots $(x_{i,\alpha})_{\alpha=1,\ldots,w_i}$, which would be identified with the $i$-th framing space of the quiver variety $\mathfrak{M}_{w,v}^\Gamma$ (see \S\ref{sec:geom_const}).
    Then, the Drinfeld polynomials $\{ \mathsf{P}_{i,u} \}_{i \in \Gamma_0}$ associated with the finite-dimensional module of the quantum affine algebra $U_q(\widehat{\mathfrak{g}}_\Gamma)$ are given by $\mathsf{P}_{i,u} = \operatorname{ch} \wedge_u \mathbf{W}_i$.
    }
    \begin{align}
        {:\sY_{w,x}:} = {:\prod_{i \in \Gamma_0} \prod_{\alpha = 1}^{w_i} \sY_{i,x_{i,\alpha}}:}
        \, .
    \end{align}
    We assume that the spectral parameters are generic, so that the $\scS$-functions appearing in the iWeyl reflection are not singular.
    \item\label{alg:iWeyl2} Apply the iWeyl reflection for the $\sY$-operator to generate a monomial, 
    \begin{align}
        \text{iWeyl} \ : 
        \sY_{i,x} \ \longmapsto \ {:\sY_{i,x} \sA_{i,x;d_i,1}^{-1}:} 
        = {:\sY_{i,x;d_i,1}^{-1} \prod_{e:i \to j} \prod_{r=0}^{d_i/d_{ij}-1} \sY_{j, x;rd_{ij},0} \prod_{e:j \to i} \prod_{r=0}^{d_i/d_{ij}-1} \sY_{j,x;(r+1)d_{ij},1}:}
        \, .
    \end{align}
    The reflection is applied to each $\sY$-operator in the numerator of each monomial.
    If the same monomial is obtained also from other reflections, it is added only once. 
    \begin{enumerate}
        \item\label{alg:iWeyl2a} If the monomial contains several $\sY$-operators which belong to the same node $i \in \Gamma_0$, multiply the $\scS$-functions:%
    \begin{align}
        :\sY_{i,x} \frac{\prod_{\alpha=1}^n \sY_{i,x_\alpha}}{\prod_{\beta=1}^m \sY_{i,x_\beta}}:
        \ \longmapsto \
        \qty(
        \frac{\prod_{\alpha=1}^n \scS_{d_i}(x_\alpha/x)}{\prod_{\beta=1}^m \scS_{d_i}(x_\beta/x)}
        )
        :\sY_{i,x} \sA_{i,x;d_i,1}^{-1} \frac{\prod_{\alpha=1}^n \sY_{i,x_\alpha}}{\prod_{\beta=1}^m \sY_{i,x_\beta}}:
        \, .
    \end{align}
    \item\label{alg:iWeyl2b}
    If the monomial contains the factor $\sY_{i,x} \sY_{i,x}$ or $\sY_{i,x} \sY_{i,xq_1^{d_i}q_2}$, it should be considered as $\lim_{x_1, x_2 \to x} \sY_{i,x_1} \sY_{i,x_2}$, $\lim_{x_1, x_2 \to x}\sY_{i,x_1} \sY_{i,x_2 q_1^{d_i}q_2}$, which gives rise to derivatives of the $\sY$-operator.
    The higher order term, e.g., $\qty(\sY_{i,x})^n$ ($n > 2$), is similarly understood.

    \end{enumerate}
    \item\label{alg:iWeyl3} If there is no $\sY$-operator of positive powers, no further reflection is applied (lowest-weight monomial).
    For finite-type quivers (The corresponding quiver Cartan matrix $c^{[0]}$ (see Definition~\ref{def:Cartan_matrix}) is positive-definite
    ), it is guaranteed that this process is terminated within finite time reflections, while the sequence of reflections are not terminated for generic quivers (affine and hyperbolic types; $c^{[0]}$ is not positive-definite).
\end{enumerate}

This is an algorithm to generate a $qq$-character that commutes with all the screening charges. 
In the step (\ref{alg:iWeyl2b}), the former combination $\sY_{i,x} \sY_{i,x}$ is always regularized using the limit, whereas the latter one $\sY_{i,x} \sY_{i,xq_1^{d_i}q_2}$ is not necessarily regular and may diverge due to the pole of the $\scS$-function.
If it diverges, this algorithm is not applicable at this moment. 
For example, the weight two $qq$-character of $A_1$ quiver \eqref{eq:qq-ch_A1_2} has the singular loci, $x_2 = x_1 q_1 q_2$, $x_1 (q_1 q_2)^{-1}$. 
See Conjecture~\ref{conj:irr} and Proposition~\ref{prop:residue_A1}.
We discuss in \S\ref{sec:D4} the $qq$-character for $D_4$ quiver, which involves the derivative term.

The $qq$-character of the tensor product module does not exhibits further decomposition for generic spectral parameters as in the case of the $q$-character.
In the following Sections, we study several examples to see the role of the spectral parameters for the irreducibility from the point of view of the $qq$-character.

\subsection{Contour integral formula}\label{sec:contour_integral}

We address another description of the $qq$-characters, which plays an essential role to discuss the geometric construction discussed in \S\ref{sec:geom_const}.

Let $\imath = \sqrt{-1}$.
We may write
\begin{align}
    {:\sY_{i,x} \sA_{i,x;d_i,1}^{-1}:} = \frac{1-q_1^{d_i}q_2}{(1-q_1^{d_i})(1-q_2)} \oint \sA_{i,z}^{-1} \sY_{i,x} \frac{\dd{z}}{2 \pi \imath z}
\end{align}
where the integration contour encircles the pole at $z = x q_1^{d_i} q_2$ of the $\scS$-function appearing from the operator product of $\sY$ and $\sA$.
Hence, the operator defined by
\begin{align}
 \mathsf{R}_i = \frac{1-q_1^{d_i}q_2}{(1-q_1^{d_i})(1-q_2)} \oint \sA_{i,z}^{-1} \frac{\dd{z}}{2 \pi \imath z} \label{eq:A_screening_ch} 
\end{align}
generates the iWeyl reflection, $\mathsf{R}_i : \sY_{i,x} \longmapsto {:\sY_{i,x} \sA_{i,x;d_i,1}^{-1}:}$.

In general, we have the following.
\begin{proposition}
Fix $v \in \{0,1,\ldots,k\}$.
The factor appearing in Proposition~\ref{prop:multi_iWeyl} has the following contour integral expression,
\begin{align}
    \sum_{\substack{I \sqcup J = \{1,\ldots,k\}\\|J| = v}} \prod_{\alpha \in I, \beta \in J} \scS_{d_i}\left(\frac{x_\alpha}{x_\beta}\right) {:\prod_{\alpha=1}^k \sY_{i,x_\alpha} \prod_{\beta \in J} \sA_{i,x_\beta;d_i,1}^{-1}:} 
    & = \frac{1}{v!} \left( \frac{1-q_1^{d_i}q_2}{(1-q_1^{d_i})(1-q_2)} \right)^{v} \oint \prod_{I=1}^v \sA_{i,z_I}^{-1} :\prod_{\alpha=1}^k \sY_{i,x_\alpha}: [dz] \, ,
\end{align}    
where we write
\begin{align}
    [\dd{z}] = \prod_{I = 1}^v \frac{\dd{z}_I}{2 \pi \imath z_I} \, ,
\end{align}
and the integration contour encircles the poles at $\{x_\alpha q_1^{d_i} q_2\}_{\alpha = 1, \ldots,k}$.
\end{proposition}
\begin{proof}
    From Definition~\ref{def:YA_op}, we have the operator product factor for the $\sA$-operators,
    \begin{align}
        \sA_{i,z} \sA_{j,z'} = {:\sA_{i,z} \sA_{i,z'}:}
        \times
        \begin{cases}
            \displaystyle
            \scS_{d_i}\left(\frac{z'}{z}\right)^{-1} \scS_{d_i}\left(\frac{z}{z'}\right)^{-1} & (i=j) \\
            \displaystyle \prod_{r=0,\ldots,d_j/d_{ij}-1}
            \scS_{d_i}\left(\nu_e q_1^{rd_{ij}}\frac{z'}{z}\right) & (e: i \to j) \\
            \displaystyle
            \prod_{r=0,\ldots,d_j/d_{ij}-1}
            \scS_{d_i}\left(\nu_e^{-1}  q_1^{rd_{ij}+1}q_2\frac{z'}{z}\right) & (e: j \to i) \\
            1 & (\text{otherwise})
        \end{cases}
    \end{align}
    Based on these factors, we write the contour integral of the $v$-variable sector as follows,
    \begin{align}
        \frac{1}{v!} \oint \prod_{I=1}^v \sA_{i,z_I}^{-1} :\prod_{\alpha=1}^k \sY_{i,x_\alpha}: [dz] 
        & = \frac{1}{v!} \oint :\prod_{\alpha=1}^k \sY_{i,x_\alpha} \prod_{I=1}^v \sA_{i,z_I}^{-1}: \prod_{\substack{\alpha = 1,\ldots,k \\ I = 1,\ldots,v}} \scS_{d_i} \left(\frac{z_I}{x_\alpha}\right) \prod_{1 \le I \neq J \le v} \scS_{d_i}\left(\frac{z_J}{z_I}\right)^{-1} [dz] \, .
    \end{align}
    Evaluating the residue of the poles at $\{x_\alpha q_1^{d_i} q_2\}_{\alpha = 1, \ldots,k}$, we obtain the result.
    This integral is identically zero when $v > k$ since there is no corresponding pole.
\end{proof}

\begin{theorem}\label{thm:qq-ch_int1}
    We assume that there exist no loop edge in the quiver.
    For a dimension vector $w = (w_i)_{i \in \Gamma_0}$ and spectral parameters $x = (x_{i,\alpha})_{i \in \Gamma_0,\alpha = 1, \ldots,w_i}$, we set the highest-weight monomial, ${:\sY_{w,x}:} = {:\prod_{i \in \Gamma_0} \prod_{\alpha = 1}^{w_i} \sY_{i,x_{i,\alpha}} :}$.
    We have
    \begin{align}
        \sT_{w,x}^{(q_1,q_2)} = \sum_{v} \sT_{w,v,x}^{(q_1,q_2)} \, ,
    \end{align}
    where $v = (v_i)_{i \in \Gamma_0} \in \mathbb{Z}_{\ge 0}^{\operatorname{rk} \Gamma}$.
    Each contribution is given by a multi-variable contour integral,
    \begin{align}
        \sT_{w,v,x}^{(q_1,q_2)} & = \frac{\mathsf{c}_{v}}{v!} \oint \sA_{v,\underline{z}}^{-1} {:\sY_{w,x}:} [\dd{\underline{z}}] \, , \qquad 
        \sA_{v,\underline{z}}^{-1} = \prod_{i \in \Gamma_0} \prod_{I = 1}^{v_i} \sA_{i,z_I}^{-1} \, ,
    \end{align}
    where
    \begin{align}
        \mathsf{c}_{v} = \prod_{i \in \Gamma_0} \left(\frac{1-q_1^{-d_i}q_{2}^{-1}}{(1-q_{1}^{-d_i})(1-q_{2}^{-1})}\right)^{v_i} \, , \qquad 
        v! = \prod_{i \in \Gamma_0} v_i! \, , \qquad 
        [\dd{\underline{z}}] = \prod_{i \in \Gamma_0} \prod_{I=1}^{v_i} \frac{\dd{z}_{i,I}}{2 \pi \iota z_{i,I}} \, .
        \label{eq:int_notation}
    \end{align}
    The integration contour encircles the poles of the operator product factors to be consistent with the iWeyl reflection: We first take poles at $\{x_{i,\alpha} q_1^{d_i} q_2\}$ for $z_{i,I}$ and then take poles associated with the $\sY$-operators generated by the iWeyl reflections.
 Moreover, using the operator~\eqref{eq:A_screening_ch}, we also have the following formal expression~\cite{Kimura:2019hnw},
\begin{align}
 \sT_{w,v,x}^{(q_1,q_2)} & = \sum_{v} \frac{1}{v!} \left(\prod_{i \in \Gamma_0} \mathsf{R}_i^{v_i}\right) {:\sY_{w,x}:} = \exp\left(\sum_{i\in\Gamma_0} \mathsf{R}_i\right) {:\sY_{w,x}:} \, .
\end{align}
\end{theorem}

\begin{proof}
    Under the assumption that there exist no loop edge in the quiver, using the operator product factors of $\sY$ and $\sA$, we have
    \begin{align}
    & \sT_{w,v,x}^{(q_1,q_2)} \nonumber \\
    & = 
    \frac{\mathsf{c}_{v}}{v!} \oint :
    \frac{\sY_{w,x}}{\sA_{v,\underline{z}}} :
    \prod_{i \in \Gamma_0} \left[ \prod_{\substack{\alpha = 1, \ldots, w_i \\ I = 1,\ldots, v_i}} \scS_{d_i}\left(\frac{z_{i,I}}{x_{i,\alpha}}\right) \prod_{1 \le I \neq J \le v_i} \scS_{d_i}\left(\frac{z_{i,J}}{z_{i,I}}\right)^{-1} \right]
    \prod_{e:i \to j} \prod_{\substack{I = 1,\ldots,v_i \\ J = 1,\ldots, v_j \\ r = 0, \ldots, d_j/d_{ij}-1}} \scS_{d_i}\left(\nu_e q_1^{r d_{ij}} \frac{z_{j,J}}{z_{i,I}}\right) [\dd{\underline{z}}] \, .
\end{align}
We first take the poles of $\scS_{d_i}(z_{i,I}/x_{i,\alpha})$ which yields the iWeyl reflection for the $\sY$-operators in the highest-weight monomial.
Once the variable $z_{i,I}$ takes a pole, we then iteratively take the pole of $\scS_{d_i}(\nu_e q_1^{r d_{ij}} z_{j,J}/z_{i,I})$ to be consistent with the iWeyl reflection.
We do not take a pole of $\scS_{d_i}(z_{j,J}/z_{i,I})^{-1}$ for a quiver without any loop edge.
\end{proof}

In fact, the choice of the integration contour mentioned above is consistent with the Jeffrey--Kirwan residue prescription~\cite{Jeffrey:1993cun} under a small modification of the $\mathscr{S}$-function.
Such a contour integral formula based on the vertex operator formalism is a consequence of the so-called \emph{BPS/CFT correspondence}~\cite{Nekrasov:2015wsu,Nekrasov:2016qym,Nekrasov:2016ydq,Nekrasov:2017rqy,Nekrasov:2017gzb}.
See \cite{Kimura:2019hnw,Kimura:2022zsm,Kimura:2023bxy,Kimura:2024xpr,Kimura:2024osv,Kimura:2025lfo} for more details about the BPS/CFT correspondence for the $qq$-character.


\section{$A_1$ quiver}\label{sec:A1}

The $A_1$ quiver consists of a single node and no edge: $\Gamma_0 = (1)$, $\Gamma_1 = \varnothing$.
We write $\sY_{1,x} = \sY_{x}$ for simplicity.
The iWeyl reflection of $A_1$ quiver is given by
\begin{align}
    \text{iWeyl} : \quad
    \sY_x \ \longmapsto \ \sY_{xq_1 q_2}^{-1} \, .
\end{align}
The fundamental $qq$-character is then given by
\begin{align}
    \sT_{1,x} = \sY_x + \sY_{xq_1 q_2}^{-1} \, .
\end{align}
\begin{remark}
    In the following Sections, we apply the following notations whenever no confusion arises:
    \begin{itemize}
        \item We write $q = q_1 q_2$.
        \item We do not write the $q_1$, $q_2$ dependence of the $qq$-character: $\sT_{w,x} = \sT_{w,x}^{(q_1,q_2)}$.
        \item We do not write the normal ordering symbol: $\sY \cdots \sY = {:\sY \cdots \sY:}$.
    \end{itemize}
\end{remark}
From the direct computation based on the algorithm shown in \S\ref{sec:iWeyl_ref}, we obtain the following expression.
\begin{proposition}\label{prop:A1_weight_gen}
For $A_1$ quiver, the $qq$-character of weight $w = (w)$ with generic parameters $x = (x_1,\ldots,x_w)$ is given by 
\begin{align}
    \sT_{w,x} 
    & = \sY_{x_1} \cdots \sY_{x_w} + \cdots
    = 
    \sum_{I \sqcup J = \{1,\ldots,w\}} \prod_{i \in I,j \in J} \scS\qty(\frac{x_i}{x_j}) \prod_{i \in I} \sY_{x_i} \prod_{j \in J} \sY_{x_j q}^{-1} \, ,
    \label{eq:A1_weight_gen}
\end{align}
which contains $2^w$ monomials in total.    
\end{proposition}
We remark that the $\scS$-factor appearing as the coefficient of the monomial in this formula is identified with the fixed point contribution of the equivariant integral over the quiver variety of $A_1$, identified with the cotangent bundle of the Grassmannian $T^\vee\mathrm{Gr}(v,w)$ with $v=|J|$ and ${w \choose v}$ fixed points.

\subsection{Weight two}

From the formula \eqref{eq:A1_weight_gen}, the $qq$-character of weight two $w = (2)$ with the spectral parameters $x = (x_1,x_2)$ is given by
\begin{align}
    \sT_{2,x} = \sY_{x_1} \sY_{x_2} + \scS\qty(\frac{x_2}{x_1}) \frac{\sY_{x_2}}{\sY_{x_1 q}} + \scS\qty(\frac{x_1}{x_2}) \frac{\sY_{x_1}}{\sY_{x_2q}} + \frac{1}{\sY_{x_1 q} \sY_{x_2 q}}
    \, .
    \label{eq:A1_T2}
\end{align}
For generic spectral parameters, this $qq$-character contains four monomials, corresponding to the tensor product of two two-dimensional modules of the quantum affine algebra $U_q(\widehat{\mathfrak{sl}}_2)$ associated with $A_1$ quiver.
As mentioned in \S\ref{sec:intro_irr}, we obtain non-highest dominant monomial ``1'' by specialization $x_1 / x_2 = q^{\pm 1}$, which appears at the pole of the $\scS$-function.
We see that the residue of the pole $x_1 = x_2$ is zero.
Hence, the $qq$-character has a unique dominant monomial, so that it is irreducible for any $(x_1,x_2)$ except on the singular loci.

Recalling that the $\scS$-function has zeros at $q_1$ and $q_2$, we may truncate the $qq$-character, which corresponds to the irreducible module of weight two, by specializing the spectral parameters,
\begin{align}
    \sT_{2,x} \ \longrightarrow \
    \widetilde{\sT}_{2,x}^{(q_1, q_2)} = 
    \begin{cases}
    \displaystyle 
    \sY_{x} \sY_{x;1,0} + \scS(q_1^{-1}) \frac{\sY_x}{\sY_{x;2,1}} + \sY_{x;1,1}^{-1} \sY_{x;2,1}^{-1} & (x_1 = x, x_2/x_1 = q_1)
    \\[1em] \displaystyle
    \sY_{x} \sY_{x;0,1} + \scS(q_2^{-1}) \frac{\sY_x}{\sY_{x;1,2}} + \sY_{x;1,1}^{-1} \sY_{x;1,2}^{-1} & (x_1 = x, x_2/x_1 = q_2)
    \end{cases}
    \label{eq:A1_T2_red}
\end{align}
This module corresponds to the following choice of the Drinfeld polynomial,
\begin{align}
\mathsf{P}(u) = (1-ux)(1-uxq_1) = \operatorname{ch} \wedge_u \mathbf{W}^{(1)}_{2_1,x} \, ,\qquad \mathsf{P}(u) = (1-ux)(1-uxq_2) = \operatorname{ch} \wedge_u \mathbf{W}^{(1)}_{2_2,x} \, .
\end{align}
where we define the vector spaces $\left(\mathbf{W}_{k_m,x}^{(i)}\right)_{i \in \Gamma_0}$ having the Chern character
\begin{align}
    \operatorname{ch} \mathbf{W}_{k_m,x}^{(i)} = 
    \begin{cases}
        x + x q_1^{d_i} + \cdots + x q_1^{d_i(k_1-1)} & (m = 1) \\
        x + x q_2 + \cdots + x q_2^{k_2-1} & (m = 2) 
    \end{cases}
\end{align}
See \S\ref{sec:notation_geom} for the notation.
Namely, this specialization \eqref{eq:A1_T2_red} corresponds to the KR module of weight two with the shift parameter $q_1, q_2$.
We also denote the corresponding $qq$-character by $\widetilde{\sT}_{2,x}^{(q_1, q_2)} = \sT[\mathbf{W}_{2_{m},x}^{(1)}]$ for $m = 1, 2$.

We remark a geometric interpretation of the truncation of the $qq$-character.
Before the truncation, there are four monomials in the expression \eqref{eq:A1_T2}.
The first and the last ones have no specific geometric structure as the corresponding quiver variety is just a point.
The second and the third ones correspond to the fixed point contributions of the cotangent bundle $T^\vee\mathbb{P}^1$, which is the quiver variety of $(v,w) = (1,2)$.
Tuning the spectral parameters, we can extract either of two monomials as the subvariety contribution.

\subsubsection*{Classical limit}

In the classical limit, the $\scS$-factor behaves as follows,
\begin{subequations}\label{eq:S_lim}
\begin{align}
    \scS(q_1^{-1}) 
    & = (1 + q_1^{-1})\frac{1 - q_1^{-1} q_2^{-1}}{1 - q_1^{-2} q_2^{-1}}
    \ \longrightarrow \
    \begin{cases}
    2 & (q_1 \to 1) \\ 1 & (q_2 \to 1)
    \end{cases} 
    \\
    \scS(q_2^{-1}) 
    & = (1 + q_2^{-1})\frac{1 - q_1^{-1} q_2^{-1}}{1 - q_1^{-1} q_2^{-2}}
    \ \longrightarrow \
    \begin{cases}
    1 & (q_1 \to 1) \\ 2 & (q_2 \to 1)
    \end{cases}
\end{align}
\end{subequations}
Hence, the $qq$-character of weight two \eqref{eq:A1_T2_red} is reduced as
\begin{subequations}
\begin{align}
    \sY_{x} \sY_{x;1,0} + \scS(q_1^{-1}) \frac{\sY_x}{\sY_{x;2,1}} + \sY_{x;1,1}^{-1} \sY_{x;2,1}^{-1}
    \ \longrightarrow \
    \begin{cases}
    \displaystyle
    \sY_{x}^2 + 2 \frac{\sY_x}{\sY_{x;0,1}} + \sY_{x;0,1}^{-2} = \qty( \sY_x + \sY_{x;0,1}^{-1} )^2  & (q_1 \to 1)
    \\[1em] \displaystyle
    \sY_{x} \sY_{x;1,0} + \frac{\sY_x}{\sY_{x;2,0}} + \sY_{x;1,0}^{-1} \sY_{x;2,0}^{-1}  & (q_2 \to 1)
    \end{cases}
    \\
    \sY_{x} \sY_{x;0,1} + \scS(q_2^{-1}) \frac{\sY_x}{\sY_{x;1,2}} + \sY_{x;1,1}^{-1} \sY_{x;1,2}^{-1}
    \ \longrightarrow \
    \begin{cases}
    \displaystyle
    \sY_{x} \sY_{x;0,1} + \frac{\sY_x}{\sY_{x;0,2}} + \sY_{x;0,1}^{-1} \sY_{x;0,2}^{-1}  & (q_1 \to 1)
    \\[1em] \displaystyle
    \sY_{x}^2 + 2 \frac{\sY_x}{\sY_{x;1,0}} + \sY_{x;1,0}^{-2} = \qty( \sY_x + \sY_{x;1,0}^{-1} )^2  & (q_2 \to 1)
    \end{cases}
\end{align}
\end{subequations}
Therefore, we obtain the $q$-character associated with the KR module $\mathbf{W}_{k_1,x}^{(1)}$ ($\mathbf{W}_{k_2,x}^{(1)}$, resp.) in the limit $q_2 \to 1$ ($q_1 \to 1$, resp.), while the degenerated $q$-character, corresponding to the tensor product of the two-dimensional modules, is obtained in the limit $q_1 \to 1$ ($q_2 \to 1$, resp.).
We emphasize that the $q$-character is commutative (Remark~\ref{rmk:q-ch_commute}).
Such a degenerated situation (folded $q$-character) has been recently discussed in the literature~\cite{Chen:2018ntf,Frenkel:2021bmx,Kashiwara:2022agj}.

\subsection{Weight three}

The $qq$-character of weight three with generic spectral parameters $x = (x_1,x_2,x_3)$ is given by
\begin{align}
    \sT_{3,x} 
    & = \sY_{x_1} \sY_{x_2} \sY_{x_3} 
    \nonumber \\
    & \quad
    + \scS\qty(\frac{x_1}{x_3}) \scS\qty(\frac{x_2}{x_3}) \frac{\sY_{x_1} \sY_{x_2}}{\sY_{x_3 q}}
    + \scS\qty(\frac{x_1}{x_2}) \scS\qty(\frac{x_3}{x_2}) \frac{\sY_{x_1} \sY_{x_3}}{\sY_{x_2 q}}
    + \scS\qty(\frac{x_2}{x_1}) \scS\qty(\frac{x_3}{x_1}) \frac{\sY_{x_2} \sY_{x_3}}{\sY_{x_1 q}}
    \nonumber \\
    & \quad\quad
    + \scS\qty(\frac{x_1}{x_2}) \scS\qty(\frac{x_1}{x_3}) \frac{\sY_{x_1}} {\sY_{x_2q} \sY_{x_3 q}}
    + \scS\qty(\frac{x_2}{x_1}) \scS\qty(\frac{x_2}{x_3}) \frac{\sY_{x_2}} {\sY_{x_1 q} \sY_{x_3 q}}
    + \scS\qty(\frac{x_3}{x_1}) \scS\qty(\frac{x_3}{x_2}) \frac{\sY_{x_3}} {\sY_{x_1 } \sY_{x_2 q}}
    \nonumber \\
    & \quad\quad\quad
    + \frac{1}{\sY_{x_1 q} \sY_{x_2 q} \sY_{x_3 q}}    
    \, .
\end{align}
We see that the non-highest dominant monomial appears at the residue of the poles.
For example, the residue of the pole at $x_2 = x_1 q$, where the monomial $\sY_{x_2} \sY_{x_3} / \sY_{x_1 q}$ becomes dominant, is given by 
\begin{align}
    \operatorname*{Res}_{x_2 = x_1 q} \sT_{3,x} & = \operatorname*{Res}_{x_2 = x_1 q} \scS\left(\frac{x_2}{x_1}\right) \left( \scS\left(\frac{x_3}{x_1}\right) \frac{\sY_{x_2}\sY_{x_3}}{\sY_{x_1 q}} + \scS\left(\frac{x_2}{x_3}\right) \frac{\sY_{x_2}}{\sY_{x_1 q}\sY_{x_3 q}} \right)
    \nonumber \\
    & = - \frac{(1-q_1)(1-q_2)}{1-q} \scS\left(\frac{x_3}{x_1}\right) \left( \sY_{x_3} + \sY_{x_3 q}^{-1} \right) \, ,
\end{align}
which is proportional to the fundamental $qq$-character.
On the other hand, the residue of the pole at $x_2 = x_1$ turns out to be zero, which does not provide any dominant monomial.
Hence, the $qq$-character is irreducible for any spectral parameters except on the singular loci.

Specializing the parameters $x = (x,xq_1,xq_1^2)$, corresponding to the KR module of weight three $\mathbf{W}_{3_1,x}^{(1)}$, we obtain
\begin{align}
    \sT_{3,x}
    \ \longrightarrow \
    \sT[\mathbf{W}_{3_1,x}^{(1)}]
    = \sY_{x} \sY_{x;1,0} \sY_{x;2,0}
    + \scS_2(q_1^{-1}) \qty(\frac{\sY_x \sY_{x;1,0}}{\sY_{x;3,1}} + \frac{\sY_x}{\sY_{x;2,1} \sY_{x;3,1}} )
    + \frac{1}{\sY_{x;1,1} \sY_{x;2,1} \sY_{x;3,1} }
    \, .
\end{align}
This corresponds to the irreducible four-dimensional module of $U_q(\widehat{\mathfrak{sl}}_2)$ associated with $A_1$ quiver.
We obtain a similar expression by another specialization $x = (x,xq_2,xq_2^2)$.

\subsubsection*{Classical limit}
In the classical limit, the $\scS$-factor is given by
\begin{align}
    \scS_2(q_1^{-1}) = (1 + q_1^{-1} + q_1^{-2}) \frac{1 - q_1^{-1} q_2^{-1}}{1 - q_1^{-3} q_2^{-1}}
    \ \longrightarrow \
    \begin{cases}
     3 & (q_1 \to 1) \\ 1 & (q_2 \to 1)
    \end{cases}
    \label{eq:S_lim2}
\end{align}
Hence, the $qq$-character is reduced to the degenerated $q$-character $(q_1 \to 1)$ and the ordinary $q$-character of weight three $(q_2 \to 1)$ as follows,
\begin{align}
    \sT[\mathbf{W}_{3_1,x}^{(1)}]
    \ \longrightarrow \
    \begin{cases}
    \displaystyle
     \sY_x^3 + \frac{3 \sY_x^2}{\sY_{x;0,1}} + \frac{3 \sY_x}{\sY_{x;0,1}^2} + \frac{1}{\sY_{x;0,1}^3} = \qty(\sY_x + \sY_{x;0,1}^{-1})^3 & (q_1 \to 1) 
     \\[1em] \displaystyle
     \sY_{x} \sY_{x;1,0} \sY_{x,2,0}
    + \frac{\sY_x \sY_{x;1,0}}{\sY_{x;3,0}} + \frac{\sY_x}{\sY_{x;2,0} \sY_{x;3,0}} 
    + \frac{1}{\sY_{x;1,0} \sY_{x;2,0} \sY_{x;3,0} } & (q_2 \to 1)
    \end{cases}
\end{align}

\subsection{General weight}

We consider the $qq$-character of $A_1$ quiver for general weight.
We have the following result for the residue structure of the $qq$-character.
\begin{proposition}\label{prop:residue_A1}
    Let $\sT_{w,x}$ be the $qq$-character of $A_1$ quiver of weight $w$ with spectral parameters $x = (x_1,\ldots,x_w)$.
    For $i, j \in \{1,\ldots,w\}$ such that $i \neq j$, set $x_{i,j} = x\backslash\{x_i,x_j\}$.
    We have
    \begin{align}
        \operatorname*{Res}_{x_i = x_j} \sT_{w,x} = 0 \, , \qquad 
        \operatorname*{Res}_{x_i = x_j q} \sT_{w,x} = - \frac{(1-q_1)(1-q_2)}{1-q} \left( \prod_{k\in\{1,\ldots,w\}\backslash\{i,j\}} \scS\left(\frac{x_k}{x_j}\right) \right) \sT_{w-2,x_{ij}} \, .
    \end{align}
\end{proposition}
\begin{proof}
    We fix $i, j \in \{1,\ldots,w\}$ such that $i \neq j$.
    By Proposition~\ref{prop:A1_weight_gen}, the contributions to the residue of the pole at $x_i = x_j$ come from the factor with $\scS(x_i/x_j)$ and $\scS(x_j/x_i)$.
    Hence, we have
    \begin{align}
        \operatorname*{Res}_{x_i = x_j} \sT_{w,x} & = \operatorname*{Res}_{x_i = x_j} \left(\scS\qty(\frac{x_i}{x_j}) \frac{\sY_{x_i}}{\sY_{x_j}} + \scS\qty(\frac{x_{j}}{x_{i}}) \frac{\sY_{x_j}}{\sY_{x_i}}\right) \nonumber \\
        & \quad \times \sum_{I \sqcup J = \{1,\ldots,w\}\backslash\{i,j\}} \prod_{i' \in I} \scS\qty(\frac{x_{i'}}{x_{j}}) \prod_{j' \in J} \scS\qty(\frac{x_{i}}{x_{j'}}) \prod_{i' \in I,j' \in J} \scS\qty(\frac{x_{i'}}{x_{j'}}) \prod_{i' \in I} \sY_{x_{i'}} \prod_{j' \in J} \sY_{x_{j'} q}^{-1}
        \nonumber \\
        & = 0 \, .
    \end{align}
    For the pole at $x_i = x_j q$, we take into account the factor with $\scS(x_i/x_j)$, which is given by
    \begin{align}
        \operatorname*{Res}_{x_i = x_j q} \sT_{w,x} & = \operatorname*{Res}_{x_i = x_j q} \scS\left(\frac{x_i}{x_j}\right) \frac{\sY_{x_i}}{\sY_{x_j q}} \sum_{I \sqcup J = \{1,\ldots,w\}\backslash\{i,j\}} \prod_{i' \in I} \scS\qty(\frac{x_{i'}}{x_{j}}) \prod_{j' \in J} \scS\qty(\frac{x_{i}}{x_{j'}}) \prod_{i' \in I,j' \in J} \scS\qty(\frac{x_{i'}}{x_{j'}}) \prod_{i' \in I} \sY_{x_{i'}} \prod_{j' \in J} \sY_{x_{j'} q}^{-1} \nonumber \\
        & = - \frac{(1-q_1)(1-q_2)}{1-q} \left( \prod_{k\in\{1,\ldots,w\}\backslash\{i,j\}} \scS\left(\frac{x_k}{x_j}\right) \right) \sT_{w-2,x_{ij}} \, ,
    \end{align}
    where we use the inversion relation~\eqref{eq:S_reflection}.
\end{proof}
This means that the $qq$-character is irreducible for any spectral parameters except on the singular loci, and each pole is associated with a non-highest $qq$-character, which proves Conjecture~\ref{conj:irr} for $A_1$ quiver.

\begin{theorem}\label{thm:A_1_I}
Under the specialization $x = (x,xq_1,\ldots,xq_1^{w-1})$ corresponding to the KR module $\mathbf{W}_{w_1,x}^{(1)}$ of degree $w$, the $qq$-character consists of $w+1$ monomials,
\begin{align}
    \sT[\mathbf{W}_{w_1,x}^{(1)}] = \sum_{v=0}^w \qty( \prod_{\substack{i=1,\ldots,w-v \\ j = 1,\ldots,v}} \scS(q_1^{1-i-j}) )
    \prod_{i=1,\ldots,w-v} \sY_{x;i-1,0} \prod_{j=w-v+1,\ldots,w} \sY_{x;j,1}^{-1}
    \, .
    \label{eq:A1_Tw_red}
\end{align}
\end{theorem}
\begin{proof}
    Under the specialization $x = (x,xq_1,\ldots,xq_1^{w-1})$, due to the $\mathscr{S}$-factors, there are only the contributions of $I = \{1,\ldots,w-v\}$, $J = \{w-v+1,\ldots,w\}$ remaining in the $qq$-character formula \eqref{eq:A1_weight_gen}, which gives rise to the expression above.
\end{proof}
This proves the part (\ref{conj1}) of Conjecture~\ref{conj:irreducibility} for $A_1$ quiver.

\subsubsection*{Classical limit}

Let us consider the classical limit of the $qq$-character.
We have the following results in the limits $q_1 \to 1$ and $q_2 \to 1$, respectively.
\begin{theorem}\label{thm:A_1_II}
The classical limit of the $qq$-character gives rise to the degenerated $q$-character and the weight $w$ $q$-character,
\begin{align}
    \sT[\mathbf{W}_{w_1,x}^{(1)}]
    \ \longrightarrow \
    \begin{cases}
    \displaystyle
    \sum_{v=0}^w { w \choose v}
    \qty(\sY_{x})^{w-v} \qty(\sY_{x;0,1})^{-v} = \qty(\sY_x + \sY_{x;0,1}^{-1})^w & (q_1 \to 1) 
    \\[1em] \displaystyle
    \sum_{v=0}^w 
    \prod_{i=1,\ldots,w-v} \sY_{x;i-1,0} \prod_{j=w-v+1,\ldots,w} \sY_{x;j,0}^{-1} & (q_2 \to 1) 
    \end{cases}
    \, .
\end{align}    
\end{theorem}
\begin{proof}
Applying the expression \eqref{eq:S-fn}, we may rewrite the $\scS$-factor in the $qq$-character \eqref{eq:A1_Tw_red} as follows,
\begin{align}
    \prod_{\substack{i=1,\ldots,w-v \\ j = 1,\ldots,v}} \scS(q_1^{1-i-j})
    & = 
    \frac{(q_1^{-1} q_2^{-1};q_1^{-1})_\infty (q_1^{-w+v-1};q_1^{-1})_\infty (q_1^{-v-1};q_1^{-1})_\infty (q_1^{-w} q_2^{-1};q_1^{-1})_\infty}{(q_1^{-1};q_1^{-1})_\infty (q_1^{-w+v-1} q_2^{-1};q_1^{-1})_\infty (q_1^{-v-1} q_2^{-1};q_1^{-1})_\infty (q_1^{-w};q_1^{-1})_\infty}
    \nonumber \\
    & = 
    \frac{(q_1^{-1} q_2^{-1};q_1^{-1})_\infty (q_1^{-w} q_2^{-1};q_1^{-1})_\infty}{(q_1^{-w+v-1} q_2^{-1};q_1^{-1})_\infty (q_1^{-v-1} q_2^{-1};q_1^{-1})_\infty}
    \frac{\Gamma_{q_1^{-1}}(w) \Gamma_{q_1^{-1}}(1)}{\Gamma_{q_1^{-1}}(v) \Gamma_{q_1^{-1}}(w-v)}
    \ \longrightarrow \
    \begin{cases}
    \displaystyle
    { {w} \choose {v} } & (q_1 \to 1) \\ 1 & (q_2 \to 1)
    \end{cases}
    \, ,
\end{align}
where we define the $q$-gamma function for $|q|<1$,
\begin{align}
    \Gamma_q(x) = (1 - q)^x \frac{(q;q)_\infty}{(q^x;q)_\infty}
    \ \xrightarrow{q \to 1} \ \Gamma(x)
\end{align}
with the $q$-shifted factorial ($q$-Pochhammer symbol)
\begin{align}
    (z;q)_\infty = \prod_{n=0}^\infty (1 - z q^n)
    \, .
\end{align}    
Then, we obtain the expression above.
\end{proof}
This completes a proof of the parts (\ref{conj2}) and (\ref{conj3}) of Conjecture~\ref{conj:irreducibility} for $A_1$ quiver.

\section{$A_2$ quiver}\label{sec:A2}

In this case, there are two $\sY$-variables, and the corresponding iWeyl reflection is given as follows,
\begin{align}
    \text{iWeyl} : \quad
    \qty(\sY_{1,x}, \sY_{2,x})
    \ \longmapsto \
    \qty( \frac{\sY_{2,x}}{\sY_{1,xq}}, \frac{\sY_{1,xq}}{\sY_{2,xq}} ) 
    \, .
\end{align}
Then, the fundamental $qq$-characters are given by
\begin{subequations}
\begin{align}
    \sT_{1,x} = \sT_{(1,0),x} & = \sY_{1,x} + \frac{\sY_{2,x}}{\sY_{1,xq}} + \frac{1}{\sY_{2,xq}}
    \, , \\
    \sT_{2,x} = \sT_{(0,1),x} & = \sY_{2,x} + \frac{\sY_{1,xq}}{\sY_{2,xq}} + \frac{1}{\sY_{1,xq^2}}
    \, ,
\end{align}
\end{subequations}
which correspond to two three-dimensional modules of $U_q(\widehat{\mathfrak{sl}}_3)$ associated with $A_2$ quiver.

\subsection{Weight two}

For $A_2$ quiver, we have three possible weight two situations, $w = (2,0), (0,2)$ and $w = (1,1)$.
The former two cases correspond to the KR modules, $\qty(\mathbf{W}_{2_m,x}^{(i)})_{i=1,2,m=1,2}$, while the last one is not.

\subsubsection{$w = (2,0)$}

We consider the $qq$-character associated with weight $w = (2,0)$ with generic parameters $x = (x_1,x_2)$,
\begin{align}
    \sT_{(2,0),x} & = \sY_{1,x_1} \sY_{1,x_2} + \frac{\sY_{2,x_1} \sY_{2,x_2}}{\sY_{1,x_1 q} \sY_{1,x_2 q}} + \frac{1}{\sY_{2,x_1 q} \sY_{2,x_2 q}}
    \nonumber \\ &
    \qquad +
    \scS\qty(\frac{x_2}{x_1}) \qty[ \frac{\sY_{1,x_2} \sY_{2,x_1}}{\sY_{1,x_1 q}} + \frac{\sY_{1,x_2}}{\sY_{2,x_1 q}} + \frac{\sY_{2,x_2}}{\sY_{1,x_2 q} \sY_{2,x_1 q}} ]
    \nonumber \\ &
    \qquad \qquad +
    \scS\qty(\frac{x_1}{x_2}) \qty[ \frac{\sY_{1,x_1} \sY_{2,x_2}}{\sY_{1,x_2 q}} + \frac{\sY_{1,x_1}}{\sY_{2,x_2 q}} + \frac{\sY_{2,x_1}}{\sY_{1,x_1 q} \sY_{2,x_2 q}} ]
    \, .
    \label{eq:A2_T_20}
\end{align}
The flow of the iWeyl reflections is described by the Hasse diagram shown in Fig.~\ref{fig:Hasse_A2_20}.
The $qq$-character \eqref{eq:A2_T_20} contains 9 monomials, corresponding to the tensor product of two three-dimensional fundamental modules of $w=(1,0)$.

\begin{figure}[t]
    \begin{center}
    \begin{tikzcd}[column sep=.5cm]
    && \sY_{1,x_1} \sY_{1,x_2} \arrow[ld,"{1,x_1}"'] \arrow[rd,"{1,x_2}"] && \\
    & \displaystyle \frac{\sY_{1,x_2} \sY_{2,x_1} }{ \sY_{1,x_1 q} } \arrow[ld,"{2,x_1}"'] \arrow[rd,"{1,x_2}"] && \displaystyle \frac{\sY_{1,x_1} \sY_{2,x_2}}{\sY_{1,x_2 q}} \arrow[ld,"{1,x_1}"'] \arrow[rd,"{2,x_2}"] & \\
    \displaystyle \frac{ \sY_{1,x_2} }{ \sY_{2,x_1 q} } \arrow[rd,"{1,x_2}"] && \displaystyle \frac{\sY_{2,x_1} \sY_{2,x_2}}{\sY_{1,x_1 q} \sY_{1,x_2 q}} \arrow[ld,"{2,x_1}"'] \arrow[rd,"{2,x_2}"] && \displaystyle \frac{\sY_{1,x_1}}{\sY_{2,x_2 q}} \arrow[ld,"{1,x_1}"'] \\
    & \displaystyle \frac{ \sY_{2,x_2} }{ \sY_{1,x_2 q} \sY_{2,x_1 q} } \arrow[rd,"{2,x_2}"] && \displaystyle \frac{\sY_{2,x_1}}{\sY_{1, x_1 q} \sY_{2,x_2 q}} \arrow[ld,"{2,x_1}"'] & \\
    && \displaystyle \frac{1}{\sY_{2,x_1 q} \sY_{2,x_2 q}} &&
    \end{tikzcd}
    \end{center}
    \caption{Hasse diagram for the iWeyl reflection flow of the weight $w=(2,0)$ for $A_2$ quiver.}
    \label{fig:Hasse_A2_20}
\end{figure}
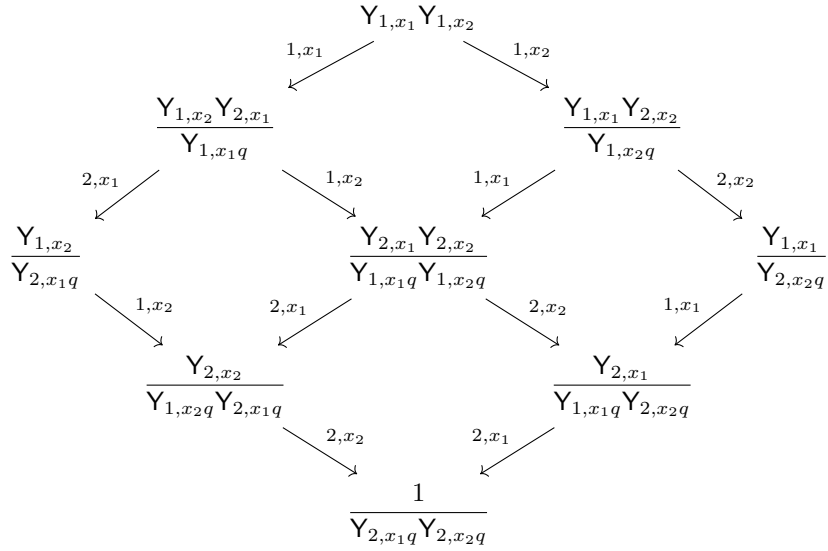

We specialize the spectral parameters $x = (x,x q_1)$, which corresponds to the weight two KR module $\mathbf{W}_{2_1,x}^{(1)}$.
Then, we have the following.
\begin{proposition}
The $qq$-character of six-dimensional module of weight two of the algebra $U_q(\widehat{\mathfrak{sl}}_3)$ is given by
\begin{align}
    \sT[\mathbf{W}_{2_1,x}^{(1)}] & = \sY_{1,x} \sY_{1,x;1,0} + \frac{\sY_{2,x} \sY_{2,x;1,0}}{\sY_{1,x;1,1} \sY_{1,x;2,1}} + \frac{1}{\sY_{2,x;1,1} \sY_{2,x;2,1}}
    \nonumber \\ & \qquad 
    + \scS(q_1^{-1}) \qty[ \frac{\sY_{1,x} \sY_{2,x;1,0}}{\sY_{1,x;2,1}} + \frac{\sY_{1,x}}{\sY_{2,x;2,1}} + \frac{\sY_{2,x}}{\sY_{1,x;1,1} \sY_{2,x;2,1}} ]
    \, .
\end{align}    
\end{proposition}
On the other hand, the eliminated factors in the $qq$-character, corresponding to the three-dimensional module, contribute to the residue of the pole at $x_1 = x_2 q$,%
\footnote{%
The residue factor plays a crucial role in the so-called quadratic relations in the context of the deformed W-algebra.
For the $A$-type quiver, in particular, the quadratic relations~\cite{Odake:2001ad,Kojima:2020vtc} are interpreted as a double quantization of Pieri's formula for the fundamental modules.
} 
\begin{align}
    \operatorname*{Res}_{x_1=x_2 q} \sT_{(2,0),x} = - \frac{(1-q_1)(1-q_2)}{1-q} \sT_{(0,1),x_2} \, .
\end{align}
This is consistent with the decomposition of the tensor product into the irreducible modules, $3 \otimes 3 = 6 \oplus 3$.
We also remark that $\operatorname*{Res}_{x_1=x_2} \sT_{(2,0),x} = 0$.

\subsubsection*{Classical limit}
Recalling the behavior of the $\scS$-factor in the classical limit \eqref{eq:S_lim}, this $qq$-character is further reduced in the classical limit as follows,
\begin{subequations}
\begin{align}
    \sT[\mathbf{W}_{2_1,x}^{(1)}]
    \ \xrightarrow{q_1 \to 1} \
    &
    \sY_{1,x}^2 + \qty(\frac{\sY_{2,x}}{\sY_{1,x;0,1}})^2 + \sY_{2,x;0,1}^{-2} + 2 \qty(\frac{\sY_{1,x} \sY_{2,x}}{\sY_{1,x;0,1}} + \frac{ \sY_{1,x} }{ \sY_{2,x;0,1}} + \frac{\sY_{2,x}}{\sY_{1,x;0,1} \sY_{2,x;0,1}} ) 
    \nonumber \\
    & = \qty(\sY_{1,x} + \frac{\sY_{2,x}}{\sY_{1,x;0,1}} + \frac{1}{\sY_{2,x;0,1}} )^2 \, ,
    \\[.5em]
    \xrightarrow{q_2\to 1} \
    & \sY_{1,x} \sY_{1,x;1,0} + \frac{\sY_{2,x} \sY_{2,x;1,0}}{\sY_{1,x;1,0} \sY_{1,x;2,0}} + \frac{1}{\sY_{2,x;1,0} \sY_{2,x;2,0}}
    + \frac{\sY_{1,x} \sY_{2,x;1,0}}{\sY_{1,x;2,0}} + \frac{\sY_{1,x}}{\sY_{2,x;2,0}} + \frac{\sY_{2,x}}{\sY_{1,x;1,0} \sY_{2,x;2,0}} \, .
\end{align}
\end{subequations}
In the limit $q_1 \to 1$, we obtain the degenerated tensor product of $w = (1,0)$, while we obtain the $q$-character associated with the KR module $\mathbf{W}_{2_1,x}^{(1)}$ in the limit $q_2 \to 1$.

\subsubsection{$w = (0,2)$}

The $qq$-character of weight $w = (0,2)$ is similarly calculated as before.
For the spectral parameters $x = (x_1,x_2)$, we obtain
\begin{align}
    \sT_{(0,2),x} & =
    \sY_{2,x_1} \sY_{2,x_2} + \frac{\sY_{1,x_1q} \sY_{1,x_2q}}{\sY_{2,x_1q} \sY_{2,x_2q}} + \frac{1}{\sY_{1,x_1 q^2} \sY_{1,x_2 q^2}}
    \nonumber \\ & \qquad
    + \scS\qty( \frac{x_2}{x_1} ) \qty[
    \frac{\sY_{1,x_1q} \sY_{2,x_2}}{\sY_{2,x_1 q}} + \frac{\sY_{2,x_2}}{\sY_{1,x_1 q^2}} + \frac{\sY_{1,x_2q}}{\sY_{1,x_1q^2} \sY_{2,x_2q}}
    ]
    \nonumber \\ & \qquad \qquad
    + \scS\qty( \frac{x_1}{x_2} ) \qty[
    \frac{\sY_{1,x_2q} \sY_{2,x_1}}{\sY_{2,x_2 q}} + \frac{\sY_{2,x_1}}{\sY_{1,x_2 q^2}} + \frac{\sY_{1,x_1q}}{\sY_{1,x_2q^2} \sY_{2,x_1q}}
    ] \, .
\end{align}
The iWeyl reflection flow is described by the Hasse diagram in Fig.~\ref{fig:Hasse_A2_02}.

\begin{figure}[t]
    \begin{center}
    \begin{tikzcd}[column sep=.5cm]
    && \sY_{2,x_1} \sY_{2,x_2} \arrow[ld,"{2,x_1}"'] \arrow[rd,"{2,x_2}"] && \\
    & \displaystyle \frac{\sY_{1,x_1 q} \sY_{2,x_2} }{ \sY_{2,x_1 q} } \arrow[ld,"{1,x_1q}"'] \arrow[rd,"{2,x_2}"] && \displaystyle \frac{\sY_{1,x_2q} \sY_{2,x_1}}{\sY_{2,x_2 q}} \arrow[ld,"{2,x_1}"'] \arrow[rd,"{1,x_2q}"] & \\
    \displaystyle \frac{ \sY_{2,x_2} }{ \sY_{1,x_1 q^2} } \arrow[rd,"{2,x_2}"] && \displaystyle \frac{\sY_{1,x_1q} \sY_{1,x_2q}}{\sY_{2,x_1 q} \sY_{2,x_2 q}} \arrow[ld,"{1,x_1q}"'] \arrow[rd,"{1,x_2q}"] && \displaystyle \frac{\sY_{2,x_1}}{\sY_{1,x_2 q^2}} \arrow[ld,"{2,x_1}"'] \\
    & \displaystyle \frac{ \sY_{1,x_2 q} }{ \sY_{1,x_1 q^2} \sY_{2,x_2 q} } \arrow[rd,"{1,x_2 q}"] && \displaystyle \frac{\sY_{1,x_1q}}{\sY_{1, x_2 q^2} \sY_{2,x_1 q}} \arrow[ld,"{1,x_1q}"'] & \\
    && \displaystyle \frac{1}{\sY_{1,x_1 q^2} \sY_{1,x_2 q^2}} &&
    \end{tikzcd}
    \end{center}
    \caption{Hasse diagram for the iWeyl reflection flow of the weight $w=(0,2)$ for $A_2$ quiver.}
    \label{fig:Hasse_A2_02}
\end{figure}
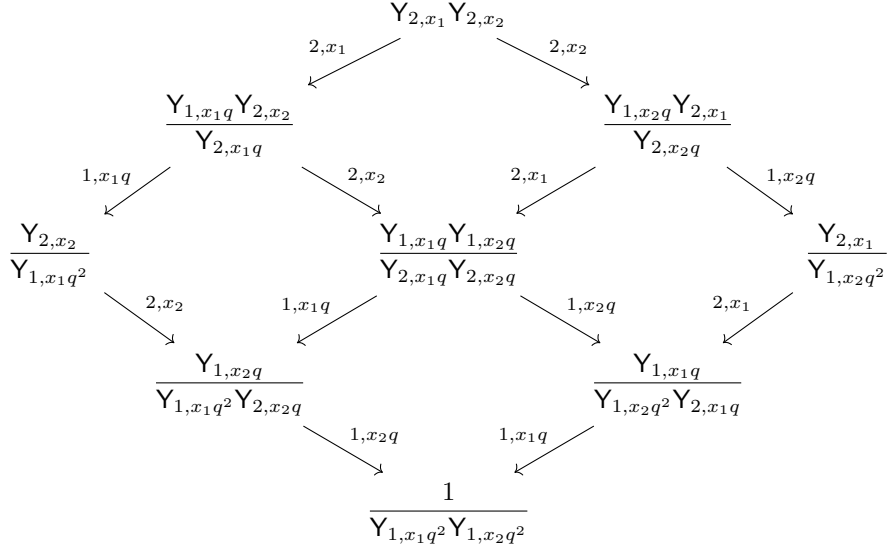

Specializing the parameters as $x = (x,xq_1)$, we obtain the $qq$-character of the KR module $\mathbf{W}_{2_1,x}^{(2)}$ as follows.
\begin{proposition}
\begin{align}
    \sT[\mathbf{W}_{2_1,x}^{(2)}] & =
    \sY_{2,x} \sY_{2,x;1,0} + \frac{\sY_{1,x;1,1} \sY_{1,x;2,1}}{\sY_{2,x;1,1} \sY_{2,x;2,1}} + \frac{1}{\sY_{1,x;2,2} \sY_{1,x;3,2}}
    \nonumber \\ & \qquad
    + \scS(q_1^{-1}) \qty[
    \frac{\sY_{1,x;2,1} \sY_{2,x}}{\sY_{2,x;2,1}} + \frac{\sY_{2,x}}{\sY_{1,x;3,2}} + \frac{\sY_{1,x;1,1}}{\sY_{1,x;3,2} \sY_{2,x;1,1}}
    ] \, .
\end{align}    
\end{proposition}
We may obtain a similar expression from another specializations $x = (x,xq_2)$.

\subsubsection*{Classical limit}
The classical limit of this $qq$-character is given by
\begin{subequations}
\begin{align}
    \sT[\mathbf{W}_{2_1,x}^{(2)}]  
    \xrightarrow{q_1 \to 1} \ &
    \sY_{2,x}^2 + \qty(\frac{\sY_{1,x;0,1}}{\sY_{2,x;0,1}})^2 + \frac{1}{\sY_{1,x;0,2}^2}
    + 2 \qty( \frac{\sY_{1,x;0,1} \sY_{2,x}}{\sY_{2,x;0,1}} + \frac{\sY_{2,x}}{\sY_{1,x;0,2}} + \frac{\sY_{1,x;0,1}}{\sY_{1,x;0,2} \sY_{2,x;0,1}} )
    \nonumber \\ & 
    = \qty(\sY_{2,x} + \frac{\sY_{1,x;0,1}}{\sY_{2,x;0,1}} + \frac{1}{\sY_{1,x;0,2}})^2 \, ,
    \\ 
    \xrightarrow{q_2 \to 1} \ &
    \sY_{2,x} \sY_{2,x;1,0} + \frac{\sY_{1,x;1,0} \sY_{1,x;2,0}}{\sY_{2,x;1,0} \sY_{2,x;2,0}} + \frac{1}{\sY_{1,x;2,0} \sY_{1,x;3,0}}
    + 
    \frac{\sY_{1,x;2,0} \sY_{2,x}}{\sY_{2,x;2,0}} + \frac{\sY_{2,x}}{\sY_{1,x;3,0}} + \frac{\sY_{1,x;1,0}}{\sY_{1,x;3,0} \sY_{2,x;1,0}}
    \, .
\end{align}
\end{subequations}
We obtain the degenerated tensor product of $w = (0,1)$ in the limit $q_1 \to 1$, while we obtain the $q$-character associated with the KR module $\mathbf{W}_{2_1,x}^{(2)}$ in the limit $q_2 \to 1$.

\subsubsection{$w = (1,1)$}\label{sec:A2_w11}

We consider the $qq$-character of weight $w=(1,1)$, 
\begin{align}
    \sT_{(1,1),x} & = \sY_{1,x_1} \sY_{2,x_2} + \frac{\sY_{2,x_1} \sY_{2,x_2}}{\sY_{1,x_1 q}} + \frac{\sY_{1,x_1} \sY_{1,x_2 q}}{\sY_{2,x_2 q}} + \frac{\sY_{1,x_2 q}}{\sY_{2,x_1 q} \sY_{2,x_2 q}} + \frac{\sY_{2,x_1 }}{\sY_{1,x_1 q} \sY_{1,x_2 q^2}} + \frac{1}{\sY_{1,x_2 q^2} \sY_{2,x_1 q}}
    \nonumber \\
    & \qquad 
    + \scS\qty(\frac{x_2}{x_1})  \frac{\sY_{2,x_2}}{\sY_{2,x_1 q}}
    + \scS\qty(\frac{x_1}{x_2}) \frac{\sY_{1,x_2 q} \sY_{2,x_1}}{\sY_{1,x_1 q} \sY_{2,x_2 q}}
    + \scS\qty(\frac{x_1}{x_2 q} ) \frac{\sY_{1,x_1}}{\sY_{1,x_2 q^2}}
    \, ,
    \label{eq:A2_T_11}
\end{align}
which is described by the Hasse diagram in Fig.~\ref{fig:Hasse_A2_11}.
The $qq$-character \eqref{eq:A2_T_11} contains 9 monomials, interpreted as the tensor product of two three-dimensional fundamental modules, $w=(1,0)$ and $w = (0,1)$.
In this case, we have the decomposition, $3 \otimes 3 = 8 \oplus 1$.
The residues of this $qq$-character are given by
\begin{align}
    \operatorname*{Res}_{x_1 = x_1} \sT_{(1,1),x} = \operatorname*{Res}_{x_1 = x_2 q} \sT_{(1,1),x} = 0 \, , \qquad 
    \operatorname*{Res}_{x_2 = x_1 q} \sT_{(1,1),x} = \operatorname*{Res}_{x_1 = x_2 q^2} \sT_{(1,1),x} = - \frac{(1-q_1)(1-q_2)}{1-q} \, .
\end{align}
The latter two residues are the non-highest dominant monomial ``1'', interpreted as the contribution of the one-dimensional module.

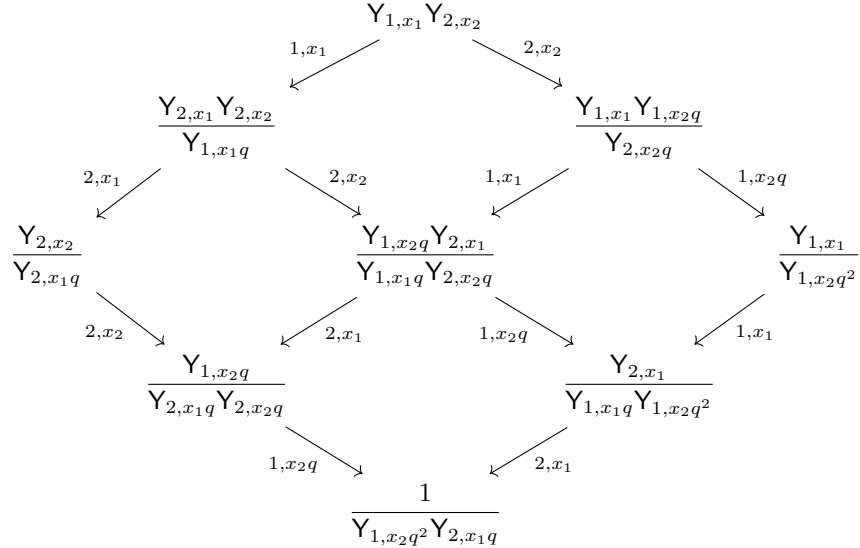
\begin{figure}[t]
    \centering
   \begin{tikzcd}[column sep=.5cm]
    && \sY_{1,x_1} \sY_{2,x_2} \arrow[ld,"{1,x_1}"'] \arrow[rd,"{2,x_2}"] && \\
    & \displaystyle \frac{\sY_{2,x_1} \sY_{2,x_2}}{\sY_{1,x_1 q}} \arrow[ld,"{2,x_1}"'] \arrow[rd,"{2,x_2}"] && \displaystyle \frac{\sY_{1,x_1} \sY_{1,x_2 q}}{\sY_{2,x_2 q}} \arrow[rd,"{1,x_2 q}"] \arrow[ld,"{1,x_1}"'] & \\
    \displaystyle \frac{\sY_{2,x_2}}{\sY_{2,x_1 q}} \arrow[rd,"{2,x_2}"'] && \displaystyle \frac{\sY_{1,x_2 q} \sY_{2,x_1}}{\sY_{1,x_1 q} \sY_{2,x_2 q}} \arrow[ld,"{2,x_1}"] \arrow[rd,"{1,x_2q}"'] && \displaystyle \frac{\sY_{1,x_1}}{\sY_{1,x_2 q^2}} \arrow[ld,"{1,x_1}"] \\
    & \displaystyle \frac{\sY_{1,x_2 q}}{\sY_{2,x_1 q} \sY_{2,x_2 q}} \arrow[rd,"{1,x_2 q}"'] && \displaystyle \frac{\sY_{2,x_1 }}{\sY_{1,x_1 q} \sY_{1,x_2 q^2}} \arrow[ld,"{2,x_1}"] & \\
    && \displaystyle \frac{1}{\sY_{1,x_2 q^2} \sY_{2,x_1 q}} &&
    \end{tikzcd}
    \caption{Hasse diagram for the iWeyl reflection flow of the weight $w= (1,1)$ for $A_2$ quiver.}
    \label{fig:Hasse_A2_11}
\end{figure}

In order to obtain the reduced $qq$-character of the weight $w = (1,1)$, we have three possible specializations: (a) $x = (x,xq_1)$, (b) $x = (xq_1,x)$, and (c) $x = (x q_1^2 q_2,x)$.
\begin{align}
    \begin{tabular*}{.85\textwidth}{@{\extracolsep{\fill}}rcccc}\toprule
    & $(x_1,x_2)$ & $\scS\qty(x_2/x_1)$ & $\scS\qty(x_1/x_2)$ & $\scS\qty(x_1/x_2q)$ \\\midrule
    (a) & $(x,xq_1)$ & $0$ & $\scS(q_1^{-1})$ & $\scS(q_1^{-2} q_2^{-1})$ \\
    (b) & $(xq_1,x)$ & $\scS(q_1^{-1})$ & $0$ & $\scS(q_2^{-1})$ \\
    (c) & $(x q_1^2q_2,x)$ & $\scS(q_1^{-2} q_2^{-1})$ & $\scS(q_1^{-1})$ & $0$ \\\bottomrule
    \end{tabular*}
\end{align}
\begin{proposition}
We have the following specializations of the $qq$-character, corresponding to the eight-dimensional module of $A_2$ quiver,
\begin{subequations}\label{eq:A2_T_11_red}
\begin{align}
    \sT_{(1,1),x}
    \ \xrightarrow{(\text{a})} \ &
    \sY_{1,x} \sY_{2,x;1,0} + \frac{\sY_{2,x} \sY_{2,x;1,0}}{\sY_{1,x;1,1}} + \frac{\sY_{1,x} \sY_{1,x;2,1}}{\sY_{2,x;2,1}} + \frac{\sY_{1,x;2,1}}{\sY_{2,x;1,1} \sY_{2,x;2,1}} + \frac{\sY_{2,x}}{\sY_{1,x;1,1} \sY_{1,x;3,2}} + \frac{1}{\sY_{1,x;3,2} \sY_{2,x;1,1}}
    \nonumber \\
    & \qquad 
    + \scS\qty(q_1^{-1}) \frac{\sY_{1,x;2,1} \sY_{2,x}}{\sY_{1,x;1,1} \sY_{2,x;2,1}}
    + \scS\qty( q_1^{-2} q_2^{-1} ) \frac{\sY_{1,x}}{\sY_{1,x;3,2}}
    \, , \\[1em]
    \xrightarrow{(\text{b})} \ &
    \sY_{1,x;1,0} \sY_{2,x} + \frac{\sY_{2,x} \sY_{2,x;1,0}}{\sY_{1,x;2,1}} + \frac{\sY_{1,x;1,0} \sY_{1,x;1,1}}{\sY_{2,x;1,1}} + \frac{\sY_{1,x;1,1}}{\sY_{2,x;1,1} \sY_{2,x;2,1}} + \frac{\sY_{2,x;1,0}}{\sY_{1,x;2,1} \sY_{1,x;2,2}} + \frac{1}{\sY_{1,x;2,2} \sY_{2,x;2,1}}
    \nonumber \\
    & \qquad 
    + \scS\qty(q_1^{-1})  \frac{\sY_{2,x}}{\sY_{2,x;2,1}}
    + \scS\qty( q_2^{-1} ) \frac{\sY_{1,x;1,0}}{\sY_{1,x;2,2}}
    \, , \\[1em]
    \xrightarrow{(\text{c})} \ &
    \sY_{1,x;2,1} \sY_{2,x} + \frac{\sY_{2,x;2,1} \sY_{2,x}}{\sY_{1,x;3,2}} + \frac{\sY_{1,x;1,1} \sY_{1,x;2,1}}{\sY_{2,x;1,1}} + \frac{\sY_{1,x;1,1}}{\sY_{2,x;1,1} \sY_{2,x;3,2}} + \frac{\sY_{2,x;2,1}}{\sY_{1,x;2,2} \sY_{1,x;3,2}} + \frac{1}{\sY_{1,x;2,2} \sY_{2,x;3,2}}
    \nonumber \\
    & \qquad 
    + \scS\qty(q_1^{-2}q_2^{-1})  \frac{\sY_{2,x}}{\sY_{2,x;3,2}}
    + \scS\qty(q_1^{-1}) \frac{\sY_{1,x;1,1} \sY_{2,x;2,1}}{\sY_{1,x;3,2} \sY_{2,x;1,1}}
    \, .
\end{align}
\end{subequations}
\end{proposition}

\subsubsection*{Classical limit}

In this case, we can discuss several classical limits of the $qq$-character.
From the classical limit of the $\scS$-factor \eqref{eq:S_lim} together with $\scS\qty(q_1^{-2}q_2^{-1}) \xrightarrow{q_1, q_2 \to 1} 1$, we have the following expressions:
\begin{itemize}
    \item[(a)] $x = (x,xq_1)$ 
    \begin{subequations}
    \begin{align}
    (q_1 \to 1) : \quad & 
    \sY_{1,x} \sY_{2,x} + \frac{\sY_{2,x}^2}{\sY_{1,x;0,1}} + \frac{\sY_{1,x} \sY_{1,x;0,1}}{\sY_{2,x;0,1}} + \frac{\sY_{1,x;0,1}}{\sY_{2,x;0,1}^2} + \frac{\sY_{2,x}}{\sY_{1,x;0,1} \sY_{1,x;0,2}} + \frac{1}{\sY_{1,x;0,2} \sY_{2,x;0,1}}
    \nonumber \\ &  \quad 
    + 2\frac{\sY_{2,x}}{\sY_{2,x;0,1}}
    + \frac{\sY_{1,x}}{\sY_{1,x;0,2}}
    = \qty( \sY_{1,x} + \frac{\sY_{2,x}}{\sY_{1,x;0,1}} + \frac{1}{\sY_{2,x;0,1}} ) \qty(\sY_{2,x} + \frac{\sY_{1,x;0,1}}{\sY_{2,x;0,1}} + \frac{1}{\sY_{1,x;0,2}})
    \, , \\[1em]
    (q_2 \to 1) : \quad & 
    \sY_{1,x} \sY_{2,x;1,0} + \frac{\sY_{2,x} \sY_{2,x;1,0}}{\sY_{1,x;1,0}} + \frac{\sY_{1,x} \sY_{1,x;2,0}}{\sY_{2,x;2,0}} + \frac{\sY_{1,x;2,0}}{\sY_{2,x;1,0} \sY_{2,x;2,0}} + \frac{\sY_{2,x}}{\sY_{1,x;1,0} \sY_{1,x;3,0}} + \frac{1}{\sY_{1,x;3,0} \sY_{2,x;1,0}}
    \nonumber \\
    & \quad 
    + \frac{\sY_{1,x;2,0} \sY_{2,x}}{\sY_{1,x;1,0} \sY_{2,x;2,0}}
    + \frac{\sY_{1,x}}{\sY_{1,x;3,0}}
    \, .
    \label{eq:A2_T_11_lima2}
    \end{align}
    \end{subequations}
    \item[(b)] $x = (xq_1,x)$ 
    \begin{subequations}
    \begin{align}
    (q_1 \to 1) : \quad & 
    \sY_{1,x} \sY_{2,x} + \frac{\sY_{2,x}^2}{\sY_{1,x;0,1}} + \frac{\sY_{1,x} \sY_{1,x;0,1}}{\sY_{2,x;0,1}} + \frac{\sY_{1,x;0,1}}{\sY_{2,x;0,1}^2} + \frac{\sY_{2,x}}{\sY_{1,x;0,1} \sY_{1,x;0,2}} + \frac{1}{\sY_{1,x;0,2} \sY_{2,x;0,1}}
    \nonumber \\
    & \quad 
    + 2 \frac{\sY_{2,x}}{\sY_{2,x;0,1}}
    + \frac{\sY_{1,x}}{\sY_{1,x;0,2}}
    = \qty( \sY_{1,x} + \frac{\sY_{2,x}}{\sY_{1,x;0,1}} + \frac{1}{\sY_{2,x;0,1}} ) \qty(\sY_{2,x} + \frac{\sY_{1,x;0,1}}{\sY_{2,x;0,1}} + \frac{1}{\sY_{1,x;0,2}})
    \, , \\[1em]
    (q_2 \to 1) : \quad & 
    \sY_{1,x;1,0} \sY_{2,x} + \frac{\sY_{2,x} \sY_{2,x;1,0}}{\sY_{1,x;2,0}} + \frac{\sY_{1,x;1,0}^2}{\sY_{2,x;1,0}} + \frac{\sY_{1,x;1,0}}{\sY_{2,x;1,0} \sY_{2,x;2,0}} + \frac{\sY_{2,x;1,0}}{\sY_{1,x;2,0}^2} + \frac{1}{\sY_{1,x;2,0} \sY_{2,x;2,0}}
    \nonumber \\
    & \quad 
    + \frac{\sY_{2,x}}{\sY_{2,x;2,0}}
    + 2 \frac{\sY_{1,x;1,0}}{\sY_{1,x;2,0}}
    = \qty( \sY_{1,x;1,0} + \frac{\sY_{2,x;1,0}}{\sY_{1,x;2,0}} + \frac{1}{\sY_{2,x;2,0}} ) \qty(\sY_{2,x} + \frac{\sY_{1,x;1,0}}{\sY_{2,x;1,0}} + \frac{1}{\sY_{1,x;2,0}})
    \, .
    \end{align}
    \end{subequations}
    \item[(c)] $x = (xq_1^2q_2,x)$
    \begin{subequations}
    \begin{align}
        (q_1 \to 1) : \quad & 
        \sY_{1,x;0,1} \sY_{2,x} + \frac{\sY_{2,x;0,1} \sY_{2,x}}{\sY_{1,x;0,2}} + \frac{\sY_{1,x;0,1}^2}{\sY_{2,x;0,1}} + \frac{\sY_{1,x;0,1}}{\sY_{2,x;0,1} \sY_{2,x;0,2}} + \frac{\sY_{2,x;0,1}}{\sY_{1,x;0,2}^2} + \frac{1}{\sY_{1,x;0,2} \sY_{2,x;0,2}}
    \nonumber \\
    & \qquad 
    + \frac{\sY_{2,x}}{\sY_{2,x;0,2}}
    + 2 \frac{\sY_{1,x;0,1}}{\sY_{1,x;0,2}}
    = \qty( \sY_{1,x;0,1} + \frac{\sY_{2,x;0,1}}{\sY_{1,x;0,2}} + \frac{1}{\sY_{2,x;0,2}} ) \qty(\sY_{2,x} + \frac{\sY_{1,x;0,1}}{\sY_{2,x;0,1}} + \frac{1}{\sY_{1,x;0,2}})
    \, , \\
    (q_2 \to 1) : \quad &
    \sY_{1,x;2,0} \sY_{2,x} + \frac{\sY_{2,x;2,0} \sY_{2,x}}{\sY_{1,x;3,0}} + \frac{\sY_{1,x;1,0} \sY_{1,x;2,0}}{\sY_{2,x;1,0}} + \frac{\sY_{1,x;1,0}}{\sY_{2,x;1,0} \sY_{2,x;3,0}} + \frac{\sY_{2,x;2,0}}{\sY_{1,x;2,0} \sY_{1,x;3,0}} + \frac{1}{\sY_{1,x;2,0} \sY_{2,x;3,0}}
    \nonumber \\
    & \qquad 
    + \frac{\sY_{2,x}}{\sY_{2,x;3,0}}
    + \frac{\sY_{1,x;1,0} \sY_{2,x;2,0}}{\sY_{1,x;3,0} \sY_{2,x;1,0}} 
    \, .
    \label{eq:A2_T_11_limc2}
    \end{align}
    \end{subequations}
\end{itemize}
We observe that \eqref{eq:A2_T_11_lima2} and \eqref{eq:A2_T_11_limc2} provide the $q$-character of the adjoint representation of $U_q(\widehat{\mathfrak{sl}}_3)$ (See, for example, \cite[\S3.1]{Nakai:2008pz}).
The remaining cases are the degenerate ones, where we obtain the product of two three-dimensional fundamental $q$-characters.

\section{$B_2/C_2$ quiver}\label{sec:BC2}

We consider $B_2/C_2$ quiver, which is the minimal rank non-simply-laced example.
In this case, we have the decoration (root) parameters $d = (d_1,d_2) = (2,1)$, and the iWeyl reflections for the $\sY$-variables are given by
\begin{align}
    \operatorname{iWeyl}: \
    \qty( \mathsf{Y}_{1,x} , \mathsf{Y}_{2,x} )
    \ & \longmapsto \
    \qty( 
    \frac{\mathsf{Y}_{2,x} \mathsf{Y}_{2,x q_1}}{\mathsf{Y}_{1,x q_1^2 q_2 }} , \frac{\mathsf{Y}_{1,x q_1 q_2}}{\mathsf{Y}_{2,x q_1 q_2}}
    )
    \, .
\end{align}
Then, the fundamental $qq$-characters are given as follows,
\begin{subequations}\label{eq:BC2_qq_ch}
\begin{align}
    \mathsf{T}_{(1,0),x} & = \mathsf{Y}_{1,x} + \frac{\mathsf{Y}_{2,x} \mathsf{Y}_{2,x;1,0}}{\mathsf{Y}_{1,x;2,1}} + \mathscr{S}(q_1^{-1}) \frac{\mathsf{Y}_{2,x}}{\mathsf{Y}_{2,x;2,1}} + \frac{\mathsf{Y}_{1,x;1,1}}{\mathsf{Y}_{2,x;1,1} \mathsf{Y}_{2,x;2,1}} + \frac{1}{\mathsf{Y}_{1,x;3,2}}
    \, , \label{eq:BC2_qq_ch_10} \\
    \mathsf{T}_{(0,1),x} & = \mathsf{Y}_{2,x} + \frac{\mathsf{Y}_{1,x;1,1}}{\mathsf{Y}_{2,x;1,1}} + \frac{\mathsf{Y}_{2,x;2,1}}{\mathsf{Y}_{1,x;3,2}} + \frac{1}{\mathsf{Y}_{2,x;3,2}}
    \, .
\end{align}
\end{subequations}
They correspond to the five and four-dimensional modules associated with the quantum affine algebra $U_q(\widehat{\mathfrak{so}}_5) \cong U_q(\widehat{\mathfrak{sp}}_2)$. 
In the following, we construct the weight two $qq$-characters based on the iWeyl reflections presented here, and examine the irreducibility of the corresponding modules.

\subsection{$w = (2,0)$}

We consider the $qq$-character of weight $w=(2,0)$,
\begin{align}
    & \sT_{(2,0),x} = \sY_{1,x_1} \sY_{1,x_2} + \frac{\sY_{2,x_1} \sY_{2,x_1;1,0} \sY_{2,x_2} \sY_{2,x_2;1,0}}{\sY_{1,x_1;2,1} \sY_{1,x_2;2,1}} + \frac{\sY_{1,x_1;1,1} \sY_{1,x_2;1,1}}{\sY_{2,x_1;1,1} \sY_{2,x_1;2,1} \sY_{2,x_2;1,1} \sY_{2,x_2;2,1}} + \frac{1}{\sY_{1,x_1;3,2} \sY_{1,x_2;3,2}}
    \nonumber \\ &
    + \scS_2\qty(\frac{x_2}{x_1}) \Bigg[
    \frac{\sY_{1,x_2} \sY_{2,x_1} \sY_{2,x_1;1,0}}{\sY_{1,x_1;2,1}} 
    + \frac{\sY_{1,x_1;1,1} \sY_{1,x_2}}{\sY_{2,x_1;1,1} \sY_{2,x_1;2,1}}
    + \frac{\sY_{2,x_2} \sY_{2,x_2;1,0}}{\sY_{1,x_1;3,2} \sY_{1,x_2;2,1}}
    + \frac{\sY_{1,x_2;1,1}}{\sY_{1,x_2;3,2} \sY_{2,x_2;1,1} \sY_{2,x_2;2,1}}
    \nonumber \\ & \hspace{5.5em}
    + \scS(q_1^{-1}) \qty( 
    \frac{\sY_{1,x_2} \sY_{2,x_1}}{\sY_{2,x_1;2,1}}
    + \frac{\sY_{2,x_1} \sY_{2,x_2} \sY_{2,x_2;1,0}}{\sY_{1,x_2;2,1} \sY_{2,x_1;2,1}}
    + \frac{\sY_{1,x_1;1,1} \sY_{2,x_2}}{\sY_{2,x_1;1,1} \sY_{2,x_1;2,1} \sY_{2,x_2;2,1}}
    + \frac{\sY_{2,x_2}}{\sY_{1,x_1;3,2} \sY_{2,x_2;2,1}}
    )
    \nonumber \\ & \hspace{5.5em}
    + \scS_2\qty(\frac{x_2}{x_1}q_1^{-1}q_2^{-1}) \frac{\sY_{1,x_2}}{\sY_{1,x_1;3,2}}
    + \scS_2\qty(\frac{x_2}{x_1}q_1) \frac{\sY_{1,x_1;1,1} \sY_{2,x_2} \sY_{2,x_2;1,0}}{\sY_{1,x_2;1,1} \sY_{2,x_1;1,1} \sY_{2,x_1;2,1}}
    \Bigg]
    \nonumber \\ &
    + \scS_2\qty(\frac{x_1}{x_2}) \Bigg[
    \frac{\sY_{1,x_1} \sY_{2,x_2} \sY_{2,x_2;1,0}}{\sY_{1,x_2;2,1}} 
    + \frac{\sY_{1,x_2;1,1} \sY_{1,x_1}}{\sY_{2,x_2;1,1} \sY_{2,x_2;2,1}}
    + \frac{\sY_{2,x_1} \sY_{2,x_1;1,0}}{\sY_{1,x_2;3,2} \sY_{1,x_1;2,1}}
    + \frac{\sY_{1,x_1;1,1}}{\sY_{1,x_1;3,2} \sY_{2,x_1;1,1} \sY_{2,x_1;2,1}}
    \nonumber \\ & \hspace{5.5em}
    + \scS(q_1^{-1}) \qty( 
    \frac{\sY_{1,x_1} \sY_{2,x_2}}{\sY_{2,x_2;2,1}}
    + \frac{\sY_{2,x_2} \sY_{2,x_1} \sY_{2,x_1;1,0}}{\sY_{1,x_1;2,1} \sY_{2,x_2;2,1}}
    + \frac{\sY_{1,x_2;1,1} \sY_{2,x_1}}{\sY_{2,x_2;1,1} \sY_{2,x_2;2,1} \sY_{2,x_1;2,1}}
    + \frac{\sY_{2,x_1}}{\sY_{1,x_2;3,2} \sY_{2,x_1;2,1}}
    )
    \nonumber \\ & \hspace{5.5em}
    + \scS_2\qty(\frac{x_1}{x_2}q_1^{-1}q_2^{-1}) \frac{\sY_{1,x_1}}{\sY_{1,x_2;3,2}}
    + \scS_2\qty(\frac{x_1}{x_2}q_1) \frac{\sY_{1,x_2;1,1} \sY_{2,x_1} \sY_{2,x_1;1,0}}{\sY_{1,x_1;1,1} \sY_{2,x_2;1,1} \sY_{2,x_2;2,1}}
    \Bigg]    
    \nonumber \\ &
    + \scS(q_1^{-1})^2 \scS\qty(\frac{x_2}{x_1}q_1^{-1}) \scS\qty(\frac{x_1}{x_2}q_1^{-1}) \frac{\sY_{2,x_1} \sY_{2,x_2}}{\sY_{2,x_1;2,1} \sY_{2,x_2;2,1}} 
    \, .
\end{align}
The corresponding Hasse diagram is shown in Fig.~\ref{fig:Hasse_BC2_20}.
This $qq$-character contains 25 monomials, corresponding to the tensor product of two modules of $w = (1,0)$.
We shall see the decomposition of this module into the irreducible ones, $5 \otimes 5 = 14 \oplus 10 \oplus 1$.
The residues of this $qq$-character are given by
\begin{align}
    \operatorname*{Res}_{x_1 = x_2} \sT_{(2,0),x} = 0 \, , \qquad 
    \operatorname*{Res}_{x_1 = x_2 q_1^2 q_2} \sT_{(2,0),x} = - \frac{(1-q_1^2)(1-q_2)}{1-q_1^2 q_2} \sT[\mathbf{W}_{2_1,x_2}^{(2)}] \, ,
\end{align}
where $\sT[\mathbf{W}_{2_1,x}^{(2)}]$ is the $qq$-character of the KR module $\mathbf{W}_{2_1,x}^{(2)}$ given in Proposition~\ref{prop:BC2_KR02}.

Recalling that $\scS_2(z) = 0$ at $z = q_1^2, q_2$, we specialize the spectral parameters as $x = (x,xq_1^2)$.
This is consistent with that the KR module $\mathbf{W}_{2_1,x}^{(1)}$ should be considered with $q_1^2$-shift for the node $d_1 = 2$.
\begin{proposition}\label{prop:BC2_KR20}
The $qq$-character associated with the KR module $\mathbf{W}_{2_1,x}^{(1)}$ is given as follows,
\begin{align}
    & \sT[\mathbf{W}_{2_1,x}^{(1)}] = \sY_{1,x} \sY_{1,x;2,0} + \frac{\sY_{2,x} \sY_{2,x;1,0} \sY_{2,x;2,0} \sY_{2,x;3,0}}{\sY_{1,x;2,1} \sY_{1,x;4,1}} + \frac{\sY_{1,x;1,1} \sY_{1,x;3,1}}{\sY_{2,x;1,1} \sY_{2,x;2,1} \sY_{2,x;3,1} \sY_{2,x;4,1}} + \frac{1}{\sY_{1,x;3,2} \sY_{1,x;5,2}}
    \nonumber \\ & 
    + \scS_2\qty(q_1^{-2}) \Bigg[
    \frac{\sY_{1,x} \sY_{2,x;2,0} \sY_{2,x;3,0}}{\sY_{1,x;4,1}} 
    + \frac{\sY_{1,x;3,1} \sY_{1,x}}{\sY_{2,x;3,1} \sY_{2,x;4,1}}
    + \frac{\sY_{2,x} \sY_{2,x;1,0}}{\sY_{1,x;5,2} \sY_{1,x;2,1}}
    + \frac{\sY_{1,x;1,1}}{\sY_{1,x;3,2} \sY_{2,x;1,1} \sY_{2,x;2,1}}
    \nonumber \\ & \hspace{5.5em}
    + \scS(q_1^{-1}) \qty( 
    \frac{\sY_{1,x} \sY_{2,x;2,0}}{\sY_{2,x;4,1}}
    + \frac{\sY_{2,x;2,0} \sY_{2,x} \sY_{2,x;1,0}}{\sY_{1,x;2,1} \sY_{2,x;4,1}}
    + \frac{\sY_{1,x;3,1} \sY_{2,x}}{\sY_{2,x;3,1} \sY_{2,x;4,1} \sY_{2,x;2,1}}
    + \frac{\sY_{2,x}}{\sY_{1,x;5,2} \sY_{2,x;2,1}}
    )
    \nonumber \\ & \hspace{5.5em}
    + \scS_2\qty(q_1^{-3}q_2^{-1}) \frac{\sY_{1,x}}{\sY_{1,x;5,2}}
    + \scS_2\qty(q_1^{-1}) \frac{\sY_{1,x;3,1} \sY_{2,x} \sY_{2,x;1,0}}{\sY_{1,x;1,1} \sY_{2,x;3,1} \sY_{2,x;4,1}}
    \Bigg]
    \, ,
\end{align}
which corresponds to the 14-dimensional module of $B_2/C_2$ quiver for the weight $w = (2,0)$.
\end{proposition}

\begin{landscape}
\begin{figure}[t]
    \centering
\begin{tikzcd}[column sep=-.75cm,row sep=.7cm]
 &&&& \sY_{1,x_1} \sY_{1,x_2} \arrow[ld,"{1,x_1}"'] \arrow[rd,"{1,x_2}"] &&&& \\
 &&& \displaystyle \frac{\sY_{1,x_2} \sY_{2,x_1} \sY_{2,x_1;1,0}}{\sY_{1,x_1;2,1}} \arrow[ld,"{2,x_1;1,0}"'] \arrow[rd,"{1,x_2}"] && \displaystyle \frac{\sY_{1,x_1} \sY_{2,x_2} \sY_{2,x_2;1,0}}{\sY_{1,x_2;2,1}} \arrow[ld,"{1,x_1}"'] \arrow[rd,"{2,x_2;1,0}"] &&& \\
 && \displaystyle \frac{\sY_{1,x_2} \sY_{2,x_1}}{\sY_{2,x_1;2,1}} \arrow[ld,"{2,x_1}"'] \arrow[rd,"{1,x_2}"] && \displaystyle \frac{\sY_{2,x_1} \sY_{2,x_1;1,0} \sY_{2,x_2} \sY_{2,x_2;1,0}}{\sY_{1,x_1;2,1} \sY_{1,x_2;2,1}} \arrow[ld,"{2,x_1;1,0}"'] \arrow[rd,"{2,x_2;1,0}"] && \displaystyle \frac{\sY_{1,x_1} \sY_{2,x_2}}{\sY_{2,x_2;2,1}} \arrow[ld,"{1,x_1}"'] \arrow[rd,"{2,x_2}"] && \\
 & \displaystyle \frac{\sY_{1,x_1;1,1} \sY_{1,x_2}}{\sY_{2,x_1;1,1} \sY_{2,x_1;2,1}} \arrow[ld,"{1,x_1;1,1}"'] \arrow[rd,"{1,x_2}"] && \displaystyle \frac{\sY_{2,x_1} \sY_{2,x_2} \sY_{2,x_2;1,0}}{\sY_{1,x_2;2,1} \sY_{2,x_1;2,1}} \arrow[ld,"{2,x_1}"'] \arrow[rd,"{2,x_2;1,0}"] && \displaystyle \frac{\sY_{2,x_1} \sY_{2,x_1;1,0} \sY_{2,x_2}}{\sY_{1,x_1;2,1} \sY_{2,x_2;2,1}} \arrow[ld,"{2,x_1;1,0}"'] \arrow[rd,"{2,x_2}"] && \displaystyle \frac{\sY_{1,x_1} \sY_{1,x_2;1,1}}{\sY_{2,x_2;1,1} \sY_{2,x_2;2,1}} \arrow[ld,"{1,x_1}"'] \arrow[rd,"{1,x_2;1,1}"] & \\
 \displaystyle \frac{\sY_{1,x_2}}{\sY_{1,x_1;3,2}} \arrow[rd,"{1,x_2}"'] && \displaystyle \frac{\sY_{1,x_1;1,1} \sY_{2,x_2} \sY_{2,x_2;1,0}}{\sY_{1,x_2;2,1} \sY_{2,x_1;1,1} \sY_{2,x_1;2,1}} \arrow[ld,"{1,x_1;1,1}"] \arrow[rd,"{2,x_2;1,0}"'] && \displaystyle \frac{\sY_{2,x_1} \sY_{2,x_2}}{\sY_{2,x_1;2,1} \sY_{2,x_2;2,1}} \arrow[ld,"{2,x_1}"] \arrow[rd,"{2,x_2}"'] && \displaystyle \frac{\sY_{1,x_2;1,1} \sY_{2,x_1} \sY_{2,x_1;1,0}}{\sY_{1,x_1;2,1} \sY_{2,x_2;1,1} \sY_{2,x_2;2,1}} \arrow[ld,"{2,x_1;1,0}"] \arrow[rd,"{1,x_2;1,1}"'] && \displaystyle \frac{\sY_{1,x_1}}{\sY_{1,x_2;3,2}} \arrow[ld,"{1,x_1}"] \\
 & \displaystyle \frac{\sY_{2,x_2} \sY_{2,x_2;1,0}}{\sY_{1,x_1;3,2} \sY_{1,x_2;2,1}} \arrow[rd,"{2,x_2;1,0}"'] && \displaystyle \frac{\sY_{1,x_1;1,1} \sY_{2,x_2}}{\sY_{2,x_1;1,1} \sY_{2,x_1;2,1} \sY_{2,x_2;2,1}} \arrow[ld,"{1,x_1;1,1}"] \arrow[rd,"{2,x_2}"'] && \displaystyle \frac{\sY_{1,x_2;1,1} \sY_{2,x_1}}{\sY_{2,x_2;1,1} \sY_{2,x_2;2,1} \sY_{2,x_1;2,1}} \arrow[ld,"{2,x_1}"] \arrow[rd,"{1,x_2;1,1}"'] && \displaystyle \frac{\sY_{2,x_1} \sY_{2,x_1;1,0}}{\sY_{1,x_1;2,1} \sY_{1,x_2;3,2}} \arrow[ld,"{2,x_1;1,0}"] & \\
 && \displaystyle \frac{\sY_{2,x_2}}{\sY_{1,x_1;3,2} \sY_{2,x_2;2,1}} \arrow[rd,"{2,x_2}"'] && \displaystyle \frac{\sY_{1,x_1;1,1} \sY_{1,x_2;1,1}}{\sY_{2,x_1;1,1} \sY_{2,x_1;2,1} \sY_{2,x_1;1,1} \sY_{2,x_2;2,1}} \arrow[ld,"{1,x_1;1,1}"] \arrow[rd,"{1,x_2;1,1}"'] && \displaystyle \frac{\sY_{2,x_1}}{\sY_{1,x_1;2,1} \sY_{1,x_2;3,2}} \arrow[ld,"{2,x_1}"] && \\
 &&& \displaystyle \frac{\sY_{1,x_2;1,1}}{\sY_{1,x_1;3,2} \sY_{2,x_2;1,1} \sY_{2,x_2;2,1}} \arrow[rd,"{1,x_2;1,1}"'] && \displaystyle \frac{\sY_{1,x_1;1,1}}{\sY_{1,x_2;3,2} \sY_{2,x_1;1,1} \sY_{2,x_1;2,1}} \arrow[ld,"{1,x_1;1,1}"] &&& \\
 &&&& \displaystyle \frac{1}{\sY_{1,x_1;3,2} \sY_{1,x_2;3,2}} &&&&
\end{tikzcd}
    \caption{Hasse diagram for the iWeyl reflection flow of the weight $w= (2,0)$ for $B_2/C_2$ quiver}
    \label{fig:Hasse_BC2_20}
\end{figure}
\end{landscape}

\subsubsection*{Classical limit}

In addition to the classical limit of the $\scS$-factors \eqref{eq:S_lim} and \eqref{eq:S_lim2}, we also have $\scS(q_1^{-3} q_2^{-1}) \xrightarrow{q_1, q_2 \to 1} 1$ and
\begin{align}
    \scS_2(q_1^{-2}) = (1 + q_1^{-2}) \frac{1 - q_1^{-2} q_2^{-1}}{1 - q_1^{-4} q_2^{-1}}
    \ \longrightarrow \
    \begin{cases}
    2 & (q_1 \to 1) \\ 1 & (q_2 \to 1)
    \end{cases}
    \label{eq:S_lim3}
\end{align}
Hence, the $qq$-character is reduced in the classical limit as follows,
\begin{subequations}
\begin{align}
    \sT[\mathbf{W}_{2_1,x}^{(1)}]
    \xrightarrow{q_1 \to 1} \ &
    \sY_{1,x}^2 + \frac{\sY_{2,x}^4}{\sY_{1,x;0,1}^2} + \frac{\sY_{1,x;0,1}^2}{\sY_{2,x;0,1}^4} + \frac{1}{\sY_{1,x;0,2}^2}
     \nonumber \\ & 
    + 2 \qty(
    \frac{\sY_{1,x} \sY_{2,x}^2}{\sY_{1,x;0,1}} 
    + \frac{\sY_{1,x} \sY_{1,x;0,1}}{\sY_{2,x;0,1}^2}
    + \frac{\sY_{2,x}^2}{\sY_{1,x;0,1} \sY_{1,x;0,2}}
    + \frac{\sY_{1,x;0,1}}{\sY_{1,x;0,2} \sY_{2,x;0,1}^2}
    + \frac{\sY_{1,x}}{\sY_{1,x;0,2}}
    )
    \nonumber \\ & 
    + 4 \qty( 
    \frac{\sY_{1,x} \sY_{2,x}}{\sY_{2,x;0,1}}
    + \frac{\sY_{2,x}^3}{\sY_{1,x;0,1} \sY_{2,x;0,1}}
    + \frac{\sY_{1,x;0,1} \sY_{2,x}}{\sY_{2,x;0,1}^3}
    + \frac{\sY_{2,x}}{\sY_{1,x;0,2} \sY_{2,x;0,1}}
    )
    + 6 \frac{\sY_{2,x}^2}{\sY_{2,x;0,1}^2}
    \nonumber \\ &
    = \qty(\sY_{1,x} + \frac{\sY_{2,x}^2}{\sT_{1,x;0,1}} + 2 \frac{\sY_{2,x}}{\sY_{2,x;0,1}} + \frac{\sY_{1,x;0,1}}{\sY_{2,x;0,1}^2} + \frac{1}{\sY_{1,x;0,2}} )^2 
    \, ,
    \\[.5em]
    \xrightarrow{q_2 \to 1} \ &
    \sY_{1,x} \sY_{1,x;2,0} + \frac{\sY_{2,x} \sY_{2,x;1,0} \sY_{2,x;2,0} \sY_{2,x;3,0}}{\sY_{1,x;2,0} \sY_{1,x;4,0}} + \frac{\sY_{1,x;1,0} \sY_{1,x;3,0}}{\sY_{2,x;1,0} \sY_{2,x;2,0} \sY_{2,x;3,0} \sY_{2,x;4,0}} + \frac{1}{\sY_{1,x;3,0} \sY_{1,x;5,0}}
    \nonumber \\ & 
    + \frac{\sY_{1,x} \sY_{2,x;2,0} \sY_{2,x;3,0}}{\sY_{1,x;4,0}} 
    + \frac{\sY_{1,x;3,0} \sY_{1,x}}{\sY_{2,x;3,0} \sY_{2,x;4,0}}
    + \frac{\sY_{2,x} \sY_{2,x;1,0}}{\sY_{1,x;5,0} \sY_{1,x;2,0}}
    + \frac{\sY_{1,x;1,0}}{\sY_{1,x;3,0} \sY_{2,x;1,0} \sY_{2,x;2,0}}
    \nonumber \\ & 
    + \frac{\sY_{1,x} \sY_{2,x;2,0}}{\sY_{2,x;4,0}}
    + \frac{\sY_{2,x;2,0} \sY_{2,x} \sY_{2,x;1,0}}{\sY_{1,x;2,0} \sY_{2,x;4,0}}
    + \frac{\sY_{1,x;3,0} \sY_{2,x}}{\sY_{2,x;3,0} \sY_{2,x;4,0} \sY_{2,x;2,0}}
    + \frac{\sY_{2,x}}{\sY_{1,x;5,0} \sY_{2,x;2,0}}
    \nonumber \\ & 
    + \frac{\sY_{1,x}}{\sY_{1,x;5,0}}
    + \frac{\sY_{1,x;3,0} \sY_{2,x} \sY_{2,x;1,0}}{\sY_{1,x;1,0} \sY_{2,x;3,0} \sY_{2,x;4,0}}
    \, .
\end{align}
\end{subequations}
We obtain the $q$-character of the 14-dimensional module in the limit $q_2 \to 1$.
On the other hand, the limit $q_1 \to 1$ provides the tensor product of the fundamental module $w = (1,0)$.
We remark that, in the limit $q_1 \to 1$, the fundamental $qq$-character \eqref{eq:BC2_qq_ch_10} is already degenerated due to the $\scS$-factor.
The resulting module is associated with the six-dimensional module of $A_3/D_3$ quiver, which is the fractionalization of $B_2/C_2$ quiver~\cite{KPfractional}.

\subsection{$w = (0,2)$}\label{sec:BC2_w02}

The $qq$-character of weight $w = (0,2)$ is given as follows,
\begin{align}
    & \sT_{(0,2),x} = \sY_{2,x_1} \sY_{2,x_2} + \frac{\sY_{1,x_1;1,1} \sY_{1,x_2;1,1}}{\sY_{2,x_1;1,1} \sY_{2,x_2;1,1}} + \frac{\sY_{2,x_1;2,1} \sY_{2,x_2;2,1}}{\sY_{1,x_1;3,2} \sY_{1,x_2;3,2}} + \frac{1}{\sY_{2,x_1;3,2} \sY_{2,x_2;3,2}}
    \nonumber \\ &
    + \scS\qty(\frac{x_2}{x_1}) \qty[ \frac{\sY_{1,x_1;1,1} \sY_{2,x_2}}{\sY_{2,x_1;1,1}} + \frac{\sY_{2,x_1;2,1} \sY_{2,x_2}}{\sY_{1,x_1;3,2}} + \scS\qty(\frac{x_2}{x_1}q_1^{-2}q_2^{-1}) \frac{\sY_{2,x_2}}{\sY_{2,x_1;3,2}} + \frac{\sY_{1,x_2;1,1}}{\sY_{2,x_1;3,2} \sY_{2,x_2;1,1}} + \frac{\sY_{2,x_2;2,1}}{\sY_{1,x_2;3,2} \sY_{2,x_1;3,2}} ]
    \nonumber \\ &
    + \scS\qty(\frac{x_1}{x_2}) \qty[
    \frac{\sY_{1,x_2;1,1} \sY_{2,x_1}}{\sY_{2,x_2;1,1}} + \frac{\sY_{2,x_2;2,1} \sY_{2,x_1}}{\sY_{1,x_2;3,2}} + \scS\qty(\frac{x_1}{x_2}q_1^{-2}q_2^{-1}) \frac{\sY_{2,x_1}}{\sY_{2,x_2;3,2}} + \frac{\sY_{1,x_1;1,1}}{\sY_{2,x_2;3,2} \sY_{2,x_1;1,1}} + \frac{\sY_{2,x_1;2,1}}{\sY_{1,x_1;3,2} \sY_{2,x_2;3,2}}
    ]
    \nonumber \\ &
    + \scS_2\qty(\frac{x_2}{x_1}) \frac{\sY_{1,x_2;1,1} \sY_{2,x_1;2,1}}{\sY_{1,x_1;3,2} \sY_{2,x_2;1,1}} + \scS_2\qty(\frac{x_1}{x_2}) \frac{\sY_{1,x_1;1,1} \sY_{2,x_2;2,1}}{\sY_{1,x_2;3,2} \sY_{2,x_1;1,1}}
    \, ,
\end{align}
where the iWeyl reflection flow is described by the Hasse diagram in Fig.~\ref{fig:Hasse_BC2_02}.
There exist 16 monomials in the $qq$-character. 
We have the decomposition, $4 \otimes 4 = 10 \oplus 5 \oplus 1$.
The residues of this $qq$-character are given by
\begin{align}
    \operatorname*{Res}_{x_1=x_2} \sT_{(0,2),x} = \operatorname*{Res}_{x_1=x_2 q_1^2 q_2} \sT_{(0,2),x} = 0 \, , \quad
    \operatorname*{Res}_{x_1=x_2 q_1 q_2} \sT_{(0,2),x} = - \frac{(1-q_1)(1-q_2)}{1-q_1q_2} \sT_{(1,0),x_2} \, .
\end{align}

\begin{figure}[t]
    \centering
\begin{tikzcd}[column sep=-.4cm]
 &&& \sY_{2,x_1} \sY_{2,x_2} \arrow[ld,"{1,x_1}"'] \arrow[rd,"{1,x_2}"] &&& \\
 && \displaystyle \frac{\sY_{1,x_1;1,1} \sY_{2,x_2}}{\sY_{2,x_1;1,1}} \arrow[ld,"{1,x_1;1,1}"'] \arrow[rd,"{2,x_2}"] && \displaystyle \frac{\sY_{1,x_2;1,1} \sY_{2,x_1}}{\sY_{2,x_2;1,1}} \arrow[ld,"{2,x_1}"'] \arrow[rd,"{1,x_2;1,1}"] && \\
 & \displaystyle \frac{\sY_{2,x_1;2,1} \sY_{2,x_2}}{\sY_{1,x_1;3,2}} \arrow[ld,"{2,x_1;2,1}"'] \arrow[rd,"{2,x_2}"] && \displaystyle \frac{\sY_{1,x_1;1,1} \sY_{1,x_2;1,1}}{\sY_{2,x_1;1,1} \sY_{2,x_2;1,1}} \arrow[ld,"{1,x_1;1,1}"'] \arrow[rd,"{1,x_2;1,1}"] && \displaystyle \frac{\sY_{2,x_2;2,1} \sY_{2,x_1}}{\sY_{1,x_2;3,2}} \arrow[ld,"{2,x_1}"'] \arrow[rd,"{2,x_2;2,1}"] & \\
 \displaystyle \frac{\sY_{2,x_2}}{\sY_{2,x_1;3,2}} \arrow[rd,"{2,x_2}"] && \displaystyle \frac{\sY_{1,x_2;1,1} \sY_{2,x_1;2,1}}{\sY_{1,x_1;3,2} \sY_{2,x_2;1,1}} \arrow[ld,"{2,x_1;2,1}"'] \arrow[rd,"{1,x_2;1,1}"] && \displaystyle \frac{\sY_{1,x_1;1,1} \sY_{2,x_2;2,1}}{\sY_{1,x_2;3,2} \sY_{2,x_1;1,1}} \arrow[ld,"{1,x_1;1,1}"'] \arrow[rd,"{2,x_2;2,1}"] && \displaystyle \frac{\sY_{2,x_1}}{\sY_{2,x_2;3,2}} \arrow[ld,"{2,x_1}"'] \\
 & \displaystyle \frac{\sY_{1,x_2;1,1}}{\sY_{2,x_1;3,2} \sY_{2,x_2;1,1}} \arrow[rd,"{1,x_2;1,1}"] && \displaystyle \frac{\sY_{2,x_1;2,1} \sY_{2,x_2;2,1}}{\sY_{1,x_1;3,2} \sY_{1,x_2;3,2}} \arrow[ld,"{2,x_1;2,1}"'] \arrow[rd,"{2,x_2;2,1}"] && \displaystyle \frac{\sY_{1,x_1;1,1}}{\sY_{2,x_2;3,2} \sY_{2,x_1;1,1}} \arrow[ld,"{1,x_1;1,1}"'] & \\
 && \displaystyle \frac{\sY_{2,x_2;2,1}}{\sY_{1,x_2;3,2} \sY_{2,x_1;3,2}} \arrow[rd,"{2,x_2;2,1}"] && \displaystyle \frac{\sY_{2,x_1;2,1}}{\sY_{1,x_1;3,2} \sY_{2,x_2;3,2}} \arrow[ld,"{2,x_1;2,1}"'] && \\
 &&& \displaystyle \frac{1}{\sY_{2,x_1;3,2} \sY_{2,x_2;3,2}} &&&
\end{tikzcd}
\caption{Hasse diagram for the iWeyl reflection flow of the weight $w=(0,2)$ for $B_2/C_2$ quiver.}
    \label{fig:Hasse_BC2_02}
\end{figure}

\begin{proposition}\label{prop:BC2_KR02}
Specializing the spectral parameters $x = (x,xq_1)$, we obtain the $qq$-character for the KR module $\mathbf{W}_{2_1,x}^{(2)}$ as follows,
\begin{align}
    \sT[\mathbf{W}_{2_1,x}^{(2)}] & = \sY_{2,x} \sY_{2,x;1,0} + \frac{\sY_{1,x;1,1} \sY_{1,x;2,1}}{\sY_{2,x;1,1} \sY_{2,x;2,1}} + \frac{\sY_{2,x;2,1} \sY_{2,x;3,1}}{\sY_{1,x;3,2} \sY_{1,x;4,2}} + \frac{1}{\sY_{2,x;3,2} \sY_{2,x;4,2}}
    \nonumber \\ &
    + \scS\qty(q_1^{-1}) \qty[
    \frac{\sY_{1,x;2,1} \sY_{2,x}}{\sY_{2,x;2,1}} + \frac{\sY_{2,x;3,1} \sY_{2,x}}{\sY_{1,x;4,2}} + \scS\qty(q_1^{-3}q_2^{-1}) \frac{\sY_{2,x}}{\sY_{2,x;4,2}} + \frac{\sY_{1,x;1,1}}{\sY_{2,x;4,2} \sY_{2,x;1,1}} + \frac{\sY_{2,x;2,1}}{\sY_{1,x;3,2} \sY_{2,x;4,2}}
    ]
    \nonumber \\ &
    + \scS_2\qty(q_1) \frac{\sY_{1,x;2,1}}{\sY_{1,x;3,2}} + \scS_2\qty(q_1^{-1}) \frac{\sY_{1,x;1,1} \sY_{2,x;3,1}}{\sY_{1,x;4,2} \sY_{2,x;1,1}}
    \, ,
\end{align}
which corresponds to the 10-dimensional module (adjoint representation of $\mathrm{SO}(5) \cong \mathrm{Sp}(2)$) with the affine extension.    
\end{proposition}

\subsubsection*{Classical limit}

In the classical limit, the $\scS$-factors appearing in the $qq$-character $\sT[\mathbf{W}_{2_1,x}^{(2)}]$ are given by $\scS\qty(q_1^{-3}q_2^{-1}) \xrightarrow{q_1, q_2 \to 1} 1$ and 
\begin{align}
    \scS_2(q_1) = - q_1^{-1} \frac{1 - q_1 q_2^{-1}}{1 - q_1^{-1} q_2^{-1}}
    \ \longrightarrow \
    \begin{cases}
     -1 & (q_1 \to 1) \\ 1 & (q_2 \to 1)
    \end{cases}
\end{align}
Therefore, we obtain the following reductions,
\begin{subequations}
\begin{align}
    \sT[\mathbf{W}_{2_1,x}^{(2)}]
    \xrightarrow{q_1 \to 1} \ &
    \sY_{2,x}^2 + \frac{\sY_{1,x;0,1}^2}{\sY_{2,x;0,1}^2} + \frac{\sY_{2,x;0,1}^2}{\sY_{1,x;0,2}^2} + \frac{1}{\sY_{2,x;0,2}^2}
    \nonumber \\ &
    + 2 \qty(
    \frac{\sY_{1,x;0,1} \sY_{2,x}}{\sY_{2,x;0,1}} + \frac{\sY_{2,x;0,1} \sY_{2,x}}{\sY_{1,x;0,2}} + \frac{\sY_{2,x}}{\sY_{2,x;0,2}} + \frac{\sY_{1,x;0,1}}{\sY_{2,x;0,2} \sY_{2,x;0,1}} + \frac{\sY_{2,x;0,1}}{\sY_{1,x;0,2} \sY_{2,x;0,2}}
    + \frac{\sY_{1,x;0,1}}{\sY_{1,x;0,2}}
    )
    \nonumber \\ &
    = \qty(\sY_{2,x} + \frac{\sY_{1,x;0,1}}{\sY_{2,x;0,1}} + \frac{\sY_{2,x;0,1}}{\sY_{1,x;0,2}} + \frac{1}{\sY_{2,x;0,2}})^2
    \, , \\[.5em]
    \xrightarrow{q_2 \to 1} \ &
    \sY_{2,x} \sY_{2,x;1,0} + \frac{\sY_{1,x;1,0} \sY_{1,x;2,0}}{\sY_{2,x;1,0} \sY_{2,x;2,0}} + \frac{\sY_{2,x;2,0} \sY_{2,x;3,0}}{\sY_{1,x;3,0} \sY_{1,x;4,0}} + \frac{1}{\sY_{2,x;3,0} \sY_{2,x;4,0}} + \frac{\sY_{1,x;2,0} \sY_{2,x}}{\sY_{2,x;2,0}}
    \nonumber \\ &
    + \frac{\sY_{2,x;3,0} \sY_{2,x}}{\sY_{1,x;4,0}} + \frac{\sY_{2,x}}{\sY_{2,x;4,0}} + \frac{\sY_{1,x;1,0}}{\sY_{2,x;4,0} \sY_{2,x;1,0}} + \frac{\sY_{2,x;2,0}}{\sY_{1,x;3,0} \sY_{2,x;4,0}}
    + \frac{\sY_{1,x;2,0}}{\sY_{1,x;3,0}} + \frac{\sY_{1,x;1,0} \sY_{2,x;3,0}}{\sY_{1,x;4,0} \sY_{2,x;1,0}}
    \, .
\end{align}
\end{subequations}
The limit $q_1 \to 1$ is the degenerate limit, which provides the tensor product of two four-dimensional modules, while we obtain the $q$-character of the 10-dimensional module with the affine extension in the limit $q_2 \to 1$.

\subsection{$w = (1,1)$}\label{sec:BC2_11}

We consider the $qq$-character of weight $w = (1,1)$,
\begin{align}
    & \sT_{(1,1),x} = 
    \sY_{1,x_1} \sY_{2,x_2} 
    + \frac{\sY_{2,x_1} \sY_{2,x_1;1,0} \sY_{2,x_2}}{\sY_{1,x_1;2,1}}
    + \frac{\sY_{1,x_1} \sY_{1,x_2;1,1}}{\sY_{2,x_2;1,1}}
    + \frac{\sY_{2,x_1} \sY_{2,x_1;1,0} \sY_{2,x_2;2,1}}{\sY_{1,x_1;2,1} \sY_{1,x_2;3,2}}
    \nonumber \\ & 
    + \frac{\sY_{1,x_1;1,1} \sY_{1,x_2;1,1}}{\sY_{2,x_1;1,1} \sY_{2,x_1;2,1} \sY_{2,x_2;1,1}}
    + \frac{\sY_{2,x_2;2,1}}{\sY_{1,x_1;3,2} \sY_{1,x_2;3,2}}
    + \frac{\sY_{1,x_1;1,1}}{\sY_{2,x_1;1,1} \sY_{2,x_1;2,1} \sY_{2,x_2;3,2}}
    + \frac{1}{\sY_{1,x_1;3,2} \sY_{2,x_2;3,2}}
    \nonumber \\ & 
    + \scS_2\qty(\frac{x_1}{x_2}q_1^{-1} q_2^{-1}) \qty[
    \frac{\sY_{1,x_1} \sY_{2,x_2;2,1}}{\sY_{1,x_2;3,2}}
    + \frac{\sY_{1,x_1}}{\sY_{2,x_2;3,2}}
    + \frac{\sY_{2,x_1} \sY_{2,x_1;1,0}}{\sY_{1,x_1;2,1} \sY_{2,x_2;3,2}}
    ]
    + \scS_2\qty(\frac{x_2}{x_1}q_1 q_2) 
    \frac{\sY_{1,x_2;1,1} \sY_{2,x_1} \sY_{2,x_1;1,0}}{\sY_{1,x_1;2,1} \sY_{2,x_2;1,1}}
    \nonumber \\ & 
    + \scS_2\qty( \frac{x_2}{x_1} ) \qty[
    \frac{\sY_{1,x_1;1,1} \sY_{2,x_2}}{\sY_{2,x_1;1,1} \sY_{2,x_1;2,1}}
    + \frac{\sY_{2,x_2}}{\sY_{1,x_1;3,2}}
    + \frac{\sY_{1,x_2;1,1}}{\sY_{1,x_1;3,2} \sY_{2,x_2;1,1}}
    ]     
    + \scS_2\qty(\frac{x_1}{x_2}) \frac{\sY_{1,x_1;1,1} \sY_{2,x_2;2,1}}{\sY_{1,x_2;3,2} \sY_{2,x_1;1,1} \sY_{2,x_1;2,1}}
    \nonumber \\ & 
    + \scS(q_1^{-1}) \qty[
    \scS\qty(\frac{x_2}{x_1}q_1^{-1}) \frac{\sY_{2,x_1} \sY_{2,x_2}}{\sY_{2,x_1;2,1}}
    + \scS\qty(\frac{x_1}{x_2}) \qty( \frac{\sY_{1,x_2;1,1} \sY_{2,x_1}}{\sY_{2,x_1;2,1} \sY_{2,x_2;1,1}}
    + \frac{\sY_{2,x_1} \sY_{2,x_2;2,1}}{\sY_{1,x_2;3,2} \sY_{2,x_1;2,1}}
    )
    + \scS\qty(\frac{x_1}{x_2}q_1^{-2}q_2^{-1}) \frac{\sY_{2,x_1}}{\sY_{2,x_1;2,1} \sY_{2,x_2;3,2}}
    ],
\end{align}
where the iWeyl reflection flow is described in Fig.~\ref{fig:Hasse_BC2_11}.
This $qq$-character shows the decomposition of the tensor product, $4 \otimes 5 = 16 \oplus 4$.
We have
\begin{align}
    \operatorname*{Res}_{x_1 = x_2} \sT_{(1,1),x} = 0 \, , \qquad
    \operatorname*{Res}_{x_1 = x_2 q_1^3 q_2^2} \sT_{(1,1),x} = - \frac{(1-q_1^2)(1-q_2)}{1-q_1^2q_2} \sT_{(0,1),x_2 q_1^2 q_2} \, .
\end{align}
\begin{proposition}
We specialize the spectral parameters as (a) $x = (x q_1^3 q_2, x)$ and (b) $x = (x,xq_1^2)$ to obtain the $qq$-character for the 16-dimensional module: 
\begin{subequations}
\begin{align}
    \sT_{(1,1),x}
    \xrightarrow{(\text{a})} \ &
    \sY_{1,x;3,1} \sY_{2,x} 
    + \frac{\sY_{2,x} \sY_{2,x;3,1} \sY_{2,x;4,1} }{\sY_{1,x;5,2}}
    + \frac{\sY_{1,x;1,1} \sY_{1,x;3,1} }{\sY_{2,x;1,1}}
    + \frac{\sY_{2,x;2,1} \sY_{2,x;3,1} \sY_{2,x;4,1} }{\sY_{1,x;3,2} \sY_{1,x;5,2} }
    \nonumber \\ & 
    + \frac{\sY_{1,x;1,1} \sY_{1,x;4,2} }{\sY_{2,x;1,1} \sY_{2,x;4,2} \sY_{2,x;5,2} }
    + \frac{\sY_{2,x;2,1}}{\sY_{1,x;3,2} \sY_{1,x;6,3} }
    + \frac{\sY_{1,x;4,2}}{\sY_{2,x;3,2} \sY_{2,x;4,2} \sY_{2,x;5,2} }
    + \frac{1}{\sY_{1,x;6,3} \sY_{2,x;3,2}}
    \nonumber \\ & 
    + \scS_2\qty( q_1^{-3} q_2^{-1} ) \qty(
    \frac{\sY_{1,x;4,2} \sY_{2,x}}{\sY_{2,x;4,2} \sY_{2,x;5,2}}
    + \frac{\sY_{2,x}}{\sY_{1,x;6,3}}
    + \frac{\sY_{1,x;1,1}}{\sY_{1,x;6,3} \sY_{2,x;1,1}}
    )
    + \scS_2\qty(q_1^{-2}) 
    \frac{\sY_{1,x;1,1} \sY_{2,x;3,1} \sY_{2,x;4,1}}{\sY_{1,x;5,2} \sY_{2,x;1,1}}
    \nonumber \\ & 
    + \scS_2\qty(q_1^{-1}) \qty( 
    \frac{\sY_{1,x_1;4,2} \sY_{2,x;2,1}}{\sY_{1,x;3,2} \sY_{2,x;4,2} \sY_{2,x;5,2}}
    + \frac{\sY_{1,x;1,1} \sY_{2,x;3,1}}{\sY_{2,x;1,1} \sY_{2,x;5,2} }
    + \frac{\sY_{2,x;2,1} \sY_{2,x;3,1} }{\sY_{1,x;3,2} \sY_{2,x;5,2}}
    )
    + \scS(q_1^{-1}) \scS\qty(q_1^{-4}q_2^{-1}) \frac{\sY_{2,x} \sY_{2,x;3,1} }{\sY_{2,x;5,2}}
    \, ,
\end{align}
\begin{align}
    \sT_{(1,1),x} \xrightarrow{(\text{b})} \ &
    \sY_{1,x} \sY_{2,x;2,0} 
    + \frac{\sY_{2,x} \sY_{2,x;1,0} \sY_{2,x;2,0}}{\sY_{1,x;2,1}}
    + \frac{\sY_{1,x} \sY_{1,x;3,1}}{\sY_{2,x;3,1}}
    + \frac{\sY_{2,x} \sY_{2,x;1,0} \sY_{2,x;4,1}}{\sY_{1,x;2,1} \sY_{1,x;5,2}}
    \nonumber \\ & 
    + \frac{\sY_{1,x;1,1} \sY_{1,x;3,1}}{\sY_{2,x;1,1} \sY_{2,x;2,1} \sY_{2,x;3,1}}
    + \frac{\sY_{2,x;4,1}}{\sY_{1,x;3,2} \sY_{1,x;5,2}}
    + \frac{\sY_{1,x;1,1}}{\sY_{2,x;1,1} \sY_{2,x;2,1} \sY_{2,x;5,2}}
    + \frac{1}{\sY_{1,x;3,2} \sY_{2,x;5,2}}
    \nonumber \\ & 
    + \scS_2\qty(q_1^{-3} q_2^{-1}) \qty(
    \frac{\sY_{1,x} \sY_{2,x;4,1}}{\sY_{1,x;5,2}}
    + \frac{\sY_{1,x}}{\sY_{2,x;5,2}}
    + \frac{\sY_{2,x} \sY_{2,x;1,0}}{\sY_{1,x;2,1} \sY_{2,x;5,2}}
    )
    + \scS_2\qty(q_1^{-2}) \frac{\sY_{1,x;1,1} \sY_{2,x;4,1}}{\sY_{1,x;5,2} \sY_{2,x;1,1} \sY_{2,x;2,1}}
    \nonumber \\ & 
    + \scS_2\qty(q_1^{-1}) 
    \qty(
    \frac{\sY_{1,x;3,1} \sY_{2,x} \sY_{2,x;1,0}}{\sY_{1,x;2,1} \sY_{2,x;3,1}}
    + \frac{\sY_{1,x;3,1} \sY_{2,x}}{\sY_{2,x;2,1} \sY_{2,x;3,1}}
    + \frac{\sY_{2,x} \sY_{2,x;4,1}}{\sY_{1,x;5,2} \sY_{2,x;2,1}}
    )
    + \scS(q_1^{-1}) \scS\qty(q_1^{-4}q_2^{-1}) \frac{\sY_{2,x}}{\sY_{2,x;2,1} \sY_{2,x;5,2}}
    \, .
\end{align}
\end{subequations}    
\end{proposition}

\begin{figure}[t]
    \centering
\begin{tikzcd}[column sep=-1cm,row sep=.7cm]
 &&& \sY_{1,x_1} \sY_{2,x_2} \arrow[ld,"{1,x_1}"'] \arrow[rd,"{2,x_2}"] &&& \\
 && \displaystyle \frac{\sY_{2,x_1} \sY_{2,x_1;1,0} \sY_{2,x_2}}{\sY_{1,x_1;2,1}} \arrow[ld,"{2,x_1;1,0}"'] \arrow[rd,"{2,x_2}"] && \displaystyle \frac{\sY_{1,x_1} \sY_{1,x_2;1,1}}{\sY_{2,x_2;1,1}} \arrow[ld,"{1,x_1}"'] \arrow[rd,"{1,x_2;1,1}"] && \\
 & \displaystyle \frac{\sY_{2,x_1} \sY_{2,x_2}}{\sY_{2,x_1;2,1}} \arrow[ld,"{2,x_1}"'] \arrow[rd,"{2,x_2}"] && \displaystyle \frac{\sY_{1,x_2;1,1} \sY_{2,x_1} \sY_{2,x_1;1,0}}{\sY_{1,x_1;2,1} \sY_{2,x_2;1,1}} \arrow[ld,"{2,x_1;1,0}"'] \arrow[rd,"{1,x_2;1,1}"] && \displaystyle \frac{\sY_{1,x_1} \sY_{2,x_2;2,1}}{\sY_{1,x_2;3,2}} \arrow[ld,"{1,x_1}"'] \arrow[rd,"{2,x_2;2,1}"] & \\
 \displaystyle \frac{\sY_{1,x_1;1,1} \sY_{2,x_2}}{\sY_{2,x_1;1,1} \sY_{2,x_1;2,1}} \arrow[d,"{1,x_1;1,1}"'] \arrow[rrd,"{2,x_2}"] && \displaystyle \frac{\sY_{1,x_2;1,1} \sY_{2,x_1}}{\sY_{2,x_1;2,1} \sY_{2,x_2;1,1}} \arrow[d,"{2,x_1}"'] \arrow[rrd,"{1,x_2;1,1}"] && \displaystyle \frac{\sY_{2,x_1;1,0} \sY_{2,x_2;2,1}}{\sY_{1,x_1;2,1} \sY_{1,x_2;3,2}} \arrow[d,"{2,x_1;1,0}"'] \arrow[rrd,"{2,x_2;2,1}"] && \displaystyle \frac{\sY_{1,x_1}}{\sY_{2,x_2;3,2}} \arrow[d,"{1,x_1}"'] \\
 \displaystyle \frac{\sY_{2,x_2}}{\sY_{1,x_1;3,2}} \arrow[rd,"{2,x_2}"'] && \displaystyle \frac{\sY_{1,x_1;1,1} \sY_{1,x_2;1,1}}{\sY_{2,x_1;1,1} \sY_{2,x_1;2,1} \sY_{2,x_2;1,1}} \arrow[ld,"{1,x_1;1,1}"] \arrow[rd,"{1,x_2;1,1}"'] && \displaystyle \frac{\sY_{2,x_1} \sY_{2,x_2;2,1}}{\sY_{1,x_2;3,2} \sY_{2,x_1;2,1}} \arrow[ld,"{2,x_1}"] \arrow[rd,"{2,x_2;2,1}"'] && \displaystyle \frac{\sY_{2,x_1} \sY_{2,x_1;1,0}}{\sY_{1,x_1;2,1} \sY_{2,x_2;3,2}} \arrow[ld,"{2,x_1;1,0}"] \\
 & \displaystyle \frac{\sY_{1,x_2;1,1}}{\sY_{1,x_1;3,2} \sY_{2,x_2;1,1}} \arrow[rd,"{1,x_2;1,1}"'] && \displaystyle \frac{\sY_{1,x_1;1,1} \sY_{2,x_2;2,1}}{\sY_{1,x_2;3,2} \sY_{2,x_1;1,1} \sY_{2,x_1;2,1}} \arrow[ld,"{1,x_1;1,1}"] \arrow[rd,"{2,x_2;2,1}"'] && \displaystyle \frac{\sY_{2,x_1}}{\sY_{2,x_1;2,1} \sY_{2,x_2;3,2}} \arrow[ld,"{2,x_1}"] & \\
 && \displaystyle \frac{\sY_{2,x_2;2,1}}{\sY_{1,x_1;3,2} \sY_{1,x_2;3,2}} \arrow[rd,"{2,x_2;2,1}"'] && \displaystyle \frac{\sY_{1,x_1;1,1}}{\sY_{2,x_1;1,1} \sY_{2,x_1;2,1} \sY_{2,x_2;3,2}} \arrow[ld,"{1,x_1;1,1}"] && \\
 &&& \displaystyle \frac{1}{\sY_{1,x_1;3,2} \sY_{2,x_2;3,2}} &&&
\end{tikzcd}
    \caption{Hasse diagram for the iWeyl reflection flow of the weight $w=(1,1)$ for $B_2/C_2$ quiver.}
    \label{fig:Hasse_BC2_11}
\end{figure}

\subsubsection*{Classical limit}

We then study the classical limit of the $qq$-character for the 16-dimensional module.
In the limit $q_2 \to 1$, we obtain the corresponding $q$-character, while we observe the tensor product structure in the degenerate limit $q_1 \to 1$.
\begin{itemize}
    \item[(a)] $x = (xq_1^3q_2, x)$
\begin{subequations}
\begin{align}
    (q_1 \to 1) : \quad & 
    \sY_{1,x;0,1} \sY_{2,x} 
    + \frac{\sY_{2,x} \sY_{2,x;0,1}^2}{\sY_{1,x;0,2}}
    + \frac{\sY_{1,x;0,1}^2 }{\sY_{2,x;0,1}}
    + \frac{\sY_{2,x;0,1}^3}{\sY_{1,x;0,2} \sY_{1,x;0,2} }
    + \frac{\sY_{1,x;0,1} \sY_{1,x;0,2} }{\sY_{2,x;0,1} \sY_{2,x;0,2}^2}
    + \frac{\sY_{2,x;0,1}}{\sY_{1,x;0,2} \sY_{1,x;0,3} }
    \nonumber \\ & 
    + \frac{\sY_{1,x;0,2}}{\sY_{2,x;0,2}^3 }
    + \frac{1}{\sY_{1,x;0,3} \sY_{2,x;0,2}}
    + \frac{\sY_{1,x;0,2} \sY_{2,x}}{\sY_{2,x;0,2}^2}
    + \frac{\sY_{2,x}}{\sY_{1,x;0,3}}
    + \frac{\sY_{1,x;0,1}}{\sY_{1,x;0,3} \sY_{2,x;0,1}}
    \nonumber \\ & 
    + 2 \qty(
    \frac{\sY_{1,x;0,1} \sY_{2,x;0,1}}{\sY_{1,x;0,2}}
    + \frac{\sY_{2,x} \sY_{2,x;0,1} }{\sY_{2,x;0,2}}
    )
    + 3 \qty( 
    \frac{\sY_{1,x_1;0,2} \sY_{2,x;0,1}}{\sY_{1,x;0,2} \sY_{2,x;0,2}^2}
    + \frac{\sY_{1,x;0,1}}{\sY_{2,x;0,2} }
    + \frac{\sY_{2,x;0,1}^2 }{\sY_{1,x;0,2} \sY_{2,x;0,2}}
    )
    \nonumber \\ & 
    = \qty( \sY_{1,x;0,1} 
    + \frac{\sY_{2,x;0,1}^2}{\sY_{1,x;0,2}} 
    + 2 \frac{\sY_{2,x;0,1}}{\sY_{2,x;0,2}} 
    + \frac{\sY_{1,x;0,2}}{\sY_{2,x;0,2}^2} 
    + \frac{1}{\sY_{1,x;0,3}} ) 
    \qty( \sY_{2,x} + \frac{\sY_{1,x;0,1}}{\sY_{2,x;0,1}} + \frac{\sY_{2,x;0,1}}{\sY_{1,x;0,2}} + \frac{1}{\sY_{2,x;0,2}} )
    \, , 
\end{align}
\begin{align}
    (q_2 \to 1) : \quad & 
    \sY_{1,x;3,0} \sY_{2,x} 
    + \frac{\sY_{2,x;3,0} \sY_{2,x;4,0} \sY_{2,x}}{\sY_{1,x;5,0}}
    + \frac{\sY_{1,x;3,0} \sY_{1,x;1,0}}{\sY_{2,x;1,0}}
    + \frac{\sY_{2,x;3,0} \sY_{2,x;4,0} \sY_{2,x;2,0}}{\sY_{1,x;5,0} \sY_{1,x;3,0}}
    \nonumber \\ & 
    + \frac{\sY_{1,x;4,0} \sY_{1,x;1,0}}{\sY_{2,x;4,0} \sY_{2,x;5,0} \sY_{2,x;1,0}}
    + \frac{\sY_{2,x;2,0}}{\sY_{1,x;6,0} \sY_{1,x;3,0}}
    + \frac{\sY_{1,x;4,0}}{\sY_{2,x;4,0} \sY_{2,x;5,0} \sY_{2,x;3,0}}
    + \frac{1}{\sY_{1,x;6,0} \sY_{2,x;3,0}}
    \nonumber \\ & 
    + \frac{\sY_{1,x;4,0} \sY_{2,x}}{\sY_{2,x;4,0} \sY_{2,x;5,0}}
    + \frac{\sY_{2,x}}{\sY_{1,x;5,0}}
    + \frac{\sY_{1,x;1,0}}{\sY_{1,x;6,0} \sY_{2,x;1,0}}
    + \frac{\sY_{1,x;1,0} \sY_{2,x;3,0} \sY_{2,x;4,0}}{\sY_{1,x;5,0} \sY_{2,x;1,0}}
    \nonumber \\ & 
    + \frac{\sY_{1,x;4,0} \sY_{2,x;2,0}}{\sY_{1,x;3,0} \sY_{2,x;4,0} \sY_{2,x;5,0}}
    + \frac{\sY_{1,x;1,0} \sY_{2,x;3,0}}{\sY_{2,x;5,0} \sY_{2,x;1,0}}
    + \frac{\sY_{2,x;3,0} \sY_{2,x;2,0}}{\sY_{1,x;3,0} \sY_{2,x;5,0}}
    + \frac{\sY_{2,x;3,0} \sY_{2,x}}{\sY_{2,x;5,0}}
\end{align}
\end{subequations}
    \item[(b)] $x = (x, xq_1^2)$
\begin{subequations}
\begin{align}
    (q_1 \to 1) : \quad &
    \sY_{1,x} \sY_{2,x} 
    + \frac{\sY_{2,x}^3}{\sY_{1,x;0,1}}
    + \frac{\sY_{1,x} \sY_{1,x;0,1}}{\sY_{2,x;0,1}}
    + \frac{\sY_{2,x}^2 \sY_{2,x;0,1}}{\sY_{1,x;0,1} \sY_{1,x;0,2}}
    + \frac{\sY_{1,x;0,1}^2}{\sY_{2,x;0,1}^3}
    + \frac{\sY_{2,x;0,1}}{\sY_{1,x;0,2}^2}
    \nonumber \\ & 
    + \frac{\sY_{1,x;0,1}}{\sY_{2,x;0,1}^2}
    + \frac{1}{\sY_{1,x;0,2}^2}
    + \frac{\sY_{1,x} \sY_{2,x;0,1}}{\sY_{1,x;0,2}}
    + \frac{\sY_{1,x}}{\sY_{2,x;0,2}}
    + \frac{\sY_{2,x}^2}{\sY_{1,x;0,1} \sY_{2,x;0,2}}
    \nonumber \\ & 
    + 2 \qty( \frac{\sY_{1,x;0,1}}{\sY_{1,x;0,2} \sY_{2,x;0,1}} + \frac{\sY_{2,x}}{\sY_{2,x;0,1} \sY_{2,x;0,2}} )
    + 3
    \qty(
    \frac{\sY_{2,x}^2}{\sY_{2,x;0,1}}
    + \frac{\sY_{1,x;0,1} \sY_{2,x}}{\sY_{2,x;0,1}^2}
    + \frac{\sY_{2,x}}{\sY_{1,x;0,2}}
    )
    \nonumber \\ &
    = \qty( \sY_{1,x} 
    + \frac{\sY_{2,x}^2}{\sY_{1,x;0,1}} 
    + 2 \frac{\sY_{2,x}}{\sY_{2,x;0,1}} 
    + \frac{\sY_{1,x;0,1}}{\sY_{2,x;0,1}^2} 
    + \frac{1}{\sY_{1,x;0,2}} ) 
    \qty( \sY_{2,x} + \frac{\sY_{1,x;0,1}}{\sY_{2,x;0,1}} + \frac{\sY_{2,x;0,1}}{\sY_{1,x;0,2}} + \frac{1}{\sY_{2,x;0,2}} )
\end{align}
\begin{align}
    (q_2 \to 1) : \quad &
    \sY_{1,x} \sY_{2,x;2,0} 
    + \frac{\sY_{2,x} \sY_{2,x;1,0} \sY_{2,x;2,0}}{\sY_{1,x;2,0}}
    + \frac{\sY_{1,x} \sY_{1,x;3,0}}{\sY_{2,x;3,0}}
    + \frac{\sY_{2,x} \sY_{2,x;1,0} \sY_{2,x;4,0}}{\sY_{1,x;2,0} \sY_{1,x;5,0}}
    \nonumber \\ & 
    + \frac{\sY_{1,x;1,0} \sY_{1,x;3,0}}{\sY_{2,x;1,0} \sY_{2,x;2,0} \sY_{2,x;3,0}}
    + \frac{\sY_{2,x;4,0}}{\sY_{1,x;3,0} \sY_{1,x;5,0}}
    + \frac{\sY_{1,x;1,0}}{\sY_{2,x;1,0} \sY_{2,x;2,0} \sY_{2,x;5,0}}
    + \frac{1}{\sY_{1,x;3,0} \sY_{2,x;5,0}}
    \nonumber \\ & 
    + \frac{\sY_{1,x} \sY_{2,x;4,0}}{\sY_{1,x;5,0}}
    + \frac{\sY_{1,x}}{\sY_{2,x;5,0}}
    + \frac{\sY_{2,x} \sY_{2,x;1,0}}{\sY_{1,x;2,0} \sY_{2,x;5,0}}
    + \frac{\sY_{1,x;1,0} \sY_{2,x;4,0}}{\sY_{1,x;5,0} \sY_{2,x;1,0} \sY_{2,x;2,0}}
    \nonumber \\ & 
    + \frac{\sY_{1,x;3,0} \sY_{2,x} \sY_{2,x;1,0}}{\sY_{1,x;2,0} \sY_{2,x;3,0}}
    + \frac{\sY_{1,x;3,0} \sY_{2,x}}{\sY_{2,x;2,0} \sY_{2,x;3,0}}
    + \frac{\sY_{2,x} \sY_{2,x;4,0}}{\sY_{1,x;5,0} \sY_{2,x;2,0}}
    + \frac{\sY_{2,x}}{\sY_{2,x;2,0} \sY_{2,x;5,0}}
\end{align}
\end{subequations}
\end{itemize}


\section{$\widehat{A}_0$ quiver}\label{sec:A_0^}

We consider $\widehat{A}_0$ quiver, which is the simplest affine quiver.
In this case, there exists a single node with the loop edge, i.e., the Jordan quiver.
The corresponding quiver variety is identified with the Hilbert scheme of points on $\mathbb{C}^2$ for the rank-one case $(w=1)$ and the Quot scheme for the higher rank case $(w > 1)$.
The Cartan matrix is given by
\begin{align}
    c 
    = 1 + q - \mu - \mu^{-1} q 
    = (1 - \mu) (1 - \mu^{-1} q)
    = (1 - q_3) (1 - q_4)
    \, ,
\end{align}
where we define
\begin{align}
    (q_3, q_4) = (\mu, \mu^{-1} q)
    \label{eq:q34}
\end{align}
and
\begin{align}
    \underline{q} = (q_1,q_2,q_3,q_4)
    \, .
\end{align}
We assume $q_{3}, q_4 \neq 1$,%
\footnote{%
From the gauge theory point of view, supersymmetry is enhanced from 8 to 16 supercharges in the limit $q_{3,4} \to 1$.}
so that the Cartan matrix is invertible.
We remark the relation
\begin{align}
    q_1 q_2 = q_3 q_4 = q
    \, .
\end{align}
Then, we have the iWeyl reflection with a formal counting parameter $\fq \in \mathbb{C}^\times$, $|\mathfrak{q}|<1$, 
\begin{align}
    \text{iWeyl}: \quad
    \sY_x \ \longmapsto \ \fq \, \scS(q_3) \frac{\sY_{x q_3} \sY_{x q_4}}{\sY_{xq}} 
    \, .
\end{align}
We remark that $\scS(q_3) = \scS(q_4)$.

\subsection{Weight one}

Let $\mathscr{P}$ the set of partitions.
The fundamental $qq$-character of the weight $w = (1)$ is given by an infinite sum over $\mathscr{P}$~\cite{Nekrasov:2015wsu,Kimura:2015rgi}, 
\begin{align}
    \sT_{1,x} = \sum_{\lambda \in \mathscr{P}} \fq^{|\lambda|} Z_\lambda(\underline{q})
    \prod_{s \in \partial_+ \lambda} \sY_{x(s)}
    \prod_{s \in \partial_- \lambda} \sY_{x(s)q}^{-1}
    \, ,
    \label{eq:A0^_qq_ch1}
\end{align}
where $\displaystyle |\lambda| = \sum_{k=1}^\infty \lambda_k$, and we denote by $\partial_+ \lambda$ and $\partial_- \lambda$ the outer and inner corners of the partition $\lambda$, where one can add/remove a box. 
For $s = (s_1,s_2)$ and the transpose of partition denoted by $\lambda^\text{T}$, we define
\begin{align}
    x(s) = x q_3^{s_1-1} q_4^{s_2-1}
    \, , \qquad
    a(s) = \lambda_{s_2} - s_1
    \, , \qquad
    \ell(s) = \lambda^\text{T}_{s_1} - s_2
    \, .
 \label{eq:arm-leg_def}
\end{align}
Then, we have
\begin{align}
    Z_\lambda(\underline{q})
    = \prod_{s \in \lambda} \scS(q_3^{a(s) + 1} q_4^{-\ell(s)})
    = \prod_{s \in \lambda} \scS(q_3^{- a(s)} q_4^{\ell(s) + 1})
    \, .
    \label{eq:Z_A0^}
\end{align}
We remark that this factor is identified with the fixed point contribution to the Nekrasov partition function of 5d $\mathcal{N} = 1^*$ U(1) gauge theory with the $\Omega$-background parameters $(q_3,q_4)$ (not $(q_1,q_2)$) with the adjoint mass $q_1$ ($q_2$, resp.), which is geometrically interpreted as the $\chi_{q_1}$-genus ($\chi_{q_2}$-genus, resp.) of the instanton moduli space~\cite{Nekrasov:2002qd,Nekrasov:2003rj}.

The $qq$-character \eqref{eq:A0^_qq_ch1} is interpreted to be associated with the infinite-dimensional Fock module of the quantum toroidal algebra of $\mathfrak{gl}_1$~\cite{Nekrasov:2013xda,Feigin:2015raa,Feigin:2016pld} and its elliptic uplift~\cite{Konno:2021zvl}.%
\footnote{%
The $qq$-character of the MacMahon module of the quantum toroidal $\mathfrak{gl}_1$ has been also constructed~\cite{Kimura:2023bxy}. 
} 
See also \cite{Liu:2022gwf} for a related geometric representation theoretical perspective of the $qq$-character of $\widehat{A}_0$ quiver.
Similarly to finite-type quivers, we can reduce the module by tuning the parameters:
Let $(\mathsf{r},\mathsf{k})$ be non-negative integers.
We define a subset of $\mathscr{P}$ as follows,
\begin{align}
 \mathscr{P}_{\mathsf{r},\mathsf{k}} = \{ \lambda \in \mathscr{P} \mid \lambda_i - \lambda_{i+\mathsf{k}+1} \le \mathsf{r} , \, 1 \le i \le \lambda^\text{T}_1 - \mathsf{k}  \}  \, .
\end{align}
We remark that this set is similar, but different from that of $(\mathsf{r},\mathsf{k})$-admissible partitions~\cite{Feigin:2002IMRN}.

\begin{proposition}
 Under the \emph{resonance} condition,
\begin{align}
    q_3^{\mathsf{r}+1} q_4^{-\mathsf{k}} = q_i \, , \quad i = 1, 2 \, ,
    \label{eq:pit_cond1}
\end{align}
 the rational function $Z_\lambda(\underline{q})$ can be non-zero only when $\lambda \in \mathscr{P}_{\mathsf{r},\mathsf{k}}$.
\end{proposition}
 One may rewrite the condition~\eqref{eq:pit_cond1} as $q_3^{-\mathsf{r}+1} q_4^{\mathsf{k}} = q_i$ for $i = 1, 2$.
\begin{proof}
 By the definition of the rational function $Z_\lambda$ in~\eqref{eq:Z_A0^}, it vanishes when there exists $s = (i,j) \in \lambda$ such that $(a(s),\ell(s)) = (\mathsf{r},\mathsf{k})$.
 This condition is indeed equivalent to $\lambda_i - \lambda_{i + \mathsf{k} + 1} \ge \mathsf{r} + 1$. 
\begin{align}
 \lambda = 
 \begin{tikzpicture}[baseline=(current bounding box.center),scale=.5]
  \draw[latex-latex] (13,-9) node [right] {$j$} -- (-1.5,-9) -- (-1.5,1.5) node [left] {$i$};
  \draw (0,0) -- ++(2,0) -- ++(0,-2);
  \draw (1,0) -- ++(0,-2);
  \draw (3,-1) -- ++(0,-1);
  \draw (4,-1) -- ++(0,-2);
  \draw (4,-4) -- ++(0,-1);
  \draw (5,-1) -- ++(0,-2);
  \draw (5,-4) -- ++(0,-1);
  \draw (6,-1) -- ++(0,-1);
  \draw (0,-1) -- ++(7,0);
  \draw (0,-2) -- ++(7,0);
  \draw (3,-5) -- ++(3,0);
  \draw (3,-6) -- ++(3,0);
  \draw (9,-5) rectangle ++(1,-1);
  \draw (8,-5) -- ++(1,0);
  \draw (8,-6) -- ++(1,0);
  \draw (10,-6) -- ++(1,0);
  \filldraw[draw=black,fill=black!20] (4,-5) rectangle ++(1,-1);
  \begin{scope}[dotted]
   \draw (2,0) -- ++(2,0) -- ++(0,-1);
   \draw (3,0) -- ++(0,-1);
   \draw (4,-3) -- ++(0,-1);
   \draw (5,-3) -- ++(0,-1);
   \draw (2,-5) -- ++(1,0);
   \draw (2,-6) -- ++(1,0);
   \draw (6,-5) -- ++(2,0);
   \draw (6,-6) -- ++(2,0);
   \draw (7,-1) -- ++(1,0);
   \draw (7,-2) -- ++(1,0);
  \draw (11,-6) -- ++(1,0);
  \end{scope}
  \begin{scope}[-latex,thick]
   \draw (4.5,-5) -- ++(0,4); 
   \draw (5,-5.5) -- ++(5,0) node [right] {$\lambda_i$}; 
  \end{scope}
  \node at (1.5,.5) {$\lambda_{i+\mathsf{k}+1}$};
  \node at (4.5,-5.5) {$s$};
  \node at (7,-6.5) {$a(s) = \mathsf{r}$};
  \node at (7,-3.5) {$\ell(s) = \mathsf{k}$};
 \end{tikzpicture}
\end{align}
 Hence $Z_\lambda = 0$ for any $\lambda \not \in \mathscr{P}_{\mathsf{r},\mathsf{k}}$.
 By the definition~\eqref{eq:arm-leg_def}, we have $\ell(s) = \lambda^\text{T}_j - i = \mathsf{k}$, and the range of the index $i$ is given by $1 \le i = \lambda^\text{T}_j - \mathsf{k} \le \lambda^\text{T}_1 - \mathsf{k}$.
\end{proof}

By this Proposition, the infinite sum \eqref{eq:A0^_qq_ch1} is restricted to the subset ${\mathscr{P}}_{\mathsf{r},\mathsf{k}}$ under the resonance condition~\eqref{eq:pit_cond1}.
If we instead consider the condition $q_3^{\mathsf{r}+1} q_4^{-\mathsf{k}}=1$, $Z_\lambda$ is singular for $\lambda \in \mathscr{P}_{\mathsf{r},\mathsf{k}}$. 
We remark that these are not a specialization of the spectral parameter.


\subsection{General weight}\label{sec:gl1_gen}

We then consider the $qq$-character of $\widehat{A}_0$ quiver for general weight $w$.
In this case, the $qq$-character is given by a summation over $w$-tuple partitions $\lambda = (\lambda_\alpha)_{\alpha = 1,\ldots,w}$ with the spectral parameters $x = (x_\alpha)_{\alpha = 1,\ldots,w}$,
\begin{align}
    \sT_{w,x} = \sum_{\lambda} \fq^{|\lambda|} Z_{\lambda}(\underline{q};x) 
    \prod_{\alpha = 1}^w \qty[ \prod_{s \in \partial_+ \lambda_\alpha} \sY_{x_\alpha(s)}
    \prod_{s \in \partial_- \lambda_\alpha} \sY_{x_\alpha(s)q}^{-1} ]
    \label{eq:A0^_qq_ch2}
\end{align}
where $\displaystyle |\lambda| = \sum_{\alpha=1}^w |\lambda_\alpha|$, and we define
\begin{align}
    Z_{\lambda}(\underline{q};x) = \prod_{\alpha=1}^w Z_{\lambda_\alpha}(\underline{q})
    \prod_{1 \le \alpha < \beta \le w} \qty[
    \prod_{s \in \lambda_\alpha} \scS\qty( \frac{x_\beta}{x_\alpha} q_3^{a_\beta(s) + 1} q_4^{-\ell_\alpha(s)})
    \prod_{s \in \lambda_\beta} \scS\qty( \frac{x_\beta}{x_\alpha} q_3^{-a_\alpha(s) } q_4^{\ell_\beta(s)+1})
    ]
\end{align}
with
\begin{align}
    a_\alpha(s) = \lambda_{\alpha,s_2} - s_1
    \, , \qquad
    \ell_\alpha(s) = \lambda^\text{T}_{\alpha,s_1} - s_2
    \, .
    \label{eq:arm_leg2}
\end{align}
We remark 
\begin{align}
    \scS\qty( \frac{x_\beta}{x_\alpha} q_3^{a_\beta(s) + 1} q_4^{-\ell_\alpha(s)})
    = \scS\qty( \frac{x_\alpha}{x_\beta} q_3^{-a_\beta(s)} q_4^{\ell_\alpha(s)+1})
    \, , \quad
    \scS\qty( \frac{x_\beta}{x_\alpha} q_3^{-a_\alpha(s) } q_4^{\ell_\beta(s)+1})
    = \scS\qty( \frac{x_\alpha}{x_\beta} q_3^{a_\alpha(s) +1} q_4^{-\ell_\beta(s)})
    \, .
\end{align}
The factor $Z_{\lambda}(\underline{q};x)$ coincides with the fixed point contribution to the Nekrasov partition function of 5d $\mathcal{N} = 1^*$ U($w$) gauge theory under identification of the spectral parameters with the Coulomb moduli.
From the representation theoretical point of view, the weight $w$ character is associated with the $w$-tensor product of the Fock modules.
In this context, the spectral parameters $(x_\alpha)_{\alpha = 1,\ldots,w}$ are identified with the evaluation parameters.

We may consider the same condition as the weight one case \eqref{eq:pit_cond1}, which imposes the same condition for all the partitions $(\lambda_\alpha)_{\alpha = 1,\ldots, w}$.
For the tensor product of the Fock modules, 
we consider the resonance condition
\begin{align}
    \frac{x_{\beta}}{x_\alpha} q_3^{\mathsf{r}+1} q_4^{-\mathsf{k}} = q_i \, , \quad i = 1, 2 \, .
\end{align}
In this case, we obtain the following restriction on the set of partitions,
\begin{align}
    \lambda_{\beta,i} - \lambda_{\alpha,i+\mathsf{k}+1} \le \mathsf{r} 
    \, , \quad 1 \le i \le \lambda_{\alpha,1}^\text{T} - \mathsf{k} \, . \label{eq:Burge'}
\end{align}
A similar restriction, known as the Burge condition~\cite{Burge:1993JCTA}, has been discussed also in the relation to the plane partitions~\cite{Gessel:1997TAMS}, the Higgsing of 5d $\mathcal{N}=1^*$ theory~\cite{Chen:2012we}, and the minimal models of W-algebras~\cite{Bershtein:2014qma,Alkalaev:2014sma,Belavin:2015ria}. 
However, its precise relation to the condition~\eqref{eq:Burge'} is not clear at this moment. 
We leave this for a future work.

\section{$\widehat{A}_{p-1}$ quiver}\label{sec:A_r^}

We study the higher rank cyclic quiver $\widehat{A}_{p-1}$, which consists of $r$ nodes.
In this case, we consider the Cartan matrix given by
\begin{align}
    c = 
    \begin{pmatrix}
    1 + q & -\nu_1 & & & - \nu_0^{-1} q \\
    - \nu_1^{-1} q & 1 + q & - \nu_2 & & \\
    & - \nu_2^{-1} q & 1 + q & \ddots & \\
    &&\ddots& \ddots & - \nu_{r-1} \\
    - \nu_0 &&& - \nu_{r-1}^{-1} q & 1 + q
    \end{pmatrix}
\end{align}
having the determinant
\begin{align}
    \det c = (1 - \nu_\text{tot}) (1 - \nu_\text{tot}^{-1} q^{p})
    \, , \qquad
    \nu_\text{tot} := \prod_{i=0}^{p-1} \nu_i
    \, .
\end{align}
We assume $\nu_\text{tot} \neq 1$, $\nu_\text{tot}^{-1} q^{p} \neq 1$ in order that the Cartan matrix is invertible.
We specialize the mass parameters to $\nu_i = \mu$ for $i = 0,\ldots,p-1$ without loss of generality, so that $\nu_\text{tot} = \mu^{p}$.
We use the same notation $(q_3,q_4) = (\mu, \mu^{-1} q)$ as before \eqref{eq:q34}.
Then, the iWeyl reflection is given as follows,
\begin{align}
    \text{iWeyl} \ : \quad 
    \sY_{i,x} \ \longmapsto \ \fq_i \, \frac{\sY_{i+1,xq_3} \sY_{i-1,xq_4}}{\sY_{i,xq}}
    \, , \qquad
    i \in \mathbb{Z}/p\mathbb{Z}
    \, ,
\end{align}
where we interpret the node index periodic $i \equiv i + p$ (mod $p$).
We also denote $\mathbb{Z}_p = \mathbb{Z}/p\mathbb{Z}$.

\subsection{Weight one}

The fundamental $qq$-character associated with the $i$-th node, corresponding to the weight $w = (w_j)_{j = 0,\ldots,p-1}$ with $w_j = \delta_{i,j}$, 
is again given by summation over the partition,
\begin{align}
    \sT_{i,x} = \sum_{\lambda} \underline{\fq}_i^\lambda Z^{(r)}_{\lambda}(\underline{q})
    \prod_{s \in \partial_+ \lambda} \sY_{\si(s)+i,x(s)}
    \prod_{s \in \partial_- \lambda} \sY_{\si(s)+i,x(s)q}^{-1}
    \, ,
    \label{eq:Ar^_qq_ch1}
\end{align}
where we define the coloring index $\mathsf{i}(s) \equiv s_1 - s_2$ (mod $p$) for $s = (s_1,s_2) \in \mathbb{Z}_{>0}^2$, and
\begin{align}
    \underline{\fq}_i^\lambda = \prod_{s \in \lambda} \fq_{\si(s) + i} = \prod_{j=0}^{p-1} \fq_{j+i}^{|\lambda|_j}
    \, , \qquad
    |\lambda|_i = \#\{ s \in \lambda \mid \si(s) = i \}
    \, .
\end{align}
We also define
\begin{align}
    Z^{(p)}_{\lambda}(\underline{q})
    = \prod_{s \in \lambda \mid h(s) \in p\mathbb{Z}} \scS(q_3^{a(s) + 1} q_4^{-\ell(s)})
    = \prod_{s \in \lambda \mid h(s) \in p\mathbb{Z}} \scS(q_3^{- a(s)} q_4^{\ell(s) + 1})
    \, ,
    \label{eq:Z_Ar^}
\end{align}
with the hook length
\begin{align}
    h(s) = a(s) + \ell(s) + 1
    \, .
\end{align}
In this case, the rational factor \eqref{eq:Z_Ar^} is identified with the fixed point contribution to the Nekrasov partition function of 5d $\mathcal{N}=1^*$ U(1) theory on $\mathbb{C}^2/\mathbb{Z}_p$ with the $\Omega$-background parameters $(q_3,q_4)$~\cite{Fucito:2004ry}.
The $qq$-character \eqref{eq:Ar^_qq_ch1} is then associated with the colored Fock module of the quantum toroidal $\mathfrak{gl}_p$~\cite{Jimbo:1983if,Takemura:1999PRIMS,Feigin:2013JA}. 

Let us discuss truncation of the configuration.
We tune the parameters as before~\eqref{eq:pit_cond1}.
In this case, the rational function~\eqref{eq:Z_Ar^} becomes zero if $s \in \lambda$ where $(a(s),\ell(s))=(\mathsf{r},\mathsf{k})$ and also $h(s) \in p \mathbb{Z}$.
The condition~\eqref{eq:pit_cond1} is rewritten in terms of the total mass $\nu_\text{tot}$ instead of $\mu$,
\begin{align}
    \nu_\text{tot}^{(\mathsf{r} + \mathsf{k} + 1)/r} q^{- \mathsf{k}} = q_{i} \, , \quad i = 1, 2 \, .
    \label{eq:pit_cond2}
\end{align}
Hence, we can exclude a configuration that contains a box $s$ of $(a(s),\ell(s))=(\mathsf{r},\mathsf{k})$ and $h(s) \in p \mathbb{Z}$ for $\widehat{A}_{p-1}$ quiver.

\subsection{General weight}

We study the case with general weight $w = (w_i)_{i = 0,\ldots,p-1} \in \mathbb{Z}_{\ge 0}^{p}$.
We define the index set $I_w = \{ (i,\alpha) \mid i = 0,\ldots,p-1, \alpha = 1,\ldots,w_i \}$ with $\displaystyle \mathsf{w} := |I_w| = \sum_{i=0}^{p-1} w_i$.
Then, the $qq$-character is given as a summation over the set of partitions $\lambda = (\lambda_{i,\alpha})_{(i,\alpha) \in I_w}$,
\begin{align}
    \sT_{w,x} = \sum_{\lambda} \underline{\fq}_w^\lambda Z_\lambda^{(p)}(\underline{q};x) \prod_{(i,\alpha) \in I_w} \qty[ \prod_{s \in \partial_+ \lambda_{i,\alpha}} \sY_{\si(s)+i,x_{i,\alpha}(s)} \prod_{s \in \partial_- \lambda_{i,\alpha}} \sY_{\si(s)+i,x_{i,\alpha}(s)q}^{-1} ]
\end{align}
where we define
\begin{align}
    \underline{\fq}_w^\lambda 
    = \prod_{(i,\alpha) \in I_w} \prod_{s \in \lambda_{i,\alpha}} \fq_{\si(s)+i}
    \, ,
\end{align}
and the rational factor
\begin{align}
    Z_\lambda^{(p)}(\underline{q};x) = \prod_{(i,\alpha)\in I_w} Z_{\lambda_{i,\alpha}}^{(p)}(\underline{q}) \prod_{(i,\alpha) < (j,\beta)} \qty[
    \prod_{\substack{s \in \lambda_{i,\alpha} \\ h^{i,\alpha}_{j,\beta}(s) \in p \mathbb{Z}}} \scS\qty( \frac{x_{j,\beta}}{x_{i,\alpha}} q_3^{a_{j,\beta}(s) + 1} q_4^{-\ell_{i,\alpha}(s)})
    \prod_{\substack{s \in \lambda_{j,\beta} \\ h_{i,\alpha}^{j,\beta}(s) \in p \mathbb{Z}}} \scS\qty( \frac{x_{j,\beta}}{x_{i,\alpha}} q_3^{-a_{i,\alpha}(s) } q_4^{\ell_{j,\beta}(s)+1})
    ]
    \label{eq:Z_Ar^_gen}
\end{align}
with
\begin{align}
    h^{i,\alpha}_{j,\beta}(s) = a_{i,\alpha}(s) + \ell_{j,\beta}(s) + 1 - (i-j)
    \, .
\end{align}
The arm and leg are similarly defined as in \eqref{eq:arm_leg2}.
In this case, this factor is identified with the fixed point contribution to the Nekrasov partition function of 5d $\mathcal{N}=1^*$ U($\mathsf{w}$) theory on $\mathbb{C}^2/\mathbb{Z}_p$.

As in the previous case (\S\ref{sec:gl1_gen}), we have two possibilities to impose the restriction on the set of partitions.
The first is the pit condition~\eqref{eq:pit_cond2} obtained from the diagonal factor $Z_{\lambda_{i,\alpha}}^{(p)}(\underline{q})$ in \eqref{eq:Z_Ar^_gen}.
The second is from the pair contribution associated with $(i,\alpha)$ and $(j,\beta)$, which yields the irreducible modules constructed from the tensor product of the Fock modules for the quantum toroidal $\mathfrak{gl}_p$~\cite{Feigin:2013JA}. 
In this case, we consider the following resonance condition 
\begin{align}
    \frac{x_{j,\beta}}{x_{i,\alpha}} q_3^{\mathsf{r}+1} q_4^{-\mathsf{k}} = q_i \, , \quad i = 1, 2 \, .
    \label{eq:resonance_Zr}
\end{align}
If $\mathsf{r} + \mathsf{k} + 1 - ( i - j ) \in p \mathbb{Z}$, we have the restriction on the set of partitions,
\begin{align}
    \lambda_{j,\beta,s} - \lambda_{i,\alpha,s+\mathsf{k}+1} \le \mathsf{r} \, .
\end{align}


\appendix

\section{$D_4$ quiver}\label{sec:D4}

We consider the $D_4$ quiver with the following labels:
\[
\dynkin[root radius=.8mm,edge length= 8mm,mark=o,label]{D}{4}
\]
This quiver consists of four nodes, so that we have four $\sY$-operators and four fundamental $qq$-characters constructed from the iWeyl reflection,
\begin{align}
    \text{iWeyl} : \quad
    \qty(\sY_{1,x}, \sY_{2,x}, \sY_{3,x}, \sY_{4,x})
    \ \longmapsto \
    \qty( \frac{\sY_{2,x}}{\sY_{1,xq}}, \frac{\sY_{1,xq} \sY_{3,xq} \sY_{4,xq}}{\sY_{2,xq}}, \frac{\sY_{2,x}}{\sY_{3,xq}}, \frac{\sY_{2,q}}{\sY_{4,xq}} ) 
    \, .
\end{align}
The fundamental $qq$-character associated with the node $i = 1$ is given by
\begin{align}
    \sT_{1,x} & = \sY_{1,x} + \frac{\sY_{2,x}}{\sY_{1,xq}} + \frac{\sY_{3,xq} \sY_{4,xq}}{\sY_{2,xq}} + \frac{\sY_{3,xq}}{\sY_{4,xq^2}} + \frac{\sY_{4,xq}}{\sY_{3,xq^2}} + \frac{\sY_{2,xq}}{\sY_{3,xq^2} \sY_{4,xq^2}} + \frac{\sY_{1,xq^2}}{\sY_{2,xq^2}} + \frac{1}{\sY_{1,xq^3}} \, .
\end{align}
We similarly obtain the fundamental $qq$-characters for the nodes $i = 3$, $4$, $\sT_{3,x}$ and $\sT_{4,x}$, through the permutation.
They correspond to the eight-dimensional vector, spinor, and conjugate spinor modules of $U_q(\widehat{\mathfrak{so}}_8)$.

\begin{proposition}[\cite{Nekrasov:2015wsu}]\label{prop:D4_T2+}
The fundamental $qq$-character for the node $i=2$ of $D_4$ quiver is given as follows:
\begin{align}
    \sT_{2,x} = \sT^+_{2,x} + \sT^0_{2,x} + \sT^-_{2,x}
\end{align}
where
\begin{subequations}
\begin{align}
    \sT^+_{2,x} & = \sY_{2,x} + \frac{\sY_{1,xq} \sY_{3,xq} \sY_{4,xq}}{\sY_{2,xq}} + \frac{\sY_{3,xq} \sY_{4,xq}}{\sY_{1,xq^2}} + \frac{\sY_{1,xq} \sY_{4,xq}}{\sY_{3,xq^2}} + \frac{\sY_{1,xq} \sY_{3,xq}}{\sY_{4,xq^2}} + \frac{\sY_{1,xq} \sY_{2,xq}}{\sY_{3,xq^2} \sY_{4,xq^2}} + \frac{\sY_{2,xq} \sY_{3,xq}}{\sY_{1,xq^2} \sY_{4,xq^2}} + \frac{\sY_{2,xq} \sY_{4,xq}}{\sY_{1,xq^2} \sY_{3,xq^2}} \nonumber \\
    & \quad + \frac{\sY_{1,xq} \sY_{1,xq^2}}{\sY_{2,xq^2}} + \frac{\sY_{3,xq} \sY_{3,xq^2}}{\sY_{2,xq^2}} + \frac{\sY_{4,xq} \sY_{4,xq^2}}{\sY_{2,xq^2}} + \frac{(\sY_{2,xq})^2}{\sY_{1,xq^2} \sY_{3,xq^2} \sY_{4,xq^2}}
    \, , \label{eq:D4_T2+} \\
    \sT^0_{2,x} & = \scS(q^{-1}) \qty( \frac{\sY_{1,xq}}{\sY_{1,xq^3}} + \frac{\sY_{3,xq}}{\sY_{3,xq^3}} + \frac{\sY_{4,xq}}{\sY_{4,xq^3}} ) + \left( \mathfrak{c}(q_1,q_2)  + \frac{(1-q_1)(1-q_2)}{1-q} \partial_{\log x} \log \frac{\sY_{2,xq} \sY_{2,xq^2}}{\prod_{i=1,3,4} \sY_{i,xq^2}} \right) \frac{\sY_{2,xq}}{\sY_{2,xq^2}}
    \, , \\
    \sT^-_{2,x} & = \frac{\sY_{1,xq^2} \sY_{3,xq^2} \sY_{4,xq^2}}{(\sY_{2,xq^2})^2} + \frac{\sY_{2,xq}}{\sY_{1,xq^2} \sY_{1,xq^3}} + \frac{\sY_{2,xq}}{\sY_{3,xq^2} \sY_{3,xq^3}} + \frac{\sY_{2,xq}}{\sY_{4,xq^2} \sY_{4,xq^3}} + \frac{\sY_{3,xq^2} \sY_{4,xq^2}}{\sY_{1,xq^3} \sY_{2,xq^2}} + \frac{\sY_{1,xq^2} \sY_{4,xq^2}}{\sY_{2,xq^2} \sY_{3,xq^3}} + \frac{\sY_{1,xq^2} \sY_{3,xq^2}}{\sY_{2,xq^2} \sY_{4,xq^3}} 
    \nonumber \\
    & \quad
    + \frac{\sY_{1,xq^2}}{\sY_{3,xq^3} \sY_{4,xq^3}} + \frac{\sY_{3,xq^2}}{\sY_{1,xq^3} \sY_{4,xq^3}} + \frac{\sY_{4,xq^2}}{\sY_{1,xq^3} \sY_{3,xq^3}} + \frac{\sY_{2,xq^2}}{\sY_{1,xq^3} \sY_{3,xq^3} \sY_{4,xq^3}} + \frac{1}{\sY_{2,xq^3}} \, ,
\end{align}
\end{subequations}
and 
\begin{align}
    \mathfrak{c}(q_1,q_2) = \frac{1-6q+q^2+(q_1+q_2)(1+q)}{(1-q)^2} \, .
\end{align}
\end{proposition}
This $qq$-character corresponds to the 28-dimensional module of $U_q(\widehat{\mathfrak{so}}_8)$ with the affine extension, involving the derivative terms.

\begin{proof}
The highest-weight monomial is given by $\sY_{2,x}$. 
Applying the iWeyl reflections, we obtain the positive part $\sT_{2,x}^+$.
In order to deal with the monomials in the second line of \eqref{eq:D4_T2+}, we should apply the step (\ref{alg:iWeyl2b}) of the iWeyl reflection algorithm:
We compute the regularized iWeyl reflection, then take the collision limit.
For the monomial $\sY_{i,xq} \sY_{i,xq^2} / \sY_{2,xq^2}$ ($i = 1, 3, 4$), we have the following reflection, 
\begin{align}
    \text{iWeyl}: \quad
     \frac{\sY_{i,xq} \sY_{i,xzq^2}}{\sY_{2,xq^2}} 
     & \longmapsto 
     \scS\left(zq\right) \frac{\sY_{i,xzq^2} \sY_{2,xq}}{\sY_{i,xq^2} \sY_{2,xq^2}} + \scS\left(z^{-1} q^{-1}\right) \frac{\sY_{i,xq} \sY_{2,xzq^2}}{\sY_{i,xzq^3} \sY_{2,xq^2}} 
    \, .
\end{align}
We study the behavior of these terms in the limit $z \to 1$.
Noticing the expansion $f(xz) = f(x) - (1-z) \pdv{f(x)}{\log x} + O((1-z)^2)$, the first term is given by
\begin{align}
    \scS\left(zq\right) \frac{\sY_{i,xzq^2} \sY_{2,xq}}{\sY_{i,xq^2} \sY_{2,xq^2}} = \scS\left(z^{-1}\right) \frac{\sY_{2,xq}}{\sY_{2,xq^2}} \left( 1 - (1-z)  \frac{\partial_{\log x}\sY_{i,xq}}{\sY_{i,xq}} \right) + O(1-z) \, , \label{eq:D4_T2a}
\end{align}
while the second term is regular in the limit,
\begin{align}
    \lim_{z \to 1} \scS\left(z^{-1} q^{-1}\right) \frac{\sY_{i,xq} \sY_{2,xzq^2}}{\sY_{i,xzq^3} \sY_{2,xq^2}} = \scS\left(q^{-1}\right) \frac{\sY_{i,xq}}{\sY_{i,xzq^3}} \, . 
\end{align}
The regularization for the monomial $(\sY_{2,xq})^2/\sY_{1,xq^2} \sY_{3,xq^2} \sY_{4,xq^2}$ is given as follows,
\begin{align}
    \frac{\sY_{2,xq} \sY_{2,xzq}}{\sY_{1,xq^2} \sY_{3,xq^2} \sY_{4,xq^2}}
    & \longmapsto
    \scS\left(z\right) \frac{\sY_{2,xzq}}{\sY_{2,xq^2}} + \scS\left(z^{-1}\right) \frac{\sY_{2,xq}}{\sY_{2,xzq^2}} \prod_{i=1,3,4} \frac{\sY_{i,xzq^2}}{\sY_{i,xq^2}} 
    \nonumber \\
    & = 
    \Bigg[ \left( \scS\left(z\right) + \scS\left(z^{-1}\right) \right) - (1-z)  \scS\left(z\right) \frac{\partial_{\log x}\sY_{2,xq}}{\sY_{2,xq}} + (1-z) \scS\left(z^{-1} \right) \frac{\partial_{\log x}\sY_{2,xq^2}}{\sY_{2,xq^2}} \nonumber \\
    & \qquad - (1-z) \scS\left(z^{-1} \right) \sum_{i=1,3,4} \frac{\partial_{\log x}\sY_{i,xq^2}}{\sY_{i,xq^2}} \Bigg] \frac{\sY_{2,xq}}{\sY_{2,xq^2}} + O(1-z) \, . \label{eq:D4_T2b}
\end{align}
The generated terms in \eqref{eq:D4_T2a} are also obtained in \eqref{eq:D4_T2b}, so that they appear only once after the reflection.
Therefore, the regularized iWeyl reflection is in total given by
\begin{align}
    & \sum_{i=1,3,4} \frac{\sY_{i,xq} \sY_{i,xzq^2}}{\sY_{2,xq^2}} + \frac{\sY_{2,xq} \sY_{2,xzq}}{\sY_{1,xq^2} \sY_{3,xq^2} \sY_{4,xq^2}} 
    \nonumber \\
    & \longmapsto \Bigg[ \left( \scS\left(z\right) + \scS\left(z^{-1}\right) \right) - (1-z)  \scS\left(z\right) \frac{\partial_{\log x}\sY_{2,xq}}{\sY_{2,xq}} + (1-z) \scS\left(z^{-1} \right) \frac{\partial_{\log x}\sY_{2,xq^2}}{\sY_{2,xq^2}} \nonumber \\
    & \qquad - (1-z) \scS\left(z^{-1} \right) \sum_{i=1,3,4} \frac{\partial_{\log x}\sY_{i,xq^2}}{\sY_{i,xq^2}} \Bigg] \frac{\sY_{2,xq}}{\sY_{2,xq^2}} + \sum_{i=1,3,4} \scS\left(z^{-1} q^{-1}\right) \frac{\sY_{i,xq} \sY_{2,xzq^2}}{\sY_{i,xzq^3} \sY_{2,xq^2}} + O(1-z) \, .
\end{align}
Taking the limit $z \to 1$, we obtain the zero-weight part $\sT_{2,x}^0$ where
\begin{align}
    \lim_{z \to 1} \left( \scS\left(z\right) + \scS\left(z^{-1}\right) \right) = \mathfrak{c}(q_1,q_2) \, .
\end{align}
The remaining part $\sT_{2,x}^-$ is obtained by the iWeyl reflections without any regularization process.
\end{proof}

\section{Geometric construction of $qq$-character}\label{sec:geom_const}

We provide a geometric construction of the $qq$-character based on the quiver variety~\cite{Nekrasov:2015wsu,KPfractional}.
In this construction, instead of the operator $\sY$ in the algebraic construction, the so-called index functor $\scY$ for a category of vector bundles over the quiver variety plays a fundamental role.

\subsection{Notation}\label{sec:notation_geom}

We introduce a notation used in this Section.
Let $\mathbf{X}$ be a vector bundle for which we denote the set of Grothendieck roots (exponentials of Chern roots) by $X = \{ x_\alpha \}_{\alpha = 1,\ldots \operatorname{rk} \mathbf{X}}$．
Then, the corresponding Chern character is given by $\operatorname{ch} \mathbf{X} = \sum_{\alpha = 1}^{\operatorname{rk} \mathbf{X}} x_\alpha = \sum_{x \in X} x$.
We denote the twisted wedge products of the vector bundle $\mathbf{X}$ by
\begin{align}
    \wedge_y \mathbf{X} = \bigoplus_{k=0}^{\operatorname{rk} \mathbf{X}} (-y)^{k} \wedge^k \mathbf{X}
    \, .
\end{align}
We simply write $\wedge \mathbf{X} = \wedge_1 \mathbf{X}$ when unambiguous.
Then, its Chern character is given by
\begin{align}
    \operatorname{ch} \wedge_y \mathbf{X} = \prod_{x \in X} ( 1 - x y)  \, .
\end{align}
Identifying $\operatorname{ch} \mathbf{X}' = x'$ for a rank-one vector bundle $\mathbf{X}'$, we have
\begin{align}
    \operatorname{ch} \wedge_{x'} \mathbf{X} = \operatorname{ch} \wedge (\mathbf{X} \otimes \mathbf{X}') \, .
\end{align}

\subsection{Quiver variety and geometric formula}\label{sec:quiv_var}

For a quiver $\Gamma = (\Gamma_0,\Gamma_1)$, we define the dimension vectors, $v = (v_i)_{i \in \Gamma_0}$ and $w = (w_i)_{i \in \Gamma_0}$, such that the vector spaces attached to the nodes and the corresponding framing spaces are given by $\mathbf{V}_i = \mathbb{C}^{v_i}$ and $\mathbf{W}_i = \mathbb{C}^{w_i}$ for $i\in \Gamma_0$.
We also write $\mathbf{V} = (\mathbf{V}_i)_{i \in \Gamma_0}$ and $\mathbf{W} = (\mathbf{W}_i)_{i \in \Gamma_0}$, and denote the automorphism groups by $G_v = \mathrm{GL}(\mathbf{V}) = \prod_{i \in \Gamma_0} \mathrm{GL}(\mathbf{V}_i)$ and $G_w = \mathrm{GL}(\mathbf{W}) = \prod_{i \in \Gamma_0} \mathrm{GL}(\mathbf{W}_i)$.
We denote the corresponding Cartan tori by $\mathbb{T}_\mathbf{V}$ and $\mathbb{T}_\mathbf{W}$. 
We define the liner maps,
\begin{align}
    B_{e:i \to j} \in \operatorname{Hom}(\mathbf{V}_i, \mathbf{V}_j) \, , \quad 
    \overline{B}_{e:i \to j} \in \operatorname{Hom}(\mathbf{V}_j, \mathbf{V}_i) \, , \quad     
    I_i \in \operatorname{Hom}(\mathbf{W}_i, \mathbf{V}_i) \, , \quad 
    J_i \in \operatorname{Hom}(\mathbf{V}_i, \mathbf{W}_i) \, ,
\end{align}
and the moment map $\mu = (\mu_i)_{i \in \Gamma_0}$, where
\begin{align}
    \mu_i = I_i J_i + \sum_{e: i \to j} \overline{B}_e B_e - \sum_{e: j \to i} B_e \overline{B}_e \, .
\end{align}
Then, Nakajima's quiver variety associated with the quiver $\Gamma$ is given by the GIT quotient, $\mathfrak{M}_{w,v}^\Gamma = \mu^{-1}(0)/\!\!/\mathrm{GL}(\mathbf{V})$, under a stability condition.
See, e.g., \cite{Kirillov:2016} for details.

Let $\mathbf{Q} = \mathbf{Q}_1 \oplus \mathbf{Q}_2$ and $\mathbf{M} = \bigoplus_{e \in \Gamma_1} \mathbf{M}_e$, such that $\operatorname{ch} \mathbf{Q}_{i} = q_{i}$ for $i = 1$, $2$ 
with $q_{12} = q_1 q_2$ and $\operatorname{ch} \mathbf{M}_e = \nu_e$. 
We consider the torus actions denoted by $\mathbb{T}_{\mathbf{Q}}$ and $\mathbb{T}_\mathbf{M}$ as follows,
\begin{align}
    \mathbb{T}_{\mathbf{Q}} \times \mathbb{T}_\mathbf{M} : (B_e,\overline{B}_e,I_i,J_i) \longmapsto (\nu_e q_{12}^{-1} B_e,\nu_e^{-1} \overline{B}_e, q_{12}^{-1} I_i, J_i) \, , \qquad \mu_i \longmapsto q_{12}^{-1} \mu_i \, .
\end{align}
We denote the total torus action by $\mathbb{T} = \mathbb{T}_\mathbf{V} \times \mathbb{T}_\mathbf{W} \times \mathbb{T}_{\mathbf{Q}} \times \mathbb{T}_\mathbf{M}$. 
We will abuse notation of the vector space (e.g., $\mathbf{V}$, $\mathbf{W}$) for the associated vector bundle over the quiver variety.
Then, we give the geometric definition of the $qq$-character as follows.

\begin{definition}\label{def:qq-ch}
Let $\mathbf{Y} = (\mathbf{Y}_i)_{i \in \Gamma_0}$ be a formal $\Gamma_0$-graded vector bundle over the quiver variety $\mathfrak{M}_{w,v}^\Gamma$. 
We define a virtual vector bundle $\mathbf{Y}_{w,v}^\Gamma = (\mathbf{Y}_{w,v;i}^\Gamma)_{i \in \Gamma_0} = \left(\mathbf{W}_i \ominus \bigoplus_{j \in \Gamma_0} \mathbf{c}_{ij} \otimes \mathbf{V}_j \right)_{i \in \Gamma_0}$ where $\operatorname{ch} \mathbf{c}_{ij} = c_{ij}$ and write $\mathbf{Y} \otimes \mathbf{Y}_{w,v}^\Gamma = \bigoplus_{i \in \Gamma_0} \mathbf{Y}_i \otimes \mathbf{Y}_{w,v;i}^\Gamma$.
We define the $qq$-character 
by a $\mathbb{T}$-equivariant integral over $\mathfrak{M}_{w,v}^\Gamma$,
\begin{align}
    \mathsf{T}_{w,x}^{(q_1, q_2)}[\mathbf{Y}] = \sum_v 
    q_2^{\frac{1}{2} \operatorname{dim} \mathfrak{M}_{w,v}^\Gamma}\int_{\mathfrak{M}_{w,v}^\Gamma} \operatorname{ch} \wedge \left( \mathbf{Y} \otimes \mathbf{Y}_{w,v}^\Gamma \right) \cdot \operatorname{ch} \wedge_{q_2^{-1}} T^\vee\mathfrak{M}_{w,v}^\Gamma \cdot \operatorname{td}(T\mathfrak{M}_{w,v}^\Gamma) \, ,
\end{align}
where we denote the Todd class by $\operatorname{td}(\cdot)$.
The dimension of the quiver variety is given by $\frac{1}{2} \operatorname{dim} \mathfrak{M}_{w,v}^\Gamma = \sum_{i \in \Gamma_0} w_i v_i - \sum_{i,j\in\Gamma_0} v_i {c}_{ij}^{[0]} v_j$. 
\end{definition}
For a finite-type quiver, the summation over $v$ in the definition of the $qq$-character is finite since the corresponding quiver variety becomes empty for sufficiently large $v$, while it will be a infinite series for affine quivers.
This implies a finiteness of the $qq$-character for finite-type quivers.

This integral is understood as an equivariant integral over the quiver variety, and thus we may apply the equivariant localization formula to evaluate it.
\begin{proposition}\label{prop:qq-ch_int}
    The equivariant integral over the quiver variety localizes to the contour integral on the torus $\mathbb{T}_\mathbf{V}$,
    \begin{align}\label{eq:qq-ch_int}
        & q_2^{\frac{1}{2} \operatorname{dim} \mathfrak{M}_{w,v}^\Gamma} \int_{\mathfrak{M}_{w,v}^\Gamma} \operatorname{ch} \wedge \left( \mathbf{Y} \otimes \mathbf{Y}_{w,v}^\Gamma \right) \cdot \operatorname{ch} \wedge_{q_2^{-1}} T^\vee\mathfrak{M}_{w,v}^\Gamma \cdot \operatorname{td}(T\mathfrak{M}_{w,v}^\Gamma)
        \nonumber \\
        & = \frac{\mathsf{c}_{v}}{v!} \oint_{\mathbb{T}_\mathbf{V}} \prod_{i \in \Gamma_0} \left[ \frac{\prod_{\alpha=1}^{w_i} \mathscr{Y}_{i,x_{i,\alpha}}}{\prod_{I=1}^{v_i}\mathscr{A}_{i,z_{i,I}}} \prod_{\substack{\alpha = 1,\ldots,w_i \\ I = 1,\ldots,v_i}} \mathscr{S}\left( \frac{z_{i,I}}{x_{i,\alpha}} \right) \prod_{1 \le I \neq J \le v_i} \mathscr{S}\left( \frac{z_{i,J}}{z_{i,I}} \right)^{-1} \right] \prod_{e: i \to j} \prod_{\substack{I = 1,\ldots,v_i \\ J = 1,\ldots,v_j}} \mathscr{S}\left( \nu_e \frac{z_{j,J}}{z_{i,I}} \right) [\dd{\underline{z}}] \, ,
    \end{align}
    where we define the index functors $\scY$ and $\scA$ as follows,
    \begin{align}
        \scY_{i,x} = \operatorname{ch} \wedge_x \mathbf{Y}_i \, , \qquad 
        \scA_{i,x} = \operatorname{ch} \wedge_x \left( \bigoplus_{j \in \Gamma_0} \mathbf{c}_{ij} \otimes \mathbf{Y}_j \right) = \scY_{i,x} \scY_{i,x q_{12}^{-1}} \prod_{e:i \to j} \scY_{j,\nu_e x q_{12}^{-1}}^{-1} \prod_{e:j \to i} \scY_{j,\nu_e^{-1}x}^{-1} \, ,
    \end{align}
    together with the same notation as before~\eqref{eq:int_notation}.
    In this context, we identify $v! = |\operatorname{Weyl}G_v|$.
\end{proposition}
\begin{proof}
    We have the following complex for the quiver variety,
    \begin{center}
    \begin{tikzcd}
        \displaystyle \bigoplus_{i \in \Gamma_0} \operatorname{Hom}(\mathbf{V}_i,\mathbf{V}_i) \arrow[r] & \substack{
        \displaystyle \bigoplus_{i \in \Gamma_0} \operatorname{Hom}(\mathbf{W}_i,\mathbf{V}_i) \otimes \det \mathbf{Q}^\vee \oplus \operatorname{Hom}(\mathbf{V}_i,\mathbf{W}_i) \\
        \bigoplus \\[.5em]
        \displaystyle \bigoplus_{e: i \to j} \operatorname{Hom}(\mathbf{V}_i,\mathbf{V}_j) \otimes \mathbf{M}_e \otimes \det \mathbf{Q}^\vee \oplus \operatorname{Hom}(\mathbf{V}_j,\mathbf{V}_i) \otimes \mathbf{M}_e^\vee
        } \arrow[r] & \displaystyle \bigoplus_{i \in \Gamma_0} \operatorname{Hom}(\mathbf{V}_i,\mathbf{V}_i) \otimes \det \mathbf{Q}^\vee
    \end{tikzcd}
    \end{center}
    from which we obtain the virtual tangent bundle,
    \begin{align}
        T\mathfrak{M}_{w,v}^\Gamma & = \bigoplus_{i \in \Gamma_0} \left( \det \mathbf{Q}^\vee \mathbf{V}_i \mathbf{W}_i^\vee \oplus \mathbf{W}_i \mathbf{V}_i^\vee \ominus \mathbf{V}_i \mathbf{V}_i^\vee \ominus \det \mathbf{Q}^\vee \mathbf{V}_i \mathbf{V}_i^\vee  \right) \oplus \bigoplus_{e: i \to j} \left( \mathbf{M}_e \det \mathbf{Q}^\vee \mathbf{V}_j \mathbf{V}_i^\vee \oplus \mathbf{M}_e^\vee \mathbf{V}_i \mathbf{V}_j^\vee \right) 
        \, .
    \end{align}
    together with the corresponding equivariant Chern characters,
    \begin{align}
        \operatorname{ch} \mathbf{W}_i = \sum_{\alpha = 1}^{w_i} x_{i,\alpha} \, , \qquad 
        \operatorname{ch} \mathbf{V}_i = \sum_{I = 1}^{v_i} z_{i,I} \, , \qquad 
        \operatorname{ch} \det \mathbf{Q} = q_{12} \, .
    \end{align}
    Hence, the $\mathbf{Y}$-independent part (the equivariant $\chi_{q_2^{-1}}$-genus of the quiver variety) is given by
    \begin{align}
        & \int_{\mathfrak{M}_{w,v}^\Gamma} \operatorname{ch} \wedge_{q_2^{-1}} T^\vee\mathfrak{M}_{w,v}^\Gamma \cdot \operatorname{td}(T\mathfrak{M}_{w,v}^\Gamma) \nonumber \\
        & = \frac{1}{v!} \oint_{\mathbb{T}_\mathbf{V}} \prod_{i \in \Gamma_0} \prod_{\substack{\alpha = 1,\ldots,w_i \\ I = 1,\ldots,v_i}} 
        \frac{(1 - q_1 x_{i,\alpha}/z_{i,I})(1 - q_{2}^{-1} z_{i,I}/x_{i,\alpha})}{(1 - q_{12} x_{i,\alpha}/z_{i,I})(1 - z_{i,I}/x_{i,\alpha})} 
        \prod_{i \in \Gamma_0} \frac{\prod_{I \neq J}^{v_i}(1 - z_{i,I}/z_{i,J}) \prod_{I,J}^{v_i} (1 - q_{12} z_{i,I}/z_{i,J})}{\prod_{I,J}^{v_i} (1 - q_{2}^{-1} z_{i,I}/z_{i,J}) (1 - q_{1} z_{i,I}/z_{i,J})}
        \nonumber \\
        & \hspace{3em} \times \prod_{e: i \to j} \prod_{\substack{I = 1,\ldots,v_i \\ J = 1,\ldots,v_j}} \frac{(1-\nu_e^{-1} q_{1} z_{i,I}/z_{j,J})(1 - \nu_e q_2^{-1} z_{j,J}/z_{i,I})}{(1-\nu_e^{-1} q_{12} z_{i,I}/z_{j,J})(1 - \nu_e z_{j,J}/z_{i,I})} [\dd{\underline{z}}] 
        \nonumber \\
        & = \frac{\mathsf{c}_{v}}{v!} q_2^{-\frac{1}{2} \operatorname{dim} \mathfrak{M}_{w,v}^\Gamma} \oint_{\mathbb{T}_\mathbf{V}} \prod_{i \in \Gamma_0} \left[ \prod_{\substack{\alpha = 1,\ldots,w_i \\ I = 1,\ldots,v_i}} \mathscr{S}\left( \frac{z_{i,I}}{x_{i,\alpha}} \right) \prod_{1 \le I \neq J \le v_i} \mathscr{S}\left( \frac{z_{i,J}}{z_{i,I}} \right)^{-1} \right] \prod_{e: i \to j} \prod_{\substack{I = 1,\ldots,v_i \\ J = 1,\ldots,v_j}} \mathscr{S}\left( \nu_e \frac{z_{j,J}}{z_{i,I}} \right) [\dd{\underline{z}}] \, .
    \end{align}
    Noticing $\operatorname{ch} \wedge \left( \mathbf{Y} \otimes \mathbf{Y}_{w,v}^\Gamma \right) = \prod_{i \in \Gamma_0} \left( \prod_{\alpha = 1,\ldots,w_i} \mathscr{Y}_{i,x_{i,\alpha}} \prod_{\alpha = 1,\ldots,w_i}\mathscr{A}_{i,z_{i,I}}^{-1} \right)$, we insert the $\mathscr{Y}$- and $\mathscr{A}$-functors to obtain the formula.
\end{proof}

The contour integral formula shown by Proposition~\ref{prop:qq-ch_int} coincides with that obtained in the operator formalism (Theorem~\ref{thm:qq-ch_int1}).
From the geometric point of view, each pole of the contour integral corresponds to the equivariant fixed point.
The integrand may have a higher pole, which gives rise to derivative terms.
Properly evaluating the integral~\eqref{eq:qq-ch_int}, we obtain the following.
\begin{proposition}\label{prop:qq-ch_alg_geom}
    The operatorial $qq$-character is mapped to the geometric $qq$-character under the map,
    \begin{align}
        : \prod_{i \in \Gamma_0} \left( \prod_{\alpha = 1,\ldots,w_i} \mathsf{Y}_{i,x_{i,\alpha}} \prod_{\alpha = 1,\ldots,w_i} \mathsf{A}_{i,z_{i,I}}^{-1} \right) : \ \longmapsto \
        \prod_{i \in \Gamma_0} \left( \prod_{\alpha = 1,\ldots,w_i} \mathscr{Y}_{i,x_{i,\alpha}} \prod_{\alpha = 1,\ldots,w_i}\mathscr{A}_{i,z_{i,I}}^{-1} \right) \, .
    \end{align}
\end{proposition}
In the context of the BPS/CFT correspondence, this map is naturally realized by taking an expectation value with respect to the so-called $Z$-state or the trace in the operator formalism~\cite{Kimura:2015rgi,Kimura:2016dys}.
In this case, the corresponding formal bundle is identified with the universal bundle over the instanton moduli space.

\bibliographystyle{ytamsalpha}
\bibliography{ref}

\end{document}